\documentclass{compositio}

\usepackage[wheretonumber = subsubsection]{intropapers}
\usepackage{cite}
\usepackage{hyperref}
\usepackage{url}
\usepackage{tikz-cd}
\usepackage{xcolor}

\date{}

\begin{document}
\title{Restriction of $p$-adic representations of~$\GL_2(\bQ_p)$ to parahoric subgroups}
\author{Andrea Dotto}
\email{andreadotto@uchicago.edu}
\address{University of Chicago, 5734 South University Ave, Chicago, IL 60637, USA}

\begin{abstract}
Without using the $p$-adic Langlands correspondence, we prove that for many finite length smooth representations of~$\GL_2(\mathbf{Q}_p)$ on $p$-torsion modules the $\GL_2(\mathbf{Q}_p)$-linear morphisms coincide with the morphisms that are linear for the normalizer of a parahoric subgroup. We identify this subgroup to be the Iwahori subgroup in the supersingular case, and~$\GL_2(\mathbf{Z}_p)$ in the principal series case. As an application, we relate the action of parahoric subgroups to the action of the inertia group of~$\mathrm{Gal}(\overline{\mathbf{Q}}_p/\mathbf{Q}_p)$, and we prove that if an irreducible Banach space representation~$\Pi$ of~$\GL_2(\mathbf{Q}_p)$ has infinite $\GL_2(\mathbf{Z}_p)$-length then a twist of~$\Pi$ has locally algebraic vectors. This answers a question of Dospinescu. We make the simplifying assumption that~$p > 3$ and that all our representations are generic.
\end{abstract}

\maketitle

\tableofcontents

\section{Introduction.}
Let $p \geq 5$ be a prime number. 
Fix a finite extension $E/\bQ_p$, with ring of integers~$\mO$ and residue field~$k$.
The purpose of this paper is to study the behaviour of representations of~$G = \GL_2(\bQ_p)$ on \textcolor{black}{$p^n$-torsion} $\mO$-modules and on $E$-Banach spaces upon restriction to a parahoric subgroup.
Our motivation for doing so arises from two applications of our results.
The first one is the following theorem, which answers a question of Dospinescu by providing a classification of those irreducible unitary Banach space representations of~$\GL_2(\bQ_p)$ which have infinite length when restricted to the maximal compact subgroup~$K = \GL_2(\bZ_p)$.
Let~$Z \cong \bQ_p^\times$ be the centre of~$G$.

\begin{thm}\label{restriction}
Let~$\Pi$ be an absolutely irreducible, admissible, very generic, unitary $E$-Banach space representation of~$\GL_2(\bQ_p)$ with central character $\zeta: Z \to \mO^\times$.
(See Section~\ref{defngeneric} for the precise genericity conditions we need.)
Then exactly one of the following holds:
\begin{enumerate}
\item $\Pi|_{KZ}$ is irreducible.
\item $\Pi$ has irreducible supersingular reduction, $\Pi \cong \Pi \otimes (\nr_{-1} \circ \det)$, and $\Pi \cong \Ind_{G^+}^G(\Pi_0)$ for some irreducible~$G^+$-representation $\Pi_0$ such that~$\Pi_0|_{KZ}$ is irreducible. (Here $G^+ = \ker(\nr_{-1} \circ \det)$ is the subgroup of~$\GL_2(\bQ_p)$ of elements whose determinant has even valuation.)
\item $\Pi$ has a closed $KZ$-stable $E$-subspace of finite dimension over~$E$.
\end{enumerate}
\end{thm}
\begin{corollary}\label{locallyalgebraic}
Let~$\Pi$ be as in the statement of Theorem~\ref{restriction}. 
If~$\Pi|_{K}$ has infinite length then a twist of~$\Pi$ has locally algebraic vectors.
\end{corollary}

The statements above refer to topological irreducibility, and all cases of Theorem~\ref{restriction} can occur (see Remark~\ref{casesforPi}).
We remark that our proofs are independent of the $p$-adic Langlands correspondence for~$\GL_2(\bQ_p)$, and so they have a chance of being applicable to other groups.
\textcolor{black}{For example, see the work of Dospinescu, Pa{\v s}k{\= u}nas and Schraen~\cite{PaskunasLudwig, DPScharacters} as well as Hu and Wang~\cite{HuWangD*, HuwangD*II} for analogues of Theorem~\ref{restriction} for the unit group~$D^\times$ of the nonsplit quaternion algebra~$D$ over~$\bQ_p$.}

On the other hand, we emphasize that to get Theorem~\ref{restriction} we make use of one of the main results of the paper \cite{Paskunasimage}, namely that absolutely irreducible admissible unitary Banach space representations of~$\GL_2(\bQ_p)$ are residually of finite length, and that a list of the possible reductions can be given in generic cases.
This uses Colmez's functor in a significant way.
We could avoid this appeal to~\cite{Paskunasimage} by making explicit assumptions in Theorem~\ref{restriction} about the reduction of a unit ball in~$\Pi$: this would yield a statement that is a priori weaker than the one we prove, but equivalent to it in light of the results of~\cite{Paskunasimage}.
Similarly, using the $p$-adic local Langlands correspondence one can easily prove a converse to Corollary~\ref{locallyalgebraic}.

The second of our applications is concerned with a $p$-adic analogue of the so-called ``inertial Langlands correspondence", which is a refinement of the classical local Langlands correspondence obtained by considering compact subgroups of the groups appearing at the two sides of the correspondence.  
In more detail, in the setting of smooth representations of $\GL_2(\bQ_p)$ with complex coefficients we have that two irreducible smooth representations are in the same Bernstein component if and only if their Langlands parameters have isomorphic restriction to the inertia group~$I_{\bQ_p}$.
On the other hand, as a consequence of the theory of types, one knows that two irreducible cuspidal $G$-representations~$\pi_1, \pi_2$ are in the same Bernstein component if and only if $\pi_1 |_K \cong \pi_2 |_K$. 
Since the cuspidal representations $\pi_1$ and~$\pi_2$ are in the same Bernstein component if and only if they are unramified twists of each other, it follows that $\pi_1 |_{KZ} \cong \pi_2|_{KZ}$ if and only if $\pi_1 \cong \pi_2 \otimes (\nr_{\pm 1} \circ \det)$.
We prove the following analogous result in the setting of $E$-Banach space representations.

\begin{thm}\label{isomorphismclasses}
Let~$\Pi_1, \Pi_2$ be absolutely irreducible, admissible, very generic, non-ordinary, unitary $E$-Banach representations of~$\GL_2(\bQ_p)$ with central character~$\zeta$.
Write~$\Iw$ for the upper-triangular Iwahori subgroup of~$G$ and~$N$ for its normalizer in~$G$.
\benum
\item If~$\Pi_1, \Pi_2$ have irreducible reduction, then $\Pi_1 |_{\Iw Z} \cong \Pi_2 |_{\Iw Z}$ if and only if $\Pi_1 \cong \Pi_2 \otimes (\nr_{\pm 1} \circ \det)$, and $\Pi_1|_N \cong \Pi_2 |_N$ if and only if~$\Pi_1 \cong \Pi_2$.
\item If~$\Pi_1, \Pi_2$ have reducible reduction, then $\Pi_1 |_{KZ} \cong \Pi_2 |_{KZ}$ if and only if~$\Pi_1 \cong \Pi_2 \otimes (\nr_{\pm 1} \circ \det)$.
\item Otherwise, there are no $\Iw Z$-linear isomorphisms $\Pi_1 \isom \Pi_2$.
\eenum
\end{thm}

\begin{corollary}\label{inertialcorrespondence}
Let~$\rho_1, \rho_2: \Gal_{\bQ_p} \to \GL_2(E)$ be absolutely irreducible continuous Galois representations with $\det \rho_1 = \det \rho_2$, and write~$\rhobar_i$ for the semisimplified mod~$p$ reduction of~$\rho_i$.
Let~$\Pi_1, \Pi_2$ be the $E$-Banach space representations of~$\GL_2(\bQ_p)$ corresponding to~$\rho_i$ under Colmez's functor, \textcolor{black}{and assume they are very generic}.
\benum
\item If~$\rhobar_1, \rhobar_2$ are irreducible, then $\rho_1|_{I_{\bQ_p}} \cong \rho_2|_{I_{\bQ_p}}$ if and only if $\Pi_1|_{\Iw Z} \cong \Pi_2 |_{\Iw Z}$.
\item If~$\rhobar_1, \rhobar_2$ are reducible, then $\rho_1 |_{I_{\bQ_p}} \cong \rho_2 |_{I_{\bQ_p}}$ if and only if~$\Pi_1 |_{KZ} \cong \Pi_2|_{KZ}$.
\item Otherwise, $\rho_1|_{I_{\bQ_p}}$ is not isomorphic to $\rho_2 |_{I_{\bQ_p}}$.
\eenum
\end{corollary}

The corollary follows from the theorem since by Clifford theory and the condition on the determinant we have $\rho_1 |_{I_{\bQ_p}} \cong \rho_2 |_{I_{\bQ_p}}$ if and only if $\rho_1 \cong \rho_2 \otimes \nr_{\pm 1}$ (see Proposition~\ref{CliffordtheoryGalois}).
We see that Corollary~\ref{inertialcorrespondence} relates the inertia group~$I_{\bQ_p}$ to a parahoric subgroup of~$\GL_2(\bQ_p)$, as in the case of smooth representations.

Although Theorems~\ref{restriction} and~\ref{isomorphismclasses} are about $E$-Banach space representations, the main input in their proof is Theorem~\ref{morphisms}, a stronger version of Theorem~\ref{isomorphismclasses} valid for $p^n$-torsion representations, which may be of independent interest.
It is proved as a combination of the results in Section~\ref{modprepresentations}, building on work of Morra and Pa{\v s}k{\= u}nas.
Theorems~\ref{restriction} and~\ref{isomorphismclasses} then follow directly in the case of supersingular reduction.
In the case of reducible reduction we need a more involved argument, making use of a Banach space version of Ribet's lemma on lattices in irreducible two-dimensional representations with reducible reduction, which we develop in Appendix~\ref{Ribetappendix}.

\subsection{Notation and conventions.}
We fix throughout the article a finite extension $E/\bQ_p$ with ring of integers~$\mO$ and residue field~$k$, as well as an algebraic closure~$\lbar k$ of~$k$, to act as coefficients.
Fix an algebraic closure $\cbQ_p / \bQ_p$ and write $G_{\bQ_p} = \Gal_{\bQ_p} = \Gal(\cbQ_p / \bQ_p)$, $I_{\bQ_p} \subset G_{\bQ_p}$ for the inertia group, and $\omega: \Gal_{\bQ_p} \to \bZ_p^\times$ for the cyclotomic character.
We normalize local class field theory so that~$\omega$ corresponds to the character $\bQ_p^\times \to \bZ_p^\times, x \mapsto x|x|$.
If~$\lambda \in \mO^\times$ we write $\nr_\lambda: \bQ_p^\times \to \mO^\times$ for the unramified character sending~$p$ to~$\lambda$, and similarly if~$\lambda \in k^\times$.
We will usually work with a fixed continuous character $\zeta: \bQ_p^\times \to \mO^\times$ to act as the central character of our $\GL_2(\bQ_p)$-representations.

Write $G = \GL_2(\bQ_p)$, $K = \GL_2(\bZ_p)$, $B$ for the upper-triangular Borel subgroup, $Z \cong \bQ_p^\times$ for the centre, and~$T$ for the diagonal torus. 
The index-two subgroup $G^+ \subset G$ is defined to be the kernel of~$\nr_{-1} \circ \det$.
We will write $U(p^n\bZ_p) = \fourmatrix 1 {p^n\bZ_p} 0 1$, and similarly for the lower-triangular unipotent subgroup~$\overline{U}$. 
Define~$K_0(p^n) = \fourmatrix{\bZ_p^\times}{\bZ_p}{p^n\bZ_p}{\bZ_p^\times}$, so that $K_0(p) = \Iw$ is the upper-triangular Iwahori subgroup. 
We will sometimes write $K_0(p^\infty)$ for~$B(\bZ_p)$. 
The principal congruence subgroups of~$K$ will be denoted $K_n = \fourmatrix{1+p^n\bZ_p}{p^n\bZ_p}{p^n \bZ_p}{1+p^n\bZ_p}$.
If~$n \geq 0$, the Iwahori decomposition for~$K_0(p^{n+1})$ states that the multiplication map
\[
\lbar U(p^{n+1}\bZ_p) \times T(\bZ_p) \times U(\bZ_p) \to K_0(p^{n+1})
\]
is a homeomorphism.
The pro-$p$ Sylow subgroup of~$\Iw$ is equal to~$\Iw_1 = \fourmatrix{1+p\bZ_p}{\bZ_p}{p\bZ_p}{1+p\bZ_p}$, and $\Iw = H \ltimes \Iw_1$ for the subgroup $H = \fourmatrix{[a]}{0}{0}{[d]}$, where~$[-]$ denotes the Teichm\"uller lift to~$\bZ_p$ of an element of~$\bF_p$. 
The character $ad^{-1}$ of~$T(\bF_p)$ (or its inflation to other groups such as $H$ and~$\Iw$) is denoted~$\alpha$.

Except in the section on Banach space representations, \textcolor{black}{in which~$\Pi$ is denoted by~$\nu$,} we will write
\[
\Pi = \fourmatrix 0 1 p 0, \, t = \fourmatrix p 0 0 1, \, s = \fourmatrix 0 1 1 0, \, \textcolor{black}{\tau = \fourmatrix 1 0 0 p}.
\]
The group~$N = \Pi^{\bZ} \ltimes \Iw$ is the normalizer of~$\Iw$ in~$G$. If~$\pi$ is a smooth representation of~$\Iw$, we will sometimes denote the twist $\ad(\Pi)^*(\pi)$ by~$\pi^+$: this is the representation of~$\Iw$ with the same representation space as~$\pi$ but the action 
\begin{equation}\label{definitiontwist}
\pi^+\fourmatrix a b {pc} d = \pi \left (\Pi \fourmatrix a b {pc} d \Pi^{-1} \right ).
\end{equation}
Given an irreducible $k[\GL_2(\bZ_p)]$-module~$\sigma$, and unless otherwise specified, we will write $\chi$ for the $\Iw$-character~$\sigma^{\Iw_1}$ and~$\kappa$ for the conjugate of $\sigma^{\Iw_1}$.
Hence~$\kappa = \chi \circ \ad(\Pi)$, and it also equals~$\chi \circ \ad(s)$ as characters of~$T(\bZ_p)$.
For this reason, we will sometimes also write~$\chi^s$ or~$\chi^+$ to denote~$\kappa$.
If~$G$ is a locally profinite group with a closed subgroup~$K$ and an open subgroup~$H$, we write $\Ind_K^G$ for induction and $\cInd_H^G$ for compact induction, so that $\Ind_K^G$ is right adjoint to~$\Res^G_K$ and $\cInd_H^G$ is left adjoint to $\Res^G_H$. If~$\pi$ is a representation of~$G$ and $x \in \pi$, we will write $\langle G \cdot x \rangle$ for the smallest $G$-subrepresentation of~$\pi$ that contains~$x$.
Unless stated otherwise, the notation~$\Ext^i_G$ will denote $\Ext$-groups computed in the category of smooth $k[G]$-representations.

\section{Mod~$p$ representations.}\label{modprepresentations}

\subsection{Preliminaries.}
We begin with some generalities about uniserial representations of profinite groups.
Then we list the irreducible representations of~$k[\GL_2(\bF_p)]$ and~$k[\GL_2(\bQ_p)]$ we will work with, and we state our genericity conditions.
We will fix a continuous character $\zeta: \bQ_p^\times \to \mO^\times$, and we will usually assume that representations of~$G$ or~$G^+$ have central character~$\zeta$.
We will assume in this section that $\zeta(p) = 1$: this can always be arranged by passing to a quadratic extension of~$E$ and twisting by an unramified character.

\subsubsection{Uniserial representations.}

Let~$H$ be a profinite group with an open pro-$p$ subgroup, and let~$\pi$ be a smooth representation of~$k[H]$. 
Even if~$\pi$ has infinite $H$-length, we can define its socle filtration by letting $\soc^0(\pi)$ be the maximal semisimple $H$-subrepresentation of~$\pi$, and then defining~$\soc^i(\pi)$ via the short exact sequence
\[
0 \to \soc^{i-1}(\pi) \to \soc^i(\pi) \to \soc^0 \left ( \pi / \soc^{i-1}(\pi) \right ) \to 0.
\] 
Since every vector $x \in \pi$ generates a finite-dimensional $H$-representation we have $\pi = \bigcup_i \soc^i(\pi)$. 
We will usually write~$\soc$ for~$\soc^0$.

\begin{lemma}\label{soclesum}
Let $\pi = \pi_1 \oplus \pi_2$ be a smooth representation of~$k[H]$. Then $\soc^i(\pi) = \soc^i(\pi_1) \oplus \soc^i(\pi_2)$.
\end{lemma}
\begin{proof}
By the exact sequence
\[
0 \to \soc(\pi) \to \soc^{i+1}(\pi) \to \soc^i(\pi/\soc(\pi)) \to 0
\]
and induction on~$i$ it suffices to prove that $\soc(\pi) = \soc(\pi_1) \oplus \soc(\pi_2)$. 
For this, observe that if~$X$ is a semisimple smooth representation of~$k[H]$, and $\iota: X \to \pi$ is an $H$-linear map, then $\im(\pr_i \circ \iota) \subseteq \soc(\pi_i)$ for~$i = 1, 2$, where~$\pr_i$ the projection $\pi \to \pi_i$, and so $\soc(\pi) \subseteq \soc(\pi_1) \oplus \soc(\pi_2)$.
The other direction is true by definition.
\end{proof}

\begin{defn}
We say that a $k[H]$-module $\pi$ is \emph{uniserial} if the set of $k[H]$-submodules of~$\pi$ is totally ordered by inclusion.
\end{defn}

\begin{lemma}\label{soclefiltration}
The representation~$\pi$ is uniserial if and only if $\soc^{i}(\pi) / \soc^{i-1}(\pi)$ is irreducible for every $i \in \bZ_{\geq 0}$.
\end{lemma}
\begin{proof}
Every semisimple reducible $H$-representation contains two submodules incomparable by inclusion. 
Conversely, assume that all the graded factors of the socle filtration of~$\pi$ are irreducible, and let~$X \subset \pi$ be a proper submodule.
Let~$i$ be such that $\soc^i(\pi) \subseteq X$ but $\soc^{i+1}(\pi) \not \subseteq X$.
Consider the submodule $X/\soc^i\pi \subset \pi / \soc^i \pi$.
Since~$\pi$ is smooth, if $X/\soc^i \pi \ne 0$ then it contains an irreducible $k[H]$-subspace.
But since $\soc(\pi/\soc^i\pi)$ is irreducible, this implies that $\soc^{i+1}(\pi) \subseteq X$, which we are assuming not to be the case.
So $X  = \soc^i(\pi)$, hence the only submodules of~$\pi$ are the~$\soc^i(\pi)$, and the lattice of submodules of~$\pi$ is totally ordered by inclusion.
\end{proof}

\begin{lemma}\label{uniserial}
If~$\pi_1, \pi_2$ are uniserial smooth representations of~$k[H]$, then every proper $H$-submodule of~$\pi_i$ is finite-dimensional. If~$\pi_1$ is infinite-dimensional, then every nonzero $k[H]$-linear morphism $\lambda: \pi_1 \to \pi_2$ is surjective.
\end{lemma}
\begin{proof}
Let $\pi_i' \subset \pi_i$ be a proper $H$-submodule. If $x \in \pi_i \setminus \pi_i'$, then the definition of a uniserial representation implies that $\pi_i' \subseteq \langle H \cdot x \rangle$, which is finite-dimensional by $H$-smoothness. The second claim follows from the first, counting dimensions on the exact sequence $0 \to \ker \lambda \to \pi_1 \to \image \lambda \to 0$.
\end{proof}

\subsubsection{Serre weights.} 
Every irreducible $k$-representation of~$\GL_2(\bF_p)$ (also known as \emph{Serre weight}) has the form $\sigma_{r, s} = \Sym^rk^2 \otimes \det^s$ for uniquely determined integers $0 \leq r \leq p-1$ and $0 \leq s \leq p-2$. We will usually realize~$\Sym^r$ on the space of degree~$r$ homogeneous polynomials in two variables~$x$ and~$y$, and coefficients in~$k$, via the action 
\[
\fourmatrix a b c d x^{r-i} y^{i} = (ax+cy)^{r-i}(bx+dy)^{i}.
\]
The group~$U(\bF_p)$ fixes a unique line in~$\sigma_{r, s}$, spanned by~$x^r$, whose eigencharacter for the Borel subgroup is $\chi: \fourmatrix a b 0 d \mapsto a^r(ad)^s$. 
This character determines~$\sigma$ whenever $r \not = 0, p-1$, or equivalently whenever $\chi$ is not equal to its conjugate~$\chi^s = \chi \circ \ad(s)$.
We will often write~$\kappa$ for~$\chi^s$. 
We will call such a character a \emph{generic} character, and we will say that a Serre weight~$\sigma$ is generic if $\sigma^{\Iw_1}$ is a generic character.
(So the nongeneric Serre weights are the characters and the twists of the Steinberg representation.)

\subsubsection{$\GL_2(\bQ_p)$-representations.} 
Recall that we have fixed a continuous character $\zeta: \bQ_p^\times \to \mO^\times$.
If~$\sigma$ is a Serre weight we will often implicitly regard it as a $K$-representation, and we will extend the $K$-action to an action of~$KZ$ by letting~$p \in Z$ act by~$\zeta(p)$.

The smooth irreducible $\cbF_p$-representations of~$\GL_2(\bQ_p)$ were classified by Barthel--Livn\'e and Breuil~\cite{BLreps, BreuilrepsI}. 
We will work over the finite field~$k$, and we will follow the normalizations of~\cite{BreuilrepsI} unless stated otherwise, hence we have the following classification of irreducible representations with central character trivial at~$p$. 
\begin{enumerate}
\item Irreducible principal series representations. These are of the form
\[
\pi(r, \lambda, \chi) = \left ( \cInd_{KZ}^G(\Sym^r k^2)/(T-\lambda) \right) \otimes (\chi \circ \det)
\]
for a smooth character $\chi: \bQ_p^\times \to k^\times$ with $\chi(p) = \pm 1$, and $(r, \lambda) \not \in \{(0, \pm 1), (p-1, \pm 1) \}$. The only intertwinings are 
\[
\pi(r, \lambda, \chi) \cong \pi(r, -\lambda, \nr_{-1}\chi) \text{ and } \pi(0, \lambda, \chi) \cong \pi(p-1, \lambda, \chi).
\]
These representations can be written as parabolic inductions by considering the characters $\psi = \nr_{\lambda^{-1}} \otimes \nr_{\lambda}\omega^r : T(\bQ_p) \to k^\times$. 
If the pair $(r, \lambda) \not \in \{ (0, \pm 1), (p-1, \pm 1) \}$, then the representation $\Ind_B^G(\psi)$ is irreducible and isomorphic to
\[
\pi(r, \lambda, 1) = \cInd_{KZ}^G(\Sym^{r}k^2)/(T-\lambda).
\]
See for example~\cite[Remarque~4.2.5]{BreuilrepsI}. 
Here and in what follows, $T$ is the usual Hecke operator.
\item Supersingular representations. These have the form
\[
\pi(r, 0, \chi) = \left ( \cInd_{KZ}^G(\Sym^r k^2)/T \right) \otimes (\chi \circ \mathrm{det}) = \left ( \cInd_{KZ}^G(\Sym^rk^2 \otimes \mathrm{det}^s) /T \right) \otimes (\nr_{\pm 1} \circ \mathrm{det})
\]
for $0 \leq r \leq p-1$, $1 \leq s \leq p-1$, where $\chi = \nr_{\pm 1}\omega^s$. The only intertwining isomorphisms are
\[
\pi(r, 0, \chi) \cong \pi(r, 0, \chi\nr_{-1}) \cong \pi(p-1-r, 0, \chi\omega^r) \cong \pi(p-1-r, 0, \chi\omega^r\nr_{-1}).
\]
\item Characters and Steinberg twists. These arise as the Jordan--H\"older factors of reducible principal series representations, and will not be studied in this article.
\end{enumerate}

\subsubsection{Genericity conditions}\label{defngeneric}
The following are the genericity conditions we use in this paper:
\benum
\item a Serre weight is generic if it is not a character or a twist of the Steinberg representation of~$\GL_2(\bF_p)$.
\item an absolutely irreducible $k[\GL_2(\bQ_p)]$-representation is generic if it is $\lbar k$-isomorphic to some $\pi(r, \lambda, \chi)$ for $r \not \in \{0, p-1\}$.
It is very generic if it is generic and not $\lbar k$-isomorphic to~$\pi(r, \lambda, \chi)$ for~$r = p-2$ and~$\lambda \ne 0$.
(See Remark~\ref{whyp-3} for why we will need to exclude twists of~$\Sym^{p-2}$ at certain stages of our arguments.)
\item A finite length~$\mO[\GL_2(\bQ_p)]$-representation is generic, resp. very generic if all its Jordan--H\"older factors are generic, resp. very generic.
\item A unitary admissible $E$-Banach space representation~$\Pi$ of~$\GL_2(\bQ_p)$ is generic, resp. very generic if $\Theta \otimes_\mO k$ has finite length and is generic, resp. very generic for all open bounded $\GL_2(\bQ_p)$-stable $\mO$-lattices~$\Theta \subset \Pi$.
\eenum

\subsection{\textcolor{black}{Restriction of principal series to parahoric subgroups.}}\label{principalseriesrestriction}
The Iwasawa decomposition $G = BK$ implies that there is a $K$-linear isomorphism 
\begin{equation}\label{Iwasawadecomposition}
 \Ind_B^G(\nr_{\lambda^{-1}} \otimes \nr_{\lambda} \omega^{r+1}) \to \Ind_{B(\bZ_p)}^K(1\otimes \omega^{r+1}), \, f \mapsto f|_K.
\end{equation}
In this section we study certain representations related to the one appearing in the right-hand side of this isomorphism, and their restriction to the Iwahori subgroup. 
Let $\chi: H \to k^\times$ be a character and~$n \geq 0$ an integer. We begin with some properties of the finite induction $\pi_{n+1}(\chi) = \Ind_{K_0(p^{n+1})}^{K_0(p)}(\chi)$, which is a uniserial $K_0(p)$-representation of dimension~$p^n$ (see for example~\cite[Proposition~1.6]{Morraexplicit}, or Proposition~\ref{Morrathesis1} for a self-contained proof). We always realize an induction to~$K_0(p)$ as a space of smooth functions on~$K_0(p)$, and we write~$\varphi_{n+1}$ for the unique function in~$\pi_{n+1}(\chi)$ that is supported in~$K_0(p^{n+1})$ and takes value~$1$ at the identity. 

\begin{lemma}\label{cokernelvarphi}
The $K_0(p)$-cosocle of~$\pi_{n+1}(\chi)$ is one-dimensional, and generated by the image of~$\varphi_{n+1}$.
\end{lemma}
\begin{proof}
By definition, we have
\[
\pi_{n+1}(\chi) = \{f: K_0(p) \to k \text{ such that } f(kx) = \chi(k)f(x) \text{ for all } k \in K_0(p^{n+1}), x \in K_0(p)\}
\]
because these functions are automatically smooth.
Every coset $K_0(p^{n+1})g$ supports a one-dimensional space of elements of~$\pi_{n+1}(\chi)$, generated by the unique function $\varphi_g \in \pi_{n+1}(\chi)$ supported in $K_0(p^{n+1})g$ and sending~$g$ to~$1$.
With this notation, we have~$\varphi_{n+1} = \varphi_{\id}$ and so $\varphi_g = g^{-1}\varphi_{n+1}$ for all~$g \in K_0(p)$.
This implies that~$\pi_{n+1}(\chi)$ is a cyclic $K_0(p)$-module generated by~$\varphi_{n+1}$. 
There remains to prove that
\[
\dim_k \cosoc_{K_0(p)} \pi_{n+1}(\chi) = 1.
\]
This is an immediate consequence of the fact that~$\pi_{n+1}(\chi)$ is uniserial and the irreducible $k$-representations of~$K_0(p)$ are one-dimensional.
Otherwise, one can give a direct proof using Frobenius reciprocity and the fact that since $K_0(p^{n+1}) \subset K_0(p)$ is an open subgroup of finite index we have
\[
\pi_{n+1}(\chi) \cong \cInd_{K_0(p^{n+1})}^{K_0(p)}(\chi).
\]
\end{proof}

For any~$n\geq 1$ there is an injection $\pi_n(\chi) \to \pi_{n+1}(\chi)$ that is an inclusion of spaces of smooth functions on~$K_0(p)$. There is also a surjection $\pi_{n+1}(\chi) \to \pi_{n}(\chi)$ that sends~$\varphi_{n+1}$ to~$\varphi_n$. 
Hence the (unique) submodule of~$\pi_{n+1}(\chi)$ of dimension~$p^{n-1}$ and the (unique) quotient of~$\pi_{n+1}(\chi)$ of dimension~$p^{n-1}$ are both isomorphic to~$\pi_{n}(\chi)$. 
In fact, something stronger is true, by the following lemma.

\begin{lemma}\label{psfiltration}
Let~$1 \leq i \leq n+1$. Then~$\pi_{n+1}(\chi)$ has an $\Iw$-stable filtration with $p^{n+1-i}$ graded factors of dimension $p^{i-1}$, of the form
\[
\pi_{i}(\chi) - \pi_i(\chi\alpha) - \pi_i(\chi\alpha^2) - \cdots.
\]
\end{lemma}
\begin{proof}
By~\cite[Proposition~1.6]{Morraexplicit} the socle filtration of~$X = \Ind_{K_0(p^{n+1})}^{K_0(p^n)}(\chi)$ is
\[
\chi - \chi\alpha - \chi\alpha^2 - \cdots - \chi\alpha^{p-2} - \chi
\]
with~$p$ graded factors, in the sense that
\[
\soc^i(X)/\soc^{i-1}(X) \cong \chi\alpha^i.
\]
This implies by transitivity of induction and exactness of the functor $\Ind_{K_0(p^n)}^{K_0(p)}(-)$ that there is a filtration on~$\pi_{n+1}(\chi) \cong \Ind_{K_0(p^n)}^{K_0(p)}\Ind_{K_0(p^{n+1})}^{K_0(p^n)}(\chi)$ of the form
\[
\pi_{n}(\chi) - \pi_n(\chi\alpha) - \cdots \pi_n(\chi\alpha^{p-2}) - \pi_n(\chi).
\]
The lemma then follows by induction on~$n$.
\end{proof}

We next study the $B(\bZ_p)$-eigenspaces in~$\pi_{n+1}(\chi)$.
Recall that~$H$ is the Teichm\"uller lift of~$\bF_p^\times \times \bF_p^\times$ in the diagonal torus~$T$.
We regard $\bF_p^\times$ as a subgroup of~$H$ via the diagonal embedding.

\begin{lemma}\label{eigenvectors}
Let~$\chi, \psi: H \to k^\times$ be characters.
Then
\begin{equation}\label{IwahoriBorel}
\Hom_{K_0(p^{n+1})}(\psi, \Ind_{K_0(p^{n+1})}^{K_0(p)}(\chi)) = \Hom_{B(\bZ_p)}(\psi, \Ind_{K_0(p^{n+1})}^{K_0(p)}(\chi)),
\end{equation}
and this space has dimension~$n+1$ if~$\chi = \psi$.
\end{lemma}
\begin{proof}
Compare~\cite[Lemma~1]{Casselmanrestriction}.
To see that the two sides of~\eqref{IwahoriBorel} are equal it suffices to prove that every $(B(\bZ_p), \psi)$-eigenvector in~$\pi_{n+1}(\chi)$ is automatically a $(K_0(p^{n+1}), \psi)$-eigenvector.
This is a consequence of the fact that the normal subgroup
\[
K_{n+1} = \fourmatrix{1+p^{n+1}\bZ_p}{p^{n+1}\bZ_p}{p^{n+1}\bZ_p}{1+p^{n+1}\bZ_p} 
\]
of~$K_0(p)$ acts trivially on~$\pi_{n+1}(\chi)$, and that~$K_0(p^{n+1}) = K_{n+1}B(\bZ_p)$.
Now, by the formula
\[
\fourmatrix a b {up^i} d = \fourmatrix 1 0 0 {ua^{-1}} \fourmatrix 1 0 {p^i} 1 \fourmatrix a b 0 {u^{-1}(ad-bup^i)} 
\]
valid whenever $a, d, u \in \bZ_p^\times$ and $b \in \bZ_p$, we deduce that there is a disjoint double coset decomposition
\[
K_0(p) = \coprod_{i=1}^{n+1} K_0(p^{n+1}) \fourmatrix 1 0 {p^i} 1 K_0(p^{n+1}) = \coprod_{i=1}^{n+1} K_0(p^{n+1}) \fourmatrix 1 0 {p^i} 1 B(\bZ_p).
\]
This implies that the dimension of either side of~\eqref{IwahoriBorel} is at most~$n+1$, since each of these double cosets supports at most a one-dimensional space of $K_0(p^{n+1})$-eigenvectors, respectively $B(\bZ_p)$-eigenvectors, with eigencharacter~$\psi$.
When $\psi = \chi$, both spaces have dimension at least~$n+1$, hence equal to~$n+1$.
To see this, let~$\varphi_i \in \pi_i(\chi)$ be the function supported in~$K_0(p^i)$ taking value~$1$ at the identity. 
It is a $K_0(p^i)$-eigenvector with eigencharacter~$\chi$. 
Since $\varphi_i$ generates the $\Iw$-cosocle of~$\pi_i(\chi)$, by Lemma~\ref{cokernelvarphi}, it is not contained in any proper $\Iw$-subrepresentation of~$\pi_i(\chi)$.  
It follows that $\{ \varphi_i : 1 \leq i \leq n+1 \}$ is a linearly independent set of $K_0(p^{n+1})$-eigenvectors in~$\pi_{n+1}(\chi)$.
\end{proof}

\begin{lemma}\label{eigenvectorssemisimple}
Let~$\psi: H \to k^\times$ be a character and~$n \geq 1$ an integer.
Then $\Ind_{K_0(p^{n+1})}^{K_0(p^n)}(\chi)$ is a semisimple $K_0(p^{n+1})$-representation, isomorphic to
\[
\bigoplus_{i=0}^{p-1} \chi \alpha^i.
\]
\end{lemma}
\begin{proof}
Let~$K_1(p^n)$ be the pro-$p$ Sylow subgroup of~$K_0(p^n)$.
The Iwahori decomposition of~$K_0(p^n)$ implies that $K_0(p^n) = K_0(p^{n+1})K_1(p^n)$.
It follows from~\cite[Section~I.5.5]{Vignerasrepsbook} that
\[
\Ind_{K_0(p^{n+1})}^{K_0(p^n)}(\chi)|_{K_1(p^n)} \cong \Ind_{K_1(p^{n+1})}^{K_1(p^n)}(1).
\]
However, $K_1(p^{n+1})$ is a normal subgroup of~$K_1(p^n)$: this is because the bottom left entry of
\[
\fourmatrix a b {p^n c} d \fourmatrix x y {p^{n+1}z} w \fourmatrix d {-b} {-p^n c} a
\]
is
\[
p^ncd(x-w) + p^{n+1}zd^2 -p^{2n}yc^2
\]
which is divisible by~$p^{n+1}$ when~$x \equiv w$ mod~$p\bZ_p$.
It follows that $\Ind_{K_0(p^{n+1})}^{K_0(p^n)}(\chi)$ is trivial as a representation of~$K_1(p^{n+1})$.
Hence it is semisimple as a representation of~$K_0(p^{n+1})$.
The rest of the lemma follows from Lemma~\ref{psfiltration}. 
We point out that the discussion so far also implies that $\Ind_{K_0(p^{n+1})}^{K_0(p^n)}(\chi)$ is already $K_1(p^n)$-uniserial, since $\Ind_{K_1(p^{n+1})}^{K_1(p^n)}(1)$ has the same lattice of submodules as the regular representation of a finite group of order~$p$, which is uniserial in characteristic~$p$.
\end{proof}

\begin{lemma}\label{eigenvectorsII}
Let~$\psi, \chi: H \to k^\times$ be distinct characters such that~$\psi|_{\bF_p^\times} = \chi|_{\bF_p^\times}$. Then
\begin{equation}\label{IwahoriBorelII}
\Hom_{K_0(p)}(\pi_{n+1}(\psi), \pi_{n+1}(\chi)) \cong \Hom_{K_0(p^{n+1})}(\psi, \pi_{n+1}(\chi)) = \Hom_{B(\bZ_p)}(\psi, \pi_{n+1}(\chi))
\end{equation}
and the dimension of these spaces is~$n$.
\end{lemma}
\begin{proof}
The three spaces in~\eqref{IwahoriBorelII} are isomorphic, by Lemma~\ref{eigenvectors} and Frobenius reciprocity.
We have seen in the proof of Lemma~\ref{eigenvectors} that the map
\[
\pi_{n+1}(\chi) \to k^{n+1}, f \mapsto \left (f(1), f\fourmatrix 1 0 p 1, \ldots, f\fourmatrix 1 0 {p^n} 1 \right )
\]
yields a linear injection of the $(K_0(p^{n+1}), \psi)$-eigenspace of~$\pi_{n+1}(\chi)$ into~$k^{n+1}$.
Hence the dimension of~\eqref{IwahoriBorelII} is at most~$n+1$.
However, if~$x \in K_0(p^{n+1})$ and~$f \in \pi_{n+1}(\chi)$ is a $(K_0(p^{n+1}), \psi)$-eigenvector then
\[
f(x) = \chi(x)f(1) = \psi(x)f(1)
\]
and so~$f(1) = 0$ since~$\chi \ne \psi$.
So the dimension of~\eqref{IwahoriBorelII} is at most~$n$.
Now we conclude the proof by induction on~$n$.
The case~$n = 1$ is a direct consequence of Lemma~\ref{eigenvectorssemisimple}, since $\psi|_{\bF_p^\times} = \chi|_{\bF_p^\times}$ if and only if~$\psi = \chi \alpha^i$ for some~$0 \leq i \leq p-1$, and~$i \neq 0, p-1$ as~$\chi \ne \psi$.
Assume the lemma true for~$n$ and recall that the unique $p^{n-1}$-dimensional $K_0(p)$-quotient of~$\pi_{n+1}(\psi)$ is isomorphic to~$\pi_n(\psi)$, and the unique $p^{n-1}$-dimensional $K_0(p)$-subspace of~$\pi_{n+1}(\chi)$ is isomorphic to~$\pi_n(\chi)$.
Then by inductive assumption the cokernel of the inclusion
\[
\Hom_{K_0(p)}(\pi_n(\psi), \pi_n(\chi)) \subset \Hom_{K_0(p)}(\pi_{n+1}(\psi), \pi_{n+1}(\chi))
\]
has dimension at most one, since the left-hand side has dimension~$n-1$ and the right-hand side has dimension at most~$n$.
Now we apply Lemma~\ref{eigenvectorssemisimple} again and choose a nonzero morphism
\[
\tau: \Ind_{K_0(p^{n+1})}^{K_0(p^n)}(\psi) \to \Ind_{K_0(p^{n+1})}^{K_0(p^n)}(\chi).
\]
To conclude, it suffices to prove that the image of $\Ind_{K_0(p^n)}^{K_0(p)}(\tau)$ is not contained in~$\pi_n(\chi)$. 
But this is true because the image of~$\tau$ strictly contains the socle of~$\Ind_{K_0(p^{n+1})}^{K_0(p^n)}(\chi)$, which is isomorphic to~$\chi$; and now, since induction is exact, it follows that the image of
\[
\Ind_{K_0(p^{n})}^{K_0(p)}(\tau) : \pi_{n+1}(\psi) \to \pi_{n+1}(\chi)
\]
strictly contains the submodule $\pi_n(\chi)$ of $\pi_{n+1}(\chi)$.
\end{proof}

\begin{lemma}\label{cokernel}
Let~$\chi_1, \chi_2: H \to k^\times$ be characters and let $\lambda: \pi_{n+1}(\chi_1) \to \pi_{n+1}(\chi_2)$ be an $\Iw$-linear morphism. 
Then either~$\chi_1 = \chi_2$ and~$\lambda$ is an isomorphism, or $\dim \coker(\lambda) \geq p^{n-1}$.
Furthermore, if~$0 \leq i \leq p-2$ and $\chi_1 =\chi_2\alpha^i$ then
\begin{equation}\label{computationtoprove}
\dim \operatorname{image}(\lambda) \in \{(i+1)p^j: 0 \leq j \leq n-1\}
\end{equation}
whenever~$\lambda \ne 0$ is not an isomorphism.
\end{lemma}
\begin{proof}
If~$\lambda$ is an isomorphism, then 
\[
\soc_{K_0(p)}\pi_{n+1}(\chi_1) \cong \soc_{K_0(p)}\pi_{n+1}(\chi_2)
\]
and so~$\chi_1 = \chi_2$.
Similarly, if $\chi_1 |_{\bF_p^\times} \ne \chi_2|_{\bF_p^\times}$ then~$\lambda = 0$ since~$\pi_{n+1}(\chi_1)$ and~$\pi_{n+1}(\chi_2)$ have different central character.
So it suffices to prove that if $\chi_1 = \chi_2\alpha^i$ and~$\lambda \ne 0$ is not an isomorphism then~\eqref{computationtoprove} is true: in fact, \eqref{computationtoprove} implies that
\[
\dim \coker(\lambda) = p^n - \dim \operatorname{image}(\lambda) \geq p^n - (p-1)p^{n-1} = p^{n-1}.
\]
We do this by induction on~$n$, noting that the case~$n = 1$ is true by Lemma~\ref{psfiltration}.
Let~$v \in \pi_{n+1}(\chi_2)$ be a $(K_0(p^{n+1}), \chi_1)$-eigenvector. 
If~$v$ is a $K_0(p^{n})$-eigenvector, then $v \in \pi_n(\chi_2)$ since $\pi_{n}(\chi_2) = \pi_{n+1}(\chi_2)^{\lbar U(p^n\bZ_p)}$.
Hence in this case we are done by induction.
So we can assume that~$v$ is not a $K_0(p^{n})$-eigenvector.
If~$\chi_1 = \chi_2$ then Lemma~\ref{eigenvectors} implies that~$v$ generates~$\pi_{n+1}(\chi_2)$, which we are assuming not to be the case since~$\lambda$ is not an isomorphism. 
So there remains to show that if $\chi_1 \ne \chi_2$ and $v$ is a $(K_0(p^{n+1}), \chi_1)$-eigenvector in~$\pi_{n+1}(\chi_2) \setminus \pi_n(\chi_2)$ then~$v$ generates an $(i+1)p^{n-1}$-dimensional $\Iw$-submodule of~$\pi_{n+1}(\chi_2)$.
By Lemma~\ref{eigenvectorssemisimple} there exists a nonzero map
\[
\tau: \Ind_{K_0(p^{n+1})}^{K_0(p^n)}(\chi_1) \to \Ind_{K_0(p^{n+1})}^{K_0(p^n)}(\chi_2)
\]
whose image has dimension~$i+1$ by Lemma~\ref{psfiltration}.
Since induction is an exact functor, this implies that there exists a $(K_0(p^{n+1}), \chi_1)$-eigenvector~$v_0 \in \pi_{n+1}(\chi_2)$ which generates an $(i+1)p^{n-1}$-dimensional $\Iw$-submodule of~$\pi_{n+1}(\chi_2)$.
Since~$i \ne 0$, this implies that~$v_0$ is not a $K_0(p^{n})$-eigenvector, because otherwise~$v_0$ would be contained in~$\pi_{n}(\chi_2) = \pi_{n+1}(\chi_2)^{\lbar U(p^n\bZ_p)}$; but~$\pi_n(\chi_2)$ has dimension~$p^{n-1} < (i+1)p^{n-1}$.
Hence Lemma~\ref{eigenvectorsII} implies that there exist nonzero $\mu, \mu_0 \in k^\times$ such that~$\mu v- \mu_0v_0 \in \pi_n(\chi_2)$.
On the other hand, since~$\pi_{n+1}(\chi_2)$ is uniserial, we know that
\[
\pi_n(\chi_2) \subset \langle \Iw \cdot v \rangle \cap \langle \Iw \cdot v_0 \rangle.
\]
Putting these together we find that $\langle \Iw \cdot v \rangle = \langle \Iw \cdot v_0 \rangle$, which implies that $\dim \langle \Iw \cdot v \rangle = (i+1)p^{n-1}$ and concludes the proof of the lemma.
\end{proof}

\begin{corollary}\label{specificeigenvector}
Let~$\chi_1, \chi_2 : H \to k^\times$ be characters and assume that $0 \leq i \leq p-2$ and $\chi_1 = \chi_2\alpha^i$.
Let~$v \in \pi_\infty(\chi_2)$ be a $(K_0(p^{n+1}), \chi_1)$-eigenvector but not a $K_0(p^n)$-eigenvector, where~$n \geq 1$.
Then $\dim\langle \Iw \cdot v \rangle$ equals~$(i+1)p^{n-1}$ if~$i \ne 0$ and~$(i+1)p^n$ if~$i = 0$. 
\end{corollary}
\begin{proof}
The Iwahori decomposition for $K_0(p^{n})$ implies that~$v$ is fixed by~$\lbar U(p^{n+1}\bZ_p)$ but not by~$\lbar U(p^n\bZ_p)$.
By Proposition~\ref{Morrathesis1}, the dimension of the space of $\lbar U(p^{i+1}\bZ_p)$-fixed vectors in~$\pi_\infty(\chi_2)$ is equal to~$p^i$.
Since~$\pi_\infty(\chi_2)$ is $\lbar U(p\bZ_p)$-uniserial, it follows that $p^{n-1} < \dim\langle \Iw \cdot v \rangle \leq p^n$. 
Now the corollary follows from Lemma~\ref{cokernel}.
\end{proof}

Next we consider the induction of characters of~$B(\bZ_p)$ to~$\Iw$. 
The resulting $\Iw$-modules were studied in~\cite{Morraexplicit}, but since we will need a slightly different perspective on these results we will reprove some of them in what follows.
Recall from~(\ref{definitiontwist}) that $\pi^+$ denotes the twist of an $\Iw$-representation~$\pi$ by $\Pi \in N$.

\begin{defn}
Let~$\chi : H \to k^\times$ be a character.
We define $\pi_\infty(\chi) = \Ind_{K_0(p^\infty)}^{K_0(p)}(\chi)$.
\end{defn}

\begin{pp}\label{Morrathesis1}
The representation $\pi_\infty(\chi)|_{\lbar U(p\bZ_p)}$ is an injective envelope of the trivial representation of~$\lbar U(p\bZ_p)$, and so $\pi_\infty(\chi)$ is both $\lbar U(p\bZ_p)$-uniserial and $\Iw$-uniserial.
Its $\Iw$-socle filtration satisfies
\[
\soc^{k+1}_\Iw(\pi_\infty(\chi))/\soc^k_\Iw(\pi_\infty(\chi)) \cong \alpha \otimes \left ( \soc^k_\Iw(\pi_\infty(\chi))/\soc_\Iw^{k-1}(\pi_\infty(\chi)) \right ).
\]
\end{pp}
\begin{proof}
The claims about the $\Iw$-action all follow from~\cite[Proposition~1.6]{Morraexplicit} and the isomorphism $\pi_\infty(\chi) \cong \varinjlim_n \pi_n(\chi)$.
By the Iwahori decomposition, we know that $\Iw = B(\bZ_p)\lbar U(p\bZ_p)$, and so by~\cite[Section~I.5.5]{Vignerasrepsbook} we have
\[
\pi_\infty(\chi)|_{\lbar U(p\bZ_p)} \cong \Ind_{1}^{\lbar U(p\bZ_p)}(1).
\]
By Frobenius reciprocity, this representation is $\lbar U(p\bZ_p)$-injective, and its $\lbar U(p\bZ_p)$-socle is one-dimensional.
Hence to prove that it is uniserial it suffices to prove that $\pi_\infty(\chi)/\soc_{\lbar U(\bZ_p)}\pi_\infty(\chi) \cong \pi_\infty(\chi)$ as $\lbar U(\bZ_p)$-representations, which follows from~\cite[Proposition~5.9]{Paskunasextensions}.
Alternatively, notice that the Pontrjagin dual of $\Ind_{1}^{\lbar U(p\bZ_p)}(1)$ is free of rank one over the Iwasawa algebra of~$\lbar U(p\bZ_p) \cong \bZ_p$, which is isomorphic to a power series ring in one variable over~$k$, and so $\Ind_{1}^{\lbar U(p\bZ_p)}(1)$ is uniserial since the radical filtration of~$k[[X]]$ as a module over itself has irreducible graded pieces.
\end{proof}

\begin{lemma}
There is a split short exact sequence of $\Iw$-representations
\begin{equation}\label{Iwahorisplit}
0 \to \pi_{\infty}^+(\chi) \to \Ind_{B(\bZ_p)}^K(\chi) \xrightarrow{\res} \pi_\infty(\chi) \to 0
\end{equation}
where the third arrow denotes restriction of functions from~$K$ to~$\Iw$.
\end{lemma}
\begin{proof}
Recall the decomposition $G = B\Iw \coprod Bs\Iw$, which implies that $K = \Iw \coprod B(\bZ_p) s \Iw$. 
By the Mackey decomposition~\cite[Section~I.5.5]{Vignerasrepsbook} the sequence~\eqref{Iwahorisplit} is split, and the kernel of restriction is isomorphic to
\[
\Ind_{s^{-1}B(\bZ_p)s\cap \Iw}^\Iw(\ad(s)^*\chi).
\] 
Since $s^{-1}B(\bZ_p)s \cap \Iw = \Pi^{-1} B(\bZ_p) \Pi$, and $\ad(s)^*\chi = \ad(\Pi)^*\chi$ as characters of this group, it follows from Lemma~\ref{twistinduction} that the kernel of restriction is isomorphic to $\pi_\infty^+(\chi)$ (notice that~$\Pi^2$ is central in~$\GL_2(\bQ_p)$).
\end{proof}

\begin{pp}\label{pstwist}
Let $\chi_1, \chi_2 : H \to k^\times$ be characters. Then 
\[
\Hom_{\Iw}(\pi_{\infty}(\chi_1), \pi_{\infty}^+(\chi_2)) = 0 \text{ and } \Hom_{\Iw}(\pi_{\infty}^+(\chi_1), \pi_{\infty}(\chi_2)) = 0. 
\]
\end{pp}
\begin{proof}
Let $\lambda: \pi_{\infty}(\chi_1) \to \pi_{\infty}^+(\chi_2)$ be a nonzero $\Iw$-linear homomorphism. By Lemma~\ref{uniserial}, it is surjective. But then the socle filtration of~$\pi_{\infty}^+(\chi_2)$ has the same property
\[
\soc^{k+1}(\pi_{\infty}^+(\chi_2))/\soc^k(\pi_{\infty}^+(\chi_2)) \cong \alpha \otimes \left ( \soc^k(\pi_{\infty}^+(\chi_2))/\soc^{k-1}(\pi_{\infty}^+(\chi_2)) \right )
\]
as that of~$\pi_{\infty}(\chi_2)$, contradicting the fact that $\pi_{\infty}^+(\chi_2) = \ad(\Pi)^* \pi_{\infty}(\chi_2)$ and $\ad(\Pi)^*\alpha = \alpha^{-1} \not = \alpha$ (since $p > 3$). 
The claim that $\Hom_{\Iw}(\pi_{\infty}^+(\chi_1), \pi_{\infty}(\chi_2)) = 0$ follows upon twisting by~$\Pi$.
\end{proof}

Our first main result in this section is the following theorem. 
The extra generality in considering quotients by proper subspaces will be useful in Section~\ref{residuallyreducible}.

\begin{thm}\label{psIwahori1}
Let $\chi_1, \chi_2: H \to k^\times$ be characters and fix $\Iw$-stable proper subspaces $X_i \subset \pi_{\infty}(\chi_i)$. Let $\lambda: \pi_{\infty}(\chi_1)/X_1 \to \pi_{\infty}(\chi_2) / X_2$ be an $\Iw$-linear morphism. If $\lambda \not = 0$, then $\chi_1 = \chi_2$, $X_1 \subseteq X_2$ and~$\lambda$ is a scalar multiple of the canonical surjection $\pi_{\infty}(\chi)/X_1 \to \pi_{\infty}(\chi)/X_2$.
\end{thm}
\begin{proof}
Assume first $\lambda \not = 0$ and $\dim(X_1) \geq \dim(X_2)$. 
It follows that for all~$n$ such that~$p^n > \dim(X_1)$, the dimension of~$\pi_{n+1}(\chi_1)/X_1$ is smaller than the dimension of~$\pi_{n+1}(\chi_2)/X_2$.
Hence for all~$n$ large enough we have $\lambda(\pi_{n+1}(\chi_1)/X_1) \subseteq \pi_{n+1}(\chi_2)/X_2$, because the dimension of the left-hand side is smaller than the dimension of the right-hand side, and the representation~$\pi_{\infty}(\chi_2)/X_2$ is uniserial. 
On the other hand, the dimension of~$\pi_{n+1}(\chi_1)/X_1$ tends to infinity with~$n$, hence if~$n$ is large enough we have $\ker(\lambda) \subset \pi_{n+1}(\chi_1)/X_1$.
Hence
\[
\dim \coker(\lambda: \pi_{n+1}(\chi_1)/X_1 \to \pi_{n+1}(\chi_2)/X_2) = \dim \ker(\lambda) + \dim(X_1) - \dim(X_2).
\]
For any~$n$ large enough, the surjection of $\pi_{n+1}(\chi_2)$ onto~$\pi_n(\chi_2)$ induces a surjection
\[
\pi_{n+1}(\chi_2)/X_2 \to \pi_n(\chi_2)
\]
since the dimension of~$X_2$ is fixed, whereas the dimension of~$K = \ker(\pi_{n+1}(\chi_2) \to \pi_n(\chi_2))$ tends to infinity with~$n$ (it is equal to $p^n - p^{n-1} = p^{n-1}(p-1)$).
Consider the composition
\begin{equation}\label{composition}
\pi_{n+1}(\chi_1) \to \pi_{n+1}(\chi_1)/X_1 \xrightarrow{\lambda} \pi_{n+1}(\chi_2)/X_2 \to \pi_{n}(\chi_2).
\end{equation}
If~$p^{n-1} > \dim \ker(\lambda) + \dim(X_1) - \dim(X_2)$ then $\image(\lambda: \pi_{n+1}(\chi_1)/X_1 \to \pi_{n+1}(\chi_2)/X_2)$ contains~$K/X_2$, since its dimension~$p^n-\dim(X_1) - \dim \ker \lambda$ is larger than~$\dim (K/X_2) = p^n - p^{n-1}-\dim(X_2)$.
Hence, for~$n$ large enough, the cokernel of the composition~\eqref{composition} has dimension~$\dim \ker(\lambda) + \dim(X_1) - \dim(X_2)$.
Furthermore, it factors through the $p^{n-1}$-dimensional quotient of~$\pi_{n+1}(\chi_1)$, which is isomorphic to~$\pi_{n}(\chi_1)$.

We conclude that for every~$n$ large enough there is a morphism $\lambda_n : \pi_n(\chi_1) \to \pi_n(\chi_2)$ whose cokernel has dimension~$\dim \ker(\lambda) + \dim(X_1) - \dim(X_2)$, which is independent of~$n$.
Taking~$n$ large enough, this implies by Lemma~\ref{cokernel} that $\chi_1 = \chi_2$, $\dim(X_1) = \dim(X_2)$ and~$\lambda$ is injective. 
By Lemma~\ref{uniserial}, this implies that~$\lambda$ is an isomorphism.
Hence without loss of generality we can assume~$\chi_1 = \chi_2 = \chi$, $X_1 = X_2 = X$, and there remains to prove that~$\lambda$ is a scalar.
Looking at the induced map on the~$\Iw$-socle, we see that there exists a scalar $x \in k$ such that
\[
\lambda - x : \pi_{\infty}(\chi)/X \to \pi_{\infty}(\chi)/X
\]
has nontrivial kernel. 
In fact, $\soc_\Iw(\pi_\infty(\chi)/X)$ is one-dimensional, hence there exists a scalar~$x$ such that the restriction of~$\lambda$ to $\soc_\Iw(\pi_\infty(\chi)/X)$ coincides with multiplication by~$x$.
This implies that~$\lambda - x$ is not injective.
But then~$\lambda = x$, since $\lambda - x$ is either zero or an isomorphism, by what we have just proved.

Now assume that~$\dim(X_1) < \dim(X_2)$ and $\lambda \not = 0$. 
Consider the injection $\overline{\lambda}: \pi_{\infty}(\chi_1) / \ker(\lambda) \to \pi_{ \infty}(\chi_2)/X_2$.
Since~$\overline{\lambda}$ is injective, it is an isomorphism by Lemma~\ref{uniserial}. 
If $\dim \ker(\lambda) \geq \dim(X_2)$, the case we have just treated implies that $\chi_1 = \chi_2$ and $\ker(\lambda) = X_2$, and furthermore~$\overline{\lambda}$ is a scalar, which implies that~$\lambda$ is a multiple of the canonical surjection. 
On the other hand, if~$\dim \ker(\lambda) < \dim(X_2)$ then the inverse $\pi_{\infty}(\chi_2)/X_2 \to \pi_{\infty}(\chi_1)/\ker(\lambda)$ is of the type we have just treated.
This implies that $\chi_1 = \chi_2$ and $X_2 = \ker(\lambda)$, which is a contradiction.
\end{proof}

\begin{corollary}\label{cor:psIwahori}
The space $\Hom_{\Iw}(\pi_{\infty}(\chi_1), \pi_{\infty}(\chi_2))$ is one-dimensional if $\chi_1 = \chi_2$, and vanishes otherwise.
\end{corollary}
\begin{proof}
Follows immediately from Theorem~\ref{psIwahori1} for $X_1 = 0$, $X_2 = 0$.
\end{proof}

\begin{corollary}\label{cor:psmorphisms}
Let~$\pi_1, \pi_2$ be generic principal series representations of~$k[\GL_2(\bQ_p)]$ of central character~$\zeta$. 
If~$\pi_1$ and~$\pi_2$ have nonisomorphic $K$-socle, then $\Hom_{\Iw}(\pi_1, \pi_2) = 0$. 
Otherwise, $\dim_k \Hom_{KZ}(\pi_1, \pi_2) = \dim_k \Hom_N(\pi_1, \pi_2) = 1$.
\end{corollary}
\begin{proof}
Let $\pi_i = \Ind_B^G(\tld \kappa_i)$ for a smooth character $\tld \kappa_i : T \to k^\times$, and let $\kappa_i = \tld \kappa_i |_H$.
Then by~\eqref{Iwahorisplit}, both representations decompose as $\pi_i|_{\Iw} \cong \pi_\infty(\kappa_i) \oplus \pi_{\infty}^+(\kappa_i)$, and~$\kappa_i \cong \left ( (\soc_K \pi_i)^{\Iw_1} \right )^+$. 
(To explain the twist, notice that $\pi|_K \cong \Ind_{B(\bZ_p)}^K(\kappa_i)$ implies that~$\pi^{K_1} \cong \Ind_{\Iw}^K(\kappa_i)$, which implies the given formula for~$\kappa_i$.)
Let $\lambda: \pi_1 \to \pi_2$ be $\Iw$-linear. 
By Proposition~\ref{pstwist} and Corollary~\ref{cor:psIwahori}, it preserves the summands in the decomposition. 
Furthermore, if~$\kappa_1 \not = \kappa_2$ then the restrictions to both summands are zero.
Since the~$\kappa_i$ determine the generic Serre weights $\soc_K(\pi_i)$, this implies that if $\soc_K(\pi_1) \not \cong \soc_K(\pi_2)$ then~$\lambda = 0$.

Otherwise, assume~$\kappa_1 = \kappa_2 = \kappa$.
Then the dimension computations follow from Frobenius reciprocity together with the isomorphisms $\pi_i|_N \cong \Ind_{\Iw Z}^N\pi_\infty(\kappa)$ and $\pi_i|_K \cong \Ind_{\Iw}^K \pi_\infty(\kappa)$.  
In more detail, choose $N$-linear isomorphisms $\pi_i \cong \Ind_{\Iw Z}^N\pi_\infty(\kappa) = \pi_\infty(\kappa) \oplus \pi_{\infty}^+(\kappa)$ and assume that $\lambda : \pi_1 \to \pi_2$ is $N$-linear.
Let~$x, x^+ \in k$ be such that $\lambda = x \cdot \id_{\pi_\infty(\kappa)} + x^+ \cdot\id_{\pi_\infty^+(\kappa)}$. 
Notice that~$\Pi$ switches the two summands in the decomposition 
\[
\pi_i|_{\Iw} \cong \pi_\infty(\kappa) \oplus \pi_{\infty}^+(\kappa).
\]
This immediately implies that $x = x^+$. 
(This can also be shown by noticing that $\Pi(\pi_\infty(\kappa))$ is an $\Iw$-subrepresentation of~$\pi_i$ isomorphic to~$\pi_\infty^+(\kappa)$, and so it is contained in the summand~$\pi_\infty^+(\kappa)$ by Lemma~\ref{pstwist}.
By Lemma~\ref{uniserial} it follows that $\Pi(\pi_\infty(\kappa))$ coincides with the summand~$\pi_\infty^+(\kappa)$.)

On the other hand, assume that~$\lambda$ is $KZ$-linear and choose $K$-linear isomorphisms $\pi_i|_K \cong \Ind_{\Iw}^K\pi_\infty(\kappa) = \pi_\infty(\kappa) \oplus \pi_{\infty}^+(\kappa)$. 
Let $\varphi \in \pi_\infty(\kappa)^{K_1}$ be the function supported on~$\Iw$ and sending~$1$ to~$1$. 
Then $s \varphi$ is supported on the complement of~$\Iw$, because $\varphi\left ( \fourmatrix a b {pc} d s\right ) = \varphi \fourmatrix b a d {pc}$ and $d \in \bZ_p^\times$.
So we have
\[
\lambda(s \varphi) = x^+ s \varphi \text{ and } \lambda(s \varphi) = s(\lambda(\varphi)) = s(x \varphi) = x s \varphi
\]
which implies $x^+ = x$, because $s \varphi \not = 0$.
\end{proof}

\subsubsection{Interaction between~$KZ$ and~$N$.}
Corollary~\ref{cor:psmorphisms} implies that there exist nonzero $KZ$-linear morphisms between nonisomorphic principal series of the same Serre weight: this is also a direct consequence of the fact that the Hecke eigenvalue $\lambda$ does not appear at the right-hand side of~(\ref{Iwasawadecomposition}).
The results of this subsection indicate how~$\lambda$ interacts with these $KZ$-morphisms.
They will be applied in Section~\ref{atomes}.
In this subsection we fix a generic Serre weight~$\sigma = \Sym^r(k^2) \otimes \det^s$. 
Recall that we write~$\chi$ for the $\Iw$-character~$\sigma^{\Iw_1}$, and that $\kappa = \chi^+$, where~$\chi^+ = \chi \circ \ad(\Pi)$.
Furthermore, we fix~$\lambda \in k^\times$ and we write $\pi = \cInd_{KZ}^G(\sigma)/(T-\lambda)$.
The image of the $H$-eigenvectors~$x^{r-i}y^i \in \sigma$ in~$\pi$ will be denoted $[1, x^{r-i}y^i]$ for all~$0 \leq i \leq r$.
Recall from~\eqref{Iwahorisplit} the decomposition $\pi|_\Iw = \pi_\infty(\kappa)^+ \oplus \pi_\infty(\kappa)$.

\begin{lemma}\label{diagonalKsocle}
For all~$0 \leq i < r$ the vector~$\psi_i = \Pi[1, x^{r-i}y^i]$ is a $(K_0(p^2), \kappa\alpha^i)$-eigenvector in~$\pi_\infty(\kappa)$.
On the other hand, there exist a nonzero $\psi^+ \in \soc_\Iw \pi_\infty^+(\kappa)$ and a nonzero $(K_0(p^2), \kappa\alpha^r)$-eigenvector $\psi_r \in \pi_\infty(\kappa)$ such that $\Pi[1, y^r] = \psi^+ + \psi_r$.
\end{lemma}
\begin{proof}
Since $\soc_\Iw\sigma \cong \kappa^+$ as $\Iw$-representations, the intersection $\sigma \cap \pi_\infty(\kappa)$ equals~$0$.
Since~$\pi_\infty(\kappa)^{K_1}$ is one-dimensional, the projection of~$\sigma$ to~$\pi_\infty(\kappa)$ is at most one-dimensional, and so
$\sigma \cap \pi_\infty^+(\kappa)$ has codimension at most one in~$\sigma$.
We claim that the codimension is exactly one.
For this, it suffices to notice that
\[
\sigma \subset \Ind_\Iw^K(\kappa) = \pi_2(\kappa)^+ \oplus \kappa \subset \pi_\infty^+(\kappa) \oplus \pi_\infty(\kappa)
\]
and so if $\sigma \subset \pi_\infty^+(\kappa)$ then $\Ind_\Iw^K(\kappa)/\sigma$ has a two-dimensional space of~$\Iw_1$-invariants, which is a contradiction (since this quotient is a Serre weight).
It follows that $[1, x^{r-i}y^i] \in \pi_\infty^+(\kappa)$ for all~$0 \leq i < r$, and there exist nonzero $(HK_1, \kappa)$-eigenvectors~$\varphi_r \in \pi_\infty(\kappa)^+$ and~$\varphi \in \soc_\Iw\pi_\infty(\kappa)$ such that
\begin{equation}\label{yposition}
[1, y^r] = \varphi_r + \varphi.
\end{equation}
The lemma follows from this and the observation that if~$x$ is an $(HK_1, \xi)$-eigenvector then~$\Pi x$ is a $(K_0(p^2), \xi^+)$-eigenvector, as well as the computation $\kappa^+ = \kappa \alpha^r$.
\end{proof}

\begin{lemma}\label{basecase}
With the notation of Lemma~\ref{diagonalKsocle}, choose a nonzero $v^+ \in \soc_\Iw(\sigma) = \soc_\Iw \pi_\infty(\kappa)^+ \subset \pi$. 
Then there exists a nonzero $\nu \in k^\times$ such that
\[
\tau v^+ = \lambda^{-1}v^+ + \nu\psi_r
\]
where~$\tau$ is the diagonal matrix~$\tau = \diag(1, p)$.
\end{lemma}
\begin{proof}
Twisting by~$\omega \circ \det^s$ we can assume that~$\sigma = \Sym^r(k^2)$.
Since~$v^+$ is a $(K_0(p), \kappa^+)$-eigenvector, $\tau v^+$ is a $(K_0(p^2), \kappa^+)$-eigenvector.
It follows from Proposition~\ref{Morrathesis1} that~$\pi_\infty(\kappa)^+$ is an injective envelope of the trivial representation of~$U(\bZ_p)$, and so its space of $U(\bZ_p)$-invariants is one-dimensional.
On the other hand, we know from Lemma~\ref{eigenvectorsII} that the $(K_0(p^2), \kappa^+)$-eigenspace in $\pi_\infty(\kappa)$ is one-dimensional.
Hence there exist~$\mu, \nu \in k$ such that $\tau v^+ = \mu v^+ +\nu\psi_r$, and we have to prove that~$\nu \ne 0$ and that~$\mu = \lambda^{-1}$.
If~$\nu = 0$ then $\tau v^+ \in \pi_\infty(\kappa)^+$ and so $\zeta(p)s v^+ = \Pi \tau v^+ \in \pi_\infty(\kappa)$, which contradicts~\eqref{yposition}, since~$s v^+$ is a nonzero multiple of~$[1, y^r]$.
Now choose a $G$-linear isomorphism $\pi \cong \Ind_B^G(\nr_{\lambda^{-1}} \otimes \nr_\lambda \omega^{r+1})$. 
Since~$\mu$ does not change if we replace~$v^+$ by a nonzero scalar multiple, we can assume that~$v^+$ is the restriction to~$K$ of the unique $(\Iw, \kappa^+)$-eigenvector $f: G \to \kappa$ supported in~$B s \Iw$ and sending~$s$ to~$1$.
Then we find that
\[
\tau v^+(s) = v^+ \fourmatrix 0 p 1 0 = v^+(ts) = \lambda^{-1}v^+(s) = \lambda^{-1}
\]
and so~$\mu = \lambda^{-1}$ since~$\psi_r(s) = 0$.
\end{proof}

\begin{pp}\label{exchange}
Let~$\pi_i = \cInd_{KZ}^G(\sigma)/(T-\lambda_i)$ for $\lambda_1, \lambda_2 \in k^\times$ be generic principal series representations of the same Serre weight and central character. 
Let $\alpha: \pi_1 \to \pi_2$ be a nonzero $KZ$-linear homomorphism and let~$z \in \soc_\Iw \pi_\infty^+(\kappa_1)$ be a generator.
Then for all~$n \geq 0$ we have
\begin{equation}\label{exchangeI}
\alpha(\tau^n \Pi z) = \lambda_2^{-(n+1)}\lambda_1^{n+1} \tau^n \Pi \alpha(z).
\end{equation}
Hence if~$\psi_0 \in \soc_\Iw \pi_\infty(\kappa_1)$ is a generator then for all~$n \geq 0$ we have
\begin{equation}\label{exchangeII}
\alpha(\tau^n \psi_0) = \lambda_2^{-n}\lambda_1^{n} \tau^n \alpha(\psi_0).
\end{equation}
\end{pp}
\begin{proof}
Since~\eqref{exchangeII} follows from~\eqref{exchangeI} applied to~$z = \Pi^{-1}\psi_0$, it suffices to prove~\eqref{exchangeI}.
By Corollary~\ref{cor:psmorphisms} we know that~$\alpha$ is an isomorphism.
Fix~$n \geq 0$.
We claim that it suffices to prove that there exists~$\mu \in k^\times$ such that $\alpha(\tau^n \Pi z) = \mu \tau^n \Pi \alpha(z)$.
Indeed, by~\cite[(19)]{BLreps}, the Hecke operator~$T$ acts on~$(\soc_K \pi_i)^{\Iw_1}$ by the formula
\[
z \mapsto \sum_{\lambda \in \bF_p} \fourmatrix 1 {[\lambda]} 0 1 \fourmatrix p 0 0 1 z.
\]
Hence the right-hand side equals~$\lambda_i z$ in our setting.
Iterating this, we deduce that there exists an element~$k_{n+1}$ of the group algebra~$k[K]$ such that $k_{n+1}\tau^n\Pi z = \lambda_1^{n+1} z$ and $k_{n+1}\tau^n \Pi \alpha(z) = \lambda_2^{n+1}\alpha(z)$.
The $K$-linearity of~$\alpha$ now implies that
\[
\lambda_1^{-(n+1)}\mu \cdot k_{n+1}(\tau^n \Pi\alpha(z)) = \alpha(z) = \lambda_2^{-(n+1)}\cdot k_{n+1}(\tau^n \Pi \alpha(z))
\]
and since~$\alpha(z) \ne 0$ this implies that~$\mu = \lambda_2^{-(n+1)}\lambda_1^{n+1}$, concluding the proof of the proposition.
There remains to prove that $\alpha(\tau^n \Pi z)$ and~$\tau^n \Pi \alpha(z)$ are collinear.
This is true when~$n = 0$ since $\alpha(\Pi z)$ and~$\Pi \alpha(z)$ are both generators of $\soc_\Iw \pi_\infty(\kappa_2)^+$.
In general, it suffices to prove that for all nonzero $(\Iw, \kappa_i)$-eigenvectors $\varphi_i \in \pi_i$ the support of $\tau^n \varphi_i|_K$ is~$K_0(p^{n+1})$.
In fact, let 
\[
\pi_i \isom \Ind_{B(\bZ_p)}^K(\kappa_i)
\]
be the restriction map to~$K$.
Since by hypothesis~$\kappa_1 = \kappa_2$, by Corollary~\ref{cor:psmorphisms} the map~$\alpha$ corresponds to a nonzero scalar multiple of the identity under these isomorphisms.
It follows that $\alpha(\tau^n \Pi z)$ and~$\tau^n \Pi \alpha(z)$ are $(K_0(p^{n+1}), \kappa)$-eigenvectors in~$\Ind_{B(\bZ_p)}^K(\kappa)$ supported in~$K_0(p^{n+1})$, and the space of such eigenvectors is one-dimensional.
So let~$\varphi_i \in \pi_i$ be a nonzero $(\Iw, \kappa_i)$-eigenvector.
Since $\pi_i^{\Iw_1}$ is two-dimensional and $G = Bs\Iw \coprod B\Iw$, we know that~$\varphi_i$ is supported in $B\Iw$.
Hence~$\tau^n \varphi_i$ is supported in $B\Iw \tau^{-n}$, and we have to prove that $B\Iw \tau^{-n} \cap K = K_0(p^{n+1})$.
Taking inverses, it suffices to prove that $\tau^n \Iw B \cap K = K_0(p^{n+1})$.
The Iwahori decomposition implies that $\tau^n \Iw B = \lbar U(p^{n+1}\bZ_p)B$.
It follows that $\tau^n \Iw B \cap K = \lbar U(p^{n+1}\bZ_p)B(\bZ_p) = K_0(p^{n+1})$ using the Iwahori decomposition for~$K_0(p^{n+1})$.
\end{proof}
\begin{corollary}\label{comparetauaction}
With the notation of Proposition~\ref{exchange}, let~$\psi \in \pi_\infty(\kappa)$ be a $K_0(p^2)$-eigenvector.
Then $\alpha(\tau^k \psi) = \lambda_2^{-k}\lambda_1^k \tau^k\alpha(\psi)$ for all~$k \geq 0$.
\end{corollary}
\begin{proof}
When~$\psi$ is a generator of~$\soc_\Iw \pi_\infty(\kappa)$ this is true by Proposition~\ref{exchange}.
Now let~$\psi$ be a $K_0(p^2)$-eigenvector but not a $K_0(p)$-eigenvector, and fix a generator $\psi_0 \in \soc_\Iw \pi_\infty(\kappa)$.
We have seen in the proof of Proposition~\ref{exchange} that~$\tau \psi_0$ is a $(K_0(p^2), \kappa)$-eigenvector supported in~$K_0(p^2)$. 
Now Lemma~\ref{cokernelvarphi} implies that~$\tau \psi_0$ generates the $p$-dimensional $\Iw$-submodule $\pi_2(\kappa) \subset \pi_\infty(\kappa)$.
By Proposition~\ref{Morrathesis1} we know that $\pi_2(\kappa) = \pi_\infty(\kappa)^{\lbar U(p^2\bZ_p)}$ is $\lbar U(p\bZ_p)$-uniserial, and so there exists $h\in k[\lbar U(p\bZ_p)]$ such that $\psi = h \tau \psi_0$.
Since~$\tau$ normalizes~$\lbar U(p\bZ_p)$, for all~$k\geq 1$ there exists~$h_k \in k[\lbar U(p\bZ_p)]$ such that $\tau^k h \tau^{-k} = h_k$ as endomorphisms of any $k[G]$-module.
It follows that 
\[
\alpha(\tau^k \psi) = \alpha(h_k \tau^{k+1} \psi_0) = h_k\lambda_2^{-(k+1)}\lambda_1^{k+1} \tau^{k+1}(\alpha(\psi_0)) = \lambda_2^{-(k+1)}\lambda_1^{k+1}\tau^k h\tau\alpha(\psi_0) = \lambda_2^{-k}\lambda_1^k\tau^k\alpha(\psi)
\]
where the final equality follows from $\alpha(\psi) = h\alpha(\tau\psi_0) = h\lambda_2^{-1}\lambda_1\tau \alpha(\psi_0)$.
\end{proof}

\subsection{Restriction to~$\Iw Z$ and~$N$ of supersingular representations.}\label{secsupersingular}
We begin this section by recalling the $K$-structure and $\Iw$-structure of irreducible generic supersingular representations of~$G$, following the viewpoint of~\cite{Paskunasextensions} (see~\cite{Morraexplicit} for a different take on this).
Let~$\sigma$ be a generic Serre weight and let~$\pi = \cInd_{KZ}^G(\sigma)/T$.
Recall that~$\chi$ denotes the $\Iw$-character~$\sigma^{\Iw_1}$.
The $K$-socle of~$\pi$ has $K$-length two and contains two nonisomorphic Serre weights~$\{\sigma, \sigma^{[s]}\}$.
Notice that $(\sigma^{[s]})^{\Iw_1}$ is conjugate to~$\chi$ under the Weyl group, hence it is the character usually denoted~$\chi^s$: this explains the notation~$\sigma^{[s]}$ (although we will sometimes write~$\chi^+$ for~$\chi^s$, for compatibility with the notation for~$\ad(\Pi)$).
We define
\[
\pi_\sigma = \langle G^+ \cdot \sigma \rangle.
\]
It follows from~\cite[Corollary~6.5, Corollary~6.6]{Paskunasextensions} that $\pi \cong \Ind_{G^+}^G(\pi_\sigma)$, and that~$\pi|_{G^+} = \pi_\sigma \oplus \pi_{\sigma^{[s]}}$. 
\textcolor{black}{It follows from this that~$\pi_{\sigma}$ is an absolutely irreducible~$G^+$-representation: indeed, since~$\Ind_{G^+}^G(-)$ is an exact functor, the reducibility of~$\pi_{\sigma}$ would imply the reducibility of~$\pi$.}
In addition, $\pi_\sigma \otimes_k \lbar k$ is the representation denoted~$\pi_\sigma$ in~\cite{Paskunasextensions}, which works with coefficients in~$\cbF_p$.
For this reason, in this section we will sometimes work over~$\lbar k$.

\subsubsection{The $\Iw$-representation~$M_\sigma$.} 
Let~$v_\sigma \in \pi$ be a generator of the line $\sigma^{\Iw_1}$. Following~\cite[Definition~4.5]{Paskunasextensions}, introduce
\[
M_\sigma = \left \langle \fourmatrix {p^{2\bZ_{\geq 0}}} {\bZ_p} 0 1 \cdot v_\sigma \right \rangle.
\]
By~\cite[Lemma~4.6]{Paskunasextensions}, this is $\Iw$-stable in~$\pi$. 
By~\cite[Definition~4.11, Corollary~6.4]{Paskunasextensions}, we have a short exact sequence
\begin{equation}\label{Iwsupersingular}
0 \to \pi_\sigma^{\Iw_1} \to M_{\sigma} \oplus \Pi M_{\sigma^{[s]}} \to \pi_\sigma \to 0
\end{equation}
of $\Iw$-representations. 
Notice that $\Pi M_{\sigma^{[s]}}$ is an $\Iw$-representation isomorphic to the twist $M_{\sigma^{[s]}}^+ = \ad(\Pi)^* M_{\sigma^{[s]}}$. 
The space of invariants~$\pi_{\sigma}^{\Iw_1} \cong \chi$ is one-dimensional, and the inclusion in~\eqref{Iwsupersingular} is the diagonal embedding, hence we have exact sequences
\begin{gather*}
0 \to M_\sigma \to \pi_\sigma \to M^+_{\sigma^{[s]}}/\chi \to 0\\
0 \to M^+_{\sigma^{[s]}} \to \pi_\sigma \to M_\sigma/\chi \to 0
\end{gather*}
where the projection restricts to the canonical surjection on~$M^+_{\sigma^{[s]}}$ and $M_\sigma$ respectively.

\begin{pp}\label{sstwist}
Let~$\sigma_1, \sigma_2$ be generic Serre weights. Then $\Hom_{\Iw}(M_{\sigma_1}, M_{\sigma_2}^+/\chi^+_2) = 0$ and $\Hom_{\Iw}(M_{ \sigma_1}^+, M_{\sigma_2} / \chi_2) = 0$.
\end{pp}
\begin{proof}
This is similar to Proposition~\ref{pstwist}. 
It suffices to prove the result over~$\lbar k$.
By~\cite[Proposition~4.7]{Paskunasextensions}, the injection $\chi \to M_\sigma$ is an injective envelope in the category of smooth representations of~$HU(\bZ_p)$. By~\cite[Proposition~5.9]{Paskunasextensions}, the representation $M_\sigma |_{HU(\bZ_p)}$ is uniserial and the layers of its socle filtration satisfy
\begin{equation}\label{socle}
\soc^{k+1}(M_\sigma) / \soc^{k}(M_\sigma) \cong \alpha^{-1} \otimes \soc^{k}(M_\sigma) / \soc^{k-1}(M_\sigma).
\end{equation}
By the following lemma, we see that~(\ref{socle}) is also true for the $\Iw$-socle filtration of~$M_\sigma$.

\begin{lemma}\label{supersingularuniserial}
The representation $M_\sigma |_{\Iw}$ is uniserial.
The socle filtration of~$M_\sigma|_{HU(\bZ_p)}$ is a filtration by $\Iw$-stable subspaces, and it coincides with the socle filtration of $M_\sigma |_{\Iw}$. 
\end{lemma}
\begin{proof}
Since every $\Iw$-subspace of~$M_\sigma$ is $HU(\bZ_p)$-stable, the $\Iw$-subspaces of~$M_\sigma$ are totally ordered by inclusion, and so~$M_\sigma |_{\Iw}$ is uniserial.
The irreducible $\lbar k$-representations of $HU(\bZ_p)$ and~$\Iw$ are both inflated from characters of~$H$, hence the rest of the lemma follows from Lemma~\ref{soclefiltration}.
\end{proof}

Twisting by~$\Pi$, we find that 
\[
\soc_\Iw^{k+1}(M_{\sigma}^+) / \soc_\Iw^{k}(M_{\sigma}^+) \cong \alpha \otimes \soc_\Iw^{k}(M^+_\sigma) / \soc_\Iw^{k-1}(M^+_\sigma).
\]
because $\alpha^+ = \alpha^{-1}$. Now the proof proceeds as for Proposition~\ref{pstwist}.
\end{proof}

\begin{defn}
The submodule $M_{\sigma, n} \subset M_\sigma$ is defined to be $\langle B(\bZ_p) \cdot t^{2n}v_\sigma \rangle$. Since $t^{2n}v_\sigma$ is a $T(\bZ_p)$-eigenvector, this is the same as $\langle U(\bZ_p) \cdot t^{2n}v_\sigma \rangle$, which is the submodule defined in the proof of~\cite[Proposition~4.7]{Paskunasextensions}.
\end{defn}

\begin{lemma}\label{structureMsigma}
The module~$M_{\sigma, n}$ is $\Iw$-stable, and isomorphic to a quotient of
\[
\Ind_{K_0^+(p^{2n+1})}^{K_0(p)}(\chi_i) \cong \left ( \Ind_{K_0(p^{2n+1})}^{K_0(p)}(\chi_i^+) \right )^+.
\]
\end{lemma}
\begin{proof}
Since both $M_\sigma |_{\Iw}$ and $M_\sigma |_{HU(\bZ_p)}$ are uniserial, the module $M_{\sigma, n}$ is $\Iw$-stable and $M_{\sigma, n} = \left \langle \Iw \cdot t^{2n}v_\sigma \right \rangle$. 
Since $v_{\sigma_i}$ is an $\Iw$-eigenvector, the vector $t^{2n} v_{\sigma_i}$ is an eigenvector for
\[
t^{2n} \cdot \Iw \cdot t^{-2n} = \fourmatrix{\bZ_p^\times}{p^{2n} \bZ_p}{p^{-2n+1}\bZ_p}{\bZ_p^\times}
\]
hence it is an eigenvector for
\[
K_0^+(p^{2n+1}) = \Pi \cdot K_0(p^{2n+1}) \cdot \Pi^{-1} = \fourmatrix{\bZ_p^\times}{p^{2n}\bZ_p}{p\bZ_p}{\bZ_p^\times}
\] 
with the same $H$-eigencharacter as~$v_{\sigma_i}$, namely $\chi_i \cong \sigma_i^{\Iw_1}$.
The isomorphism~$\Ind_{K_0^+(p^{2n+1})}^{K_0(p)}(\chi_i)\cong \left ( \Ind_{K_0(p^{2n+1})}^{K_0(p)}(\chi_i^+) \right )^+$ is a consequence of Lemma~\ref{twistinduction} to follow.
\end{proof}

The following is the analogue of Theorem~\ref{psIwahori1} for supersingular representations.

\begin{thm}\label{ssIwahori}
Let $\sigma_1, \sigma_2$ be generic Serre weights, with $\sigma_i^{\Iw_1} = \chi_i$, and let $X_i \subset M_{\sigma_i}$ be proper (hence finite-dimensional) $\Iw$-stable subspaces. Then $\Hom_{k[\Iw]}(M_{\sigma_1}/X_1, M_{\sigma_2}/X_2)$ is one-dimensional if $\sigma_1 = \sigma_2$ and~$X_1 \subseteq X_2$, and vanishes otherwise. 
\end{thm}
\begin{proof}
\textcolor{black}{It suffices to prove the same result over~$\lbar k$.
To do so, it suffices to prove that if there exists an $\Iw$-linear isomorphism $M_{\sigma_1}/X_1 \cong M_{\sigma_2}/X_2$ then~$\sigma_1 = \sigma_2$ and~$X_1 = X_2$.
In fact, assume that this is true, and let 
\[
\lambda: M_{\sigma_1}/X_1 \to M_{\sigma_2} / X_2
\]
be a nonzero $\Iw$-linear morphism.
Then~$\lambda$ induces an injection
\[
\lambda: M_{\sigma_1}/\ker(\lambda) \to M_{\sigma_2}/X_2,
\]
which is an isomorphism by Lemma~\ref{uniserial}. 
By our assumption, this implies that $\sigma_1 = \sigma_2$ and~$\ker(\lambda) = X_2$, hence~$X_1 \subseteq X_2$.
There remains to prove that $\End_{k[\Iw]}(M_\sigma/X, M_\sigma/X)$ is one-dimensional. 
To see this, notice that for all $\Iw$-linear $\lambda: M_\sigma/X \to M_\sigma/X$ there exists~$x \in \lbar k^\times$ such that~$\lambda-x|_{\soc_\Iw(M_\sigma/X)} = 0$.
If~$\lambda - x \ne 0$ then $\lambda - x$ factors through an isomorphism
\[
M_\sigma/\ker(\lambda-x) \to M_\sigma/X
\]
which we are assuming not to exist, since~$X$ is strictly contained in~$\ker(\lambda-x)$.}

To conclude the proof of the theorem, let
\[
\lambda : M_{\sigma_1}/X_1 \to M_{\sigma_2}/X_2
\]
be an $\Iw$-linear isomorphism.
By Lemma~\ref{Msigmasamedimension} we deduce that $\dim(\sigma_1) = \dim(\sigma_2)$, which implies that~$\dim(M_{\sigma_1, n}) = \dim(M_{\sigma_2, n})$ for all~$n$.
We have to prove that~$\dim X_1 = \dim X_2$ and~$\sigma_1 = \sigma_2$ (or equivalently, $\chi_1 = \chi_2$).
Assume this is false. 
After possibly replacing~$\lambda$ by~$\lambda^{-1}$, we can assume that $\dim(X_1) \geq \dim(X_2)$, and so~$\lambda$ factors through a morphism
\[
\lambda_n: M_{\sigma_1, n}/X_1 \to M_{\sigma_2, n}/X_2,
\]
since the dimension of the left-hand side is smaller than the dimension of the right-hand side, and~$M_{\sigma_2}/X_2$ is a uniserial representation (compare the first paragraph of the proof of Theorem~\ref{psIwahori1}).
We are going to prove that the dimension of $\coker(\lambda_n)$ tends to infinity with~$n$, which would contradict the fact that the dimension of $\ker(\lambda_n)$ is independent of~$n$ for~$n$ large enough (since it coincides with~$\ker(\lambda)$ for~$n$ large enough, by uniseriality).
To do so, fix~$n > 0$. 
We are going to prove that $\coker(\lambda_m) \geq p^{n-1}$ if~$m$ is large enough: by the above discussion, this suffices to conclude the proof of the theorem. 
By Lemma~\ref{structureMsigma} we know that $M_{\sigma_i, m}$ is a quotient of $\Ind_{K_0^+(p^{2m+1})}^{K_0(p)}(\chi_i)$. 
Since the dimension of~$M_{\sigma_i ,m}$ tends to infinity with~$m$, for all~$m$ large enough the module~$M_{\sigma_i, m}$ will have an $\Iw$-quotient~$Y_{\sigma_i, m}$ of dimension~$p^{n}$. 
Since $\Ind_{K_0^+(p^{2m+1})}^{K_0(p)}(\chi_i)$ is uniserial, necessarily $Y_{\sigma_i, m } \cong \Ind_{K_0^+(p^{n+1})}^{K_0(p)}(\chi_i)$.
Furthermore, if $\dim M_{\sigma_1, m}/X_1 \geq p^n$, the map $\lambda_m: M_{\sigma_1, m}/X_1 \to M_{\sigma_2, m}/X_2$ induces an $\Iw$-linear map $\lambda_{Y, m}: Y_{\sigma_1, m} \to Y_{\sigma_2, m}$, and $\coker(\lambda_m)$ surjects onto $\coker(\lambda_{Y, m})$.

\textcolor{black}{If $\dim(X_1) > \dim(X_2)$ then $\lambda_m$ cannot be surjective: in fact, in this case its domain has smaller dimension than its codomain, because $\dim(M_{\sigma_1, n}) = \dim(M_{\sigma_2, n})$. 
Hence~$\lambda_{Y, m}$ is not surjective, since if $\lambda_{Y, m}$ is surjective then~$\lambda_m$ is surjective on $\Iw$-cosocles, hence $\lambda_m$ is surjective. 
On the other hand, if~$\chi_1 \ne \chi_2$ then~$\lambda_{Y, m}$ is not a surjection either, since it cannot be an isomorphism when $\chi_1 \ne \chi_2$.
In each of these cases we deduce by Lemma~\ref{cokernel} and Lemma~\ref{twistinduction} that $\dim(\coker \lambda_m) \geq \dim \coker (\lambda_{Y, m}) \geq p^{n-1}$, and the claim follows.}
\end{proof}

\begin{corollary}\label{ssmorphisms}
The space $\Hom_{k[\Iw Z]}(\pi_{\sigma_1}, \pi_{\sigma_2})$ is one-dimensional if $\sigma_1 = \sigma_2$ and vanishes otherwise. Hence $\Hom_{k[G^+]}(\pi_{\sigma_1}, \pi_{\sigma_2}) = \Hom_{k[\Iw Z]}(\pi_{\sigma_1}, \pi_{\sigma_2})$.
\end{corollary}
\begin{proof}
By Proposition~\ref{sstwist}, an $\Iw$-linear map $\pi_{\sigma_1} \to \pi_{\sigma_2}$ has to send $M_{\sigma_1}$ to~$M_{\sigma_2}$ and $\Pi M_{\sigma^{[s]}_1}$ to~$\Pi M_{\sigma^{[s]}_2}$. 
By Theorem~\ref{ssIwahori} (and after twisting by~$\Pi$), these restrictions are zero if~$\sigma_1 \not = \sigma_2$, and otherwise they are scalar endomorphisms. The corollary follows as these have to agree when restricted to~$\soc_{\Iw}(\pi_{\sigma})$.
\end{proof}

\begin{corollary}\label{ssmorphismsN}
Let~$\pi_1, \pi_2$ be generic irreducible supersingular $k[\GL_2(\bQ_p)]$-representations of Serre weight~$\sigma_1, \sigma_2$ respectively. Then $\Hom_{k[G]}(\pi_1, \pi_2) = \Hom_{k[N]}(\pi_1, \pi_2)$.
\end{corollary}
\begin{proof}
Since $\pi_1 \cong \cInd_{G^+}^G(\pi_{\sigma_1})$, we have $\pi_1|_N \cong \cInd_{\Iw Z}^N \pi_{\sigma_1}$.
Now by Frobenius reciprocity we have
\[
\Hom_N(\pi_1, \pi_2) = \Hom_{\Iw Z}(\pi_{\sigma_1}, \pi_{\sigma_2}) \oplus \Hom_{\Iw Z}(\pi_{\sigma_1}, \pi_{\sigma_2^{[s]}}).
\]
Corollary~\ref{ssmorphisms} together with the fact that~$\sigma_2$ and~$\sigma_2^{[s]}$ are not isomorphic implies that the right-hand side is at most one-dimensional, and does not vanish if and only if~$\pi_1$ and~$\pi_2$ are $G$-isomorphic.
\end{proof}

The following lemmas were used in the proof of Theorem~\ref{ssIwahori}.

\begin{lemma}\label{twistinduction}
Let~$G$ be a locally profinite group, $H$ a closed subgroup of~$G$, and $\alpha: G \to G$ a continuous group automorphism. If~$\theta$ is a $k$-representation of~$H$, we have a representation $\theta_\alpha = (\alpha^{-1})^*(\theta)$ of $\alpha(H)$. Then there is a $G$-linear isomorphism
\[
\Ind_{H}^G(\theta)_\alpha \to \Ind_{\alpha H}^G(\theta_\alpha).
\]
\end{lemma}
\begin{proof}
We can assume that~$\theta_\alpha$ has the same representation space of~$\theta$ with the action $\theta_\alpha(\alpha(h)) t =\theta(h) t$. 
Given a smooth function $f: G \to \theta$ with $f(hg) = \theta(h) f(g)$, let $f_\alpha(g) = f \left ( \alpha^{-1}(g) \right )$. Then $f_\alpha(\alpha(h) g) = f(h \alpha^{-1}(g)) = \theta(h) f_\alpha(g) = \theta_\alpha(\alpha(h))f_\alpha(g)$, hence $f_\alpha \in \Ind_{\alpha H}^G(\theta_\alpha)$. The map we are looking for is $f \mapsto f_\alpha$.
\end{proof}

\begin{lemma}\label{Msigmasamedimension}
\textcolor{black}{Let $\sigma_1, \sigma_2$ be generic Serre weights, with $\sigma_i^{\Iw_1} = \chi_i$, and let $X_i \subset M_{\sigma_i}$ be proper (hence finite-dimensional) $\Iw$-stable subspaces.
If there exists a nonzero $\Iw$-linear morphism $\lambda: M_{\sigma_1}/X_1 \to M_{\sigma_2}/X_2$ then $\dim(\sigma_1) = \dim(\sigma_2)$, and so $\dim(M_{\sigma_1, n}) = \dim(M_{\sigma_2, n})$ for all~$n$.}
\end{lemma}
\begin{proof}
\textcolor{black}{It suffices to prove the same result over~$\lbar k$.
Let
\[
\lambda: M_{\sigma_1}/X_1 \to M_{\sigma_2} / X_2
\]
be a nonzero $\Iw$-linear morphism.
We will prove that~$\dim(\sigma_1) = \dim(\sigma_2)$ by a numerical argument based on the computation of $\dim M_{\sigma_i, n}$ from~\cite[Proposition~4.7]{Paskunasextensions}.
More precisely, if we define the integers~$0< r_i < p-1, 0 \leq a_i \leq p-2$ by the formula
\[
\sigma_i \cong \Sym^{r_i}(k^2) \otimes \operatorname{det}^{a_i}
\] 
then
\begin{equation}\label{dimensionMsigma}
d_{i, n} = \dim M_{\sigma_i, n} - 1 = r_i + (p-1-r_i)p + \cdots + r_i p^{2n -2} + (p-1-r_i)p^{2n-1}
\end{equation}
and we need to prove that~$r_1 = r_2$.
We will do so by assuming~$r_1 \ne r_2$ and deriving a contradiction from the fact that the $j$-th $p$-adic digits of~$d_{1, n}$ and~$d_{2, n}$ are different for all values of~$j$.}

\textcolor{black}{Replacing~$X_1$ by~$\ker(\lambda)$, we can assume without loss of generality that~$\lambda$ is injective.
It follows from Lemma~\ref{uniserial} that~$\lambda$ is an isomorphism, and replacing~$\lambda$ by~$\lambda^{-1}$ we can assume without loss of generality that~$r_2<r_1$.
Since~$d_{i, n}$ tends to infinity with~$n$, and~$X_i$ is finite-dimensional, there exists~$n>0$ such that $X_i \subset M_{\sigma_i, n-2}$ for~$i = 1,2$. 
Fix such an integer~$n$.
Then
\[
\lambda(M_{\sigma_1, n}/X_1) \subset M_{\sigma_2, n}/X_2
\]
since the dimension of the left-hand side is smaller than the dimension of the right-hand side and~$M_{\sigma_2}/X_2$ is uniserial (in fact, $\dim M_{\sigma_2, n} - \dim X_2 > (p-1-r_2)p^{2n-1} > \dim M_{\sigma_1, n}$).
Since~$\lambda$ is injective, the dimension of the left-hand side is equal to
\begin{equation}\label{padicexpansion1}
(\dim M_{\sigma_1, n-2} - \dim X_1) + r_1 p^{2n -4} + (p-1-r_1)p^{2n-3} + r_1 p^{2n -2} + (p-1-r_1)p^{2n-1},
\end{equation}
and since~$M_{\sigma_2, n}/X_2$ is uniserial we find that $\lambda(M_{\sigma_1, n}/X_1)$ contains every $\Iw$-submodule of $M_{\sigma_2, n}/X_2$ whose dimension is smaller than $\dim M_{\sigma_1, n}/X_1$.
It follows that~$\lambda(M_{\sigma_1, n}/X_1)$ contains the unique $\Iw$-submodule $Z_{\sigma_2, n} \subset M_{\sigma_2, n}/X_2$ of dimension
\[
(\dim M_{\sigma_2, n-2} - \dim X_2) + r_2 p^{2n -4} + (p-1-r_2)p^{2n-3} + r_2 p^{2n -2}.
\]
We are going to conclude by studying the dimension of the image of $\lambda(M_{\sigma_1, n}/X_1)$ in~$M_{\sigma_2, n}/Z_{\sigma_2, n}$.
Since $M_{\sigma_2, n}/Z_{\sigma_2, n}$ has dimension $(p-1-r_2)p^{2n-1}$, it follows from Lemma~\ref{structureMsigma} and~Lemma~\ref{psfiltration} that there is an $\Iw$-stable filtration
\[
0 = \Fil^{p-1-r_2} \subset \cdots \subset \Fil^1 \subset \Fil^0 = M_{\sigma_2, n}/Z_{\sigma_2, n}
\]
such that $\Fil^i/\Fil^{i+1} \cong \pi_{2n}(\chi_2\alpha^{-i})^+$.}

\textcolor{black}{Let~$b$ be the smallest integer such that $\Fil^{b} \subseteq \lambda(M_{\sigma_1, n}/X_1)/Z_{\sigma_2, n}$, so that~$1 \leq b \leq p-1-r_2$.
Since $M_{\sigma_2, n}/Z_{\sigma_2, n}$ is uniserial, it follows that $\lambda(M_{\sigma_1, n}/X_1)/Z_{\sigma_2, n}$ is strictly contained in~$\Fil^{b-1}$, and so the induced map
\[
\lambda(M_{\sigma_1, n}/X_1)/Z_{\sigma_2, n} \to \Fil^{b-1}/\Fil^{b}
\]
is not surjective.
By Lemma~\ref{structureMsigma} and Lemma~\ref{cokernel} it follows that there exist integers $0 \leq c \leq p-1, 0 \leq d \leq 2n-2,$ such that
\[
\dim \lambda(M_{\sigma_1, n}/X_1)/Z_{\sigma_2, n} = (p-1-r_2-b)p^{2n-1} + c p^d.
\]
In all, we find that the dimension of $\lambda(M_{\sigma_1, n}/X_1)$ equals
\begin{equation}\label{padicexpansion2}
(\dim M_{\sigma_2, n-2} - \dim X_2) + r_2 p^{2n -4} + (p-1-r_2)p^{2n-3} + r_2 p^{2n -2} + (p-1-r_2-b)p^{2n-1} + c p^d
\end{equation}
and so this number has to equal~\eqref{padicexpansion1}.
Since $0 < r_i < p-1$ and $\dim M_{\sigma_i, n-2} < p^{2n-4}$ , adding~$c p^d$ can modify at most two of the coefficients of~$p^{2n-4}, p^{2n-3}, p^{2n-2}$ in~\eqref{padicexpansion2}.
Hence~\eqref{padicexpansion1} and~\eqref{padicexpansion2} have different $p$-adic expansions, and so cannot be equal.
This concludes the proof that~$r_1 = r_2$.}
\end{proof}

\subsubsection{Extensions.} Now we study the restriction to the Iwahori subgroup of extensions of $G$-representations and $G^+$-representations.

\begin{thm}\label{ssextensions}
Let $\sigma_1, \sigma_2$ be generic Serre weights, and let $\pi_i = \left ( \cInd_{KZ}^G \sigma_i \right )/T$. Then the restriction map $\Ext^1_{k[G]}(\pi_1, \pi_2) \to \Ext^1_{k[\Iw Z]}(\pi_1, \pi_2)$ is injective.
\end{thm}
\begin{proof}
Notice that any element in the kernel of this map is contained in $\Ext^1_{k[G], \zeta}(\pi_1, \pi_2)$, i.e. it has central character~$\zeta$.
Hence, by~\cite[Lemma~5.7]{Paskunasimage}, it suffices to prove the theorem over~$\lbar k$.
By~\cite[Theorem~1.1]{Paskunasextensions}, the theorem is true if~$\sigma_1$ and~$\sigma_2$ are not conjugate, since then the representations~$\pi_i$ are not $G$-isomorphic and the space of $G$-extensions vanishes. 
So it suffices to prove that if $\sigma$ is a generic Serre weight, $\pi = \left ( \cInd_{KZ}^G \sigma \right )/T$, and 
\begin{equation}\label{tosplit}
0 \to \pi \to X \to \pi \to 0
\end{equation}
is an exact sequence of $G$-representations with central character~$\zeta$ that splits over~$\Iw$, then the sequence is split over~$G$.

For this, choose $v_\sigma$ in the $K$-socle of the quotient, invariant under~$\Iw_1$ and such that $\langle K \cdot v_\sigma \rangle \cong \sigma$. Then~$\langle K \cdot \Pi v_\sigma \rangle \cong \sigma^{[s]}$, because $\pi^{\Iw_1}$ is two-dimensional. Let $w_\sigma$ be an $\Iw_1$-invariant lift of~$v_\sigma$ to~$X$ with $\Iw$-eigencharacter~$\chi$, which exists since we assume that the extension is $\Iw$-split.

We are going to study the representations $\langle K \cdot w_\sigma \rangle$ and $\langle K \cdot \Pi w_\sigma \rangle$, which are quotients of finite principal series $\Ind_{\Iw}^K(\chi)$ and $\Ind_{\Iw}^K(\chi^s)$, respectively. 
Recall (see for instance~\cite[Section~2]{BPmodp}) that the representation $\Ind_{\Iw}^K(\chi)$ has a two-dimensional space of $\Iw_1$-invariants, spanned by the function~$\varphi \in \Ind_{\Iw}^K(\chi)$ supported in~$\Iw$ and satisfying $\varphi(1) = 1$, and the function $f_0 = S_0 \varphi$, where we have written $S_0 = \sum_{\lambda \in \bF_p}\fourmatrix {[\lambda]} 1 1 0$. The functions~$\varphi$ and~$f_0$ are $H$-eigenvectors, with eigencharacter $\chi$ and~$\chi^s$ respectively. 
We have an exact sequence
\[
0 \to \Ind_{\Iw}^K(\chi)_0 \to \Ind_{\Iw}^K(\chi) \xrightarrow{\res} \chi \to 0
\]
defined by restricting functions to~$\Iw$, which is $\Iw$-linearly split.
Looking at the $H$-eigencharacter, it follows that~$f_0 \in \Ind_{\Iw}^K(\chi)_0$ and generates its $\Iw$-socle.

\begin{pp}\label{liftgeneratesweight}
The representations~$\langle K \cdot w_\sigma \rangle$ and~$\langle K \cdot \Pi w_\sigma \rangle$ are irreducible.
\end{pp}
\begin{proof}
Assume that 
\[
\alpha: \Ind_{\Iw}^K(\chi) \to \langle K \cdot w_\sigma \rangle, \varphi \mapsto w_\sigma
\]
is injective.
Since~$\langle K \cdot v_\sigma \rangle$ is irreducible, we have $S_0 \varphi \in \pi \subset X$.
Since~$S_0 \varphi$ is an $\Iw$-eigenvector with eigencharacter~$\chi^s$, we deduce furthermore that~$S_0 \varphi \in \pi_{\sigma^{[s]}}$.

By assumption, there is an $\Iw$-linear retraction $r: X \to \pi$ of the inclusion of~$\pi$ in~$X$.
Consider the composition
\[
\Ind_{\Iw}^K(\chi)_0 \xrightarrow{\alpha} X \xrightarrow{r} \pi \to \pi_{\sigma^{[s]}}.
\]
Since~$S_0 \varphi$ generates the $\Iw$-socle of~$\Ind_{\Iw}^K(\chi)_0$, this composition is injective.
But the dimension of~$\Ind_{\Iw}^K(\chi)_0$ is equal to~$p$. 
Since the first congruence subgroup~$K_1$ acts trivially on $\Ind_{\Iw}^K(\chi)_0$, this contradicts the fact that $\dim( \pi_{\sigma^{[s]}}^{K_1}) = p-1$ (see~\cite[Section~20]{BPmodp} or~\cite[Theorem~1.4]{Morrainvariants}).

The same proof works for~$\Ind_{\Iw}^K(\chi^s) \to \langle K \cdot \Pi w_{\sigma} \rangle$, or argue by symmetry, using that
\[
\left ( \cInd_{KZ}^G \sigma \right )/T \cong \left ( \cInd_{KZ}^G \sigma^{[s]} \right )/T
\]
to conclude.
\end{proof}

It follows from Proposition~\ref{liftgeneratesweight} that~$\langle K \cdot w_\sigma \rangle$ is $K$-isomorphic to~$\sigma$.
To complete the proof of the theorem  it suffices to prove that~$T w_\sigma = 0$, because then~$\langle G \cdot w_\sigma \rangle \cong \pi$ maps isomorphically to the quotient~$\pi$ and defines a $G$-splitting of the exact sequence~(\ref{tosplit}).
Recall that we have the equality
\[
T w_\sigma = \sum_{\lambda \in \bF_p} \fourmatrix 1 {[\lambda]} 0 1 t w_\sigma = \sum_{\lambda \in \bF_p}\fourmatrix{[\lambda]}{1}{1}{0}\Pi w_\sigma = S_0 \Pi w_\sigma.
\]
By Proposition~\ref{liftgeneratesweight} the surjection $\Ind_{\Iw}^K(\chi^s) \to \langle K \cdot \Pi w_\sigma \rangle$ sending~$\varphi$ to~$\Pi w_\sigma$ is equal to zero on the $K$-socle. 
Hence the image of~$S_0 \varphi$ under this map is zero, because~$S_0 \varphi$ generates the $K$-socle of~$\Ind_{\Iw}^K(\chi^s)$. 
But this image is $S_0 \Pi w_\sigma = Tw_\sigma$, hence~$T w_\sigma = 0$.

\end{proof}

\begin{corollary}\label{ssextensions2}
Let $\sigma_1, \sigma_2$ be generic Serre weights. Then the restriction map 
\[
\Ext^1_{k[G^+]}(\pi_{\sigma_2}, \pi_{\sigma_1}) \to \Ext^1_{k[\Iw Z]}(\pi_{\sigma_2}, \pi_{\sigma_1}) 
\]
is injective.
\end{corollary}
\begin{proof}
Let $0 \to \pi_{\sigma_1} \to X \to \pi_{\sigma_2} \to 0$ be a short exact sequence of $G^+$-representations that is split over~$\Iw Z$. The exact sequence
\[
0 \to \pi_{\sigma_1} \oplus \pi_{\sigma_1}^+ \to X \oplus X^+ \to \pi_{\sigma_2} \oplus \pi_{\sigma_2}^+ \to 0
\]
is $\Iw Z$-split, because $(-)^+ = \ad(\Pi)^*(-)$ and~$\Pi$ normalizes~$\Iw Z$. It is isomorphic to the restriction to~$G^+$ of 
\[
0 \to \Ind_{G^+}^G(\pi_{\sigma_1}) \to \Ind_{G^+}^G(X) \to \Ind_{G^+}^G(\pi_{\sigma_2}) \to 0.
\]
Since $\Ind_{G^+}^G(\pi_{\sigma_i}) \cong \left ( \cInd_{KZ}^G \sigma_i \right )/T$, this sequence is $G$-split by Theorem~\ref{ssextensions}. But then the inclusion $\pi_{\sigma_1} \to X \oplus X^+$ has a $G^+$-linear retraction, which we can restrict to~$X$ to prove that the original exact sequence was already $G^+$-split.
\end{proof}

\subsection{\textcolor{black}{Restriction to~$\Iw Z$ and~$KZ$ of atomes automorphes.}}\label{atomes}
Since we will refer to the work of Breuil and Morra on the structure of atomes automorphes of length two, we will follow some of the notation of~\cite{BreuilrepsII, Morraatomes}.
For instance, the element $[1, x] \in \cInd_{KZ}^G(\sigma)$ is the function supported on~$KZ$ and sending the identity to~$x \in \sigma$.
\begin{defn}
Let $r \in \{1, \ldots, p-4\}$ and $\lambda \in k^\times$.
In this section we let~$\mA_{r, \lambda}$ denote a representation of~$\GL_2(\bQ_p)$ such that there exists a nonsplit exact sequence
\[
0 \to \Ind_B^G (\nr_{\lambda^{-1}} \otimes \nr_\lambda \omega^r) \to \mA_{r, \lambda} \to \Ind_B^G(\nr_\lambda \omega^{r+1} \otimes \nr_{\lambda^{-1}} \omega^{-1}) \to 0.
\]
\end{defn}

\begin{rk}\label{whyp-3}
The papers~\cite{Morraatomes} also studies the representations~$\mA_{r, \lambda}$ when~$r = p-3$, in which case they are not generic according to our conventions.
Since some of our results in Section~\ref{principalseriesrestriction} are for generic principal series, we exclude these cases from consideration.
When~$r = p-2$ there are many atomes automorphes of length~$2$, whose factors are generic principal series, but they are not covered by the results in~\cite{Morraatomes}.
This is the reason why we introduced ``very generic" representations in Section~\ref{defngeneric}.
\end{rk}

We write $\pi_1 = \Ind_B^G (\nr_{\lambda^{-1}} \otimes \nr_\lambda \omega^r)$ and $\pi_2 = \Ind_B^G(\nr_\lambda \omega^{r+1} \otimes \nr_{\lambda^{-1}} \omega^{-1})$, so that~$\pi_1$ has Serre weight~$\Sym^r(k^2)$ and $\pi_2$ has Serre weight~$\Sym^{p-3-r}(k^2) \otimes \det^{r+1}$. 
We fix an exact sequence of $G$-representations with central character~$\zeta$
\begin{equation}\label{atome}
0 \to \pi_1 \to \mA_{r, \lambda} \to \pi_2 \to 0.
\end{equation}
There exists a unique isomorphism class of $\GL_2(\bQ_p)$-representations~$\mA_{r, \lambda}$ for which a nonsplit exact sequence of the form~\eqref{atome} exists (see~\cite[Theorem~11.5]{Paskunasextensions} for a reference).
We write~$\kappa_i$ for the $\Iw$-character conjugate to~$\soc_K(\pi_i)^{\Iw_1}$, so that $\pi_i|_{\Iw} \cong \pi_\infty^+(\kappa_i) \oplus \pi_{\infty}(\kappa_i)$ by~\eqref{Iwahorisplit}, and the element~$\Pi$ switches the two summands in this direct sum decomposition (see the proof of Corollary~\ref{cor:psmorphisms}).
More explicitly, we have $\kappa_1 = d^r$ and $\kappa_2 = a^{r+1}d^{-1}$.
We will also write~$\mA_{r, s, \lambda}$ for~$\mA_{r, \lambda} \otimes (\omega^s \circ \det)$.

\subsubsection{Restriction of~$\mA_{r, s, \lambda}$ to~$\Iw Z$.}\label{Iwahoriatomes}
Let~$\mB_{r, s, \lambda}$ be the preimage of~$\pi_\infty(\kappa_2)$ in~$\mA_{r, s, \lambda}$, which is an $\Iw$-submodule of~$\mA_{r, s, \lambda}$.
We define
\[
\mC_{r, s, \lambda} = \mB_{r, s, \lambda}/\pi_{\infty}(\kappa_1) \text{ and } \mD_{r, s, \lambda} = \mB_{r, s, \lambda}/\pi_{\infty}(\kappa_1)^+.
\]
Our main result in this section is the following theorem.
\begin{thm}\label{finitesplittingIw}
The sequence 
\begin{equation}\label{almostsplit}
0 \to \pi_\infty(\kappa_1)^+ \to \mC_{r, s, \lambda} \to \pi_\infty(\kappa_2) \to 0
\end{equation}
is not split.
For any finite-dimensional $\Iw$-submodule $X \subset \pi_\infty(\kappa_1)$, the sequence
\begin{equation}\label{staysnonsplit}
0 \to \pi_\infty(\kappa_1)/X \to \mD_{r, s, \lambda}/X \to \pi_\infty(\kappa_2) \to 0
\end{equation}
is not split.
\end{thm}
By twisting it suffices to prove this theorem when~$s = 0$, which we assume for the rest of this section.
As a first step towards the proof of Theorem~\ref{finitesplittingIw} we will follow \cite[Section~4]{Morraatomes} and~\cite{BreuilrepsII} in giving an explicit presentation for~$\mA_{r, \lambda}$.
So we introduce the notation~$\sigma_{j}$ for~$\Sym^{j}(k^2)$. 
When~$j>p-1$ this is a reducible representation of~$KZ$.
By~\cite[Lemme~5.1.3]{BreuilrepsII} there is a nonsplit exact sequence
\[
0 \to (\sigma_r \otimes \det) \oplus \sigma_{r+2} \to \sigma_{p+1+r} \to \sigma_{p-3-r} \otimes \operatorname{det}^{r+2} \to 0
\]
such that the inclusion $\sigma_r \otimes \det \to \sigma_{p+1+r}$ sends~$x^{r-i}y^i$ to~$X^{p+r-i}Y^{i+1} - X^{r+1-i}Y^{p+i}$.
(Notice that the condition $p+3 \leq k \leq 2p$ in the reference translates to $0 \leq r \leq p-3$, hence the lemma applies to our choice of~$r$.)
It gives rise to a short exact sequence
\begin{equation}\label{weightsequence}
0 \to \cInd_{KZ}^G(\sigma_r \otimes \det) \xrightarrow{\iota} \left ( \cInd_{KZ}^G \sigma_{p+1+r} \right )/T \xrightarrow{\pr} \cInd_{KZ}^G(\sigma_{p-3-r} \otimes \mathrm{det}^{r+2}) \to 0
\end{equation}
where $\iota[1, x^{r-j}y^j] = \frac{r+2}{j+1}[1, X^{p+r-j}Y^{j+1}]$ and $\pr[1, X^{p-1}Y^{r+2}] = [1, x^{p-3-r}]$. 
Furthermore, $[1, X^{p+1+r}] \in \image(T)$, and so it vanishes in the quotient.
See \cite[Lemme~4.5]{Morraatomes} for this computation.

\begin{lemma}\label{onedim}
The space $\Hom_{KZ}(\sigma_{p+1+r} \otimes \det^{-1}, \mA_{r, \lambda})$ is at most one-dimensional.
\end{lemma}
\begin{proof}
By~\cite[Section~20]{BPmodp} the $K$-socle of~$\mA_{r, \lambda}$ is isomorphic to $\sigma_r$. Hence every nonzero $KZ$-morphism $\sigma_{p+1+r} \otimes \det^{-1} \to \mA_{r, \lambda}$ has to be zero on the factor $\sigma_{r+2} \otimes \det^{-1}$ and cannot be zero on~$\sigma_r$, since it cannot factor through the cosocle. This implies the claim since any two such morphisms are linearly dependent when restricted to~$\sigma_r$ (since~$\soc_K \mA_{r, \lambda}$ is irreducible).
\end{proof}
By Lemma~\ref{onedim} and the results in \cite[(17)]{Morraatomes} or~\cite[Section~5.3]{BreuilrepsII}, the space in the statement of Lemma~\ref{onedim} is actually one-dimensional. 
Furthermore, every nonzero $KZ$-linear map $\sigma_{p+1+r} \otimes \det^{-1} \to \mA_{r, \lambda}$ factors through an injection $\sigma_{p+1+r} \otimes \det^{-1}/\sigma_{r+2} \otimes \det^{-1} \to \mA_{r, \lambda}$ and extends uniquely to a commutative diagram
\begin{equation}\label{presentation}
\begin{tikzcd}
0 \arrow{r} & \cInd_{KZ}^G(\sigma_r) \arrow{r}{\iota} \arrow{d} & \left ( \cInd_{KZ}^G \sigma_{p+1+r} \otimes \operatorname{det}^{-1} \right )/T \arrow{r}{\pr} \arrow{d} & \cInd_{KZ}^G(\sigma_{p-3-r} \otimes \mathrm{det}^{r+1}) \arrow{r} \arrow{d} & 0\\
0 \arrow{r} & \pi_1 \arrow{r} & \mA_{r, \lambda} \arrow{r} & \pi_2 \arrow{r} & 0
\end{tikzcd}
\end{equation}
with all vertical arrows surjective.
In the rest of this section we fix such a $KZ$-linear map and we work with the corresponding commutative diagram, which gives rise to special elements $[1, X^{p+1+r-i}Y^i] \in \mA_{r, \lambda}$ for all~$1 \leq i \leq p+1+r$, as well as $x_1 = [1, x^r] \in \pi_1$ and $x_2 = [1, x^{p-3-r}] \in \pi_2$, such that~$x_i$ generates~$\soc_K(\pi_i)^{\Iw_1}$.
As in~\cite[Section~4.2]{Morraatomes} we write for any~$n \geq 0$
\begin{equation}\label{specialelementI}
e_{n+1} = \tau^n \Pi [1, X^{p-1}Y^{r+2}] = \fourmatrix 0 1 {p^{n+1}} 0 [1, X^{p-1}Y^{r+2}] \in \mA_{r, \lambda},
\end{equation}
which is a lift of $\tau^n \Pi [1, x^{p-3-r}] \in \pi_2$. 
Since~$X^{p-1}Y^{r+2} \in \sigma_{p+1+r} \otimes \det^{-1}$ is an $(HK_1, a^{-1}d^{r+1})$-eigenvector, a direct computation using the fact that~$\kappa_2 = a^{r+1}d^{-1}$ implies that $e_{n+1}$ is a $(K_0(p^{n+2}), \kappa_2)$-eigenvector.
We will be studying the cocycle $(u_{n+1}-1)e_{n+1}$, where
\[
u_{n+1} = \fourmatrix 1 0 {p^{n+1}} 1.
\]
This cocycle is computed in~\cite[Lemme~5.1]{Morraatomes}, and we will need a slightly different perspective on this computation, so we recall it in full.
We have
\begin{align*}
u_{n+1}e_{n+1} &= \fourmatrix 1 0 {p^{n+1}} 1 \fourmatrix 0 1 {p^{n+1}} 0 [1, X^{p-1}Y^{r+2}] = \fourmatrix 0 1 {p^{n+1}} 0 \fourmatrix 1 1 0 1 [1, X^{p-1}Y^{r+2}]\\
&= \fourmatrix 0 1 {p^{n+1}} 0 [1, X^{p-1}(X+Y)^{r+2}] = \fourmatrix 0 1 {p^{n+1}} 0 \sum_{j=0}^{r+2} \binom {r+2}{j} [1, X^{p+r+1-j}Y^j].
\end{align*}
By~\eqref{specialelementI} and the vanishing of~$[1, X^{p+1+r}]$ in~$\mA_{r, \lambda}$, we deduce that
\[
u_{n+1}e_{n+1} = e_{n+1} + \fourmatrix 0 1 {p^{n+1}} 0 \sum_{j=1}^{r+1} \binom {r+2}{j} [1, X^{p+r+1-j}Y^{j}] = e_{n+1} + \fourmatrix 0 1 {p^{n+1}} 0 \sum_{i=0}^{r} \binom {r+2}{i+1} [1, X^{p+r-i}Y^{i+1}]
\]
and finally
\begin{equation}\label{cocyclen+1}
(u_{n+1}-1)e_{n+1} = \fourmatrix 0 1 {p^{n+1}} 0 \sum_{i=0}^{r} \binom {r+2}{i+1} \frac{i+1}{r+2} [1, x^{r-i}y^i] = \fourmatrix 0 1 {p^{n+1}} 0 \sum_{i=0}^{r} \binom {r+1}{i}[1, x^{r-i}y^i].
\end{equation}
The following definition will be useful to compare various cocycles in~$\mA_{r, \lambda}$.
We will use it to give a quantitative refinement of certain results from~\cite[Section~5]{Morraatomes}, although we work with $\lbar U(p^{n+1}\bZ_p)$-cocycles rather than~$B(\bZ_p)$-cocycles.
Recall from Proposition~\ref{Morrathesis1} that~$\pi_\infty(\kappa_1)$ is~$\lbar U(p\bZ_p)$-uniserial and $\Iw$-uniserial, and so
\begin{equation}\label{samesoclefiltration}
\soc^i_\Iw \pi_\infty(\kappa_1) = \soc^i_{\lbar U(p\bZ_p)} \pi_\infty(\kappa_1) 
\end{equation}
for all~$i$.
This property implies that a subspace of~$\pi_\infty(\kappa_1)$ is $\Iw$-stable if and only if it is $\lbar U(p\bZ_p)$-stable, and in fact the only such subspaces are displayed in~\eqref{samesoclefiltration}.

\begin{defn}
For nonzero $x \in \pi_\infty(\kappa_1)$ we define~$\deg(x)-1$ to be the only integer~$d$ such that $x \in \soc_\Iw^d(\pi_\infty(\kappa_1))$ but~$x \not \in \soc_\Iw^{d-1}(\pi_\infty(\kappa_1))$.
We define~$\deg(0) = 0$.
Equivalently, for all~$x \in \pi_\infty(\kappa_1)$ we have
\[
\deg(x) = \dim_{k} \langle \Iw \cdot x \rangle = \dim_{k} \langle \lbar U(p\bZ_p) \cdot x \rangle.
\]
\end{defn}

\begin{lemma}\label{degreelinearindependence}
Assume that~$S = \{v_0, \ldots, v_n\} \subset \pi_\infty(\kappa_1)$, that $i \ne j$ implies~$\deg(v_i)\ne\deg(v_j)$, and that~$v_j \ne 0$ for all~$j$.
Then~$S$ is linearly independent.
\end{lemma}
\begin{proof}
Without loss of generality, $i < j$ implies $0<\deg(v_i)<\deg(v_j)$.
Assume $\sum_{i=0}^n\mu_jv_j = 0$.
Since~$\sum_{i=0}^{n-1}\mu_j v_j \in \soc^{\deg(v_{n-1})-1}_\Iw\pi_\infty(\kappa_1)$, if~$\mu_n \ne 0$ then $v_n \in \soc^{\deg(v_{n-1})-1}_\Iw\pi_\infty(\kappa_1)$, which contradicts $\deg(v_n) > \deg(v_{n-1})$.
So~$\mu_n = 0$.
Repeating this argument we conclude the proof since~$v_0 \ne 0$.
\end{proof}

\begin{lemma}\label{cocycledegree}
Assume that~$x \in \pi_\infty(\kappa_1)$ is not fixed by~$u_{n+1}$.
Then~$\deg(x) > p^n$ and
\[
\deg((u_{n+1}-1)x) = \deg(x)-p^{n}.
\]
\end{lemma}
\begin{proof}
By Proposition~\ref{Morrathesis1} the representation~$\pi_\infty(\kappa_1)$ is an injective envelope of the trivial representation of~$\lbar U(p\bZ_p)$, hence the same is true for~$\pi_\infty(\kappa_1)/\soc^i_\Iw \pi_\infty(\kappa_1)$ for all~$i$: see~\cite[Proposition~5.9]{Paskunasextensions}, or use the fact that the Iwasawa algebra of~$\lbar U(p\bZ_p)$ is a discrete valuation ring.
It follows that the space of $u_{n+1}$-fixed vectors in~$\pi_\infty(\kappa_1)/\soc^i_\Iw \pi_\infty(\kappa_1)$ has the same dimension as $\pi_\infty(\kappa_1)^{U(p^{n+1}\bZ_p)}$, which is $p^n$-dimensional. 
Since~$\pi_\infty(\kappa_1)$ is $\lbar U(p\bZ_p)$-uniserial, it follows that if~$x$ is not fixed by~$u_{n+1}$ then $\deg(x) = \dim_k\langle \lbar U(p\bZ_p) \cdot x \rangle > p^n$. 
This concludes the proof of the first claim.

For the second claim, by definition we have $(u_{n+1}-1)x \in \soc^i_\Iw \pi_\infty(\kappa_1)$ if and only if $i \geq \deg((u_{n+1}-1)x)-1$.
We will show this is equivalent to $i \geq \deg(x)-p^n-1$.
To see this, notice that $(u_{n+1}-1)x \in \soc_\Iw^i \pi_\infty(\kappa_1)$ if and only if the image~$\lbar x$ of~$x$ in~$\pi_\infty(\kappa_1)/\soc^i_\Iw \pi_\infty(\kappa_1)$ is $u_{n+1}$-fixed.
This holds if and only if~$\lbar x$ of~$x$ generates a $\lbar U(p\bZ_p)$-submodule of $\pi_\infty(\kappa_1)/\soc^i_\Iw \pi_\infty(\kappa_1)$ of dimension at most~$p^n$.
We claim that 
\[
\langle \lbar U(p\bZ_p) \cdot \lbar x \rangle \leq p^n \text{ if and only if } \deg(x)-i-1 \leq p^n.
\]
To see this, notice first that if~$i \geq \deg(x)$ then~$x \in \soc_\Iw^i\pi_\infty(\kappa_1)$, and so $\lbar x = 0$ and the equivalence is true.
On the other hand, if~$i < \deg(x)$ then~$\soc_\Iw^i\pi_\infty(\kappa_1) \subseteq \langle \Iw \cdot x \rangle$, and so~$\lbar x$ generates an $\Iw$-submodule of $\pi_\infty(\kappa_1)/\soc^i_\Iw \pi_\infty(\kappa_1)$ of dimension~$\deg(x)-i-1$.

In summary, we have shown that $i \geq \deg(x)-p^n-1$ if and only if $i \geq \deg((u_{n+1}-1)x)-1$, which implies that~$\deg((u_{n+1}-1)x) = \deg(x)-p^n$.
\end{proof}
In what follows, recall that we write~$\tau$ for the diagonal matrix~$\diag(1, p)$.
Furthermore, if $0 \leq i < r$ we introduce the notation
\begin{equation}\label{notationpsiI}
\psi_i = \Pi[1, x^{r-i}y^i] \in \pi_\infty(\kappa_1).
\end{equation}
By Lemma~\ref{diagonalKsocle} this is a $(K_0(p^2), \kappa_1\alpha^i)$-eigenvector, and it is a $K_0(p)$-eigenvector if and only if $i = 0$.
Finally, by Lemma~\ref{diagonalKsocle} again there exist nonzero a $\psi^+ \in \soc_\Iw \pi_\infty(\kappa_1)^+$ and a nonzero $(K_0(p^2), \kappa_1\alpha^r)$-eigenvector $\psi_r \in \pi_\infty(\kappa_1)$ such that
\begin{equation}\label{notationpsiII}
\Pi[1, y^r] = \psi^+ + \psi_r.
\end{equation}
We will apply the following three lemmas to~$v_i = \psi_i$ and~$v^+ = \psi^+$.

\begin{lemma}\label{degreezero}
Let~$v_0 \in \soc_\Iw \pi_\infty(\kappa_1)$ be a generator.
Then~$\tau^nv_0$ is a~$(K_0(p^{n+1}), \kappa_1)$-eigenvector in~$\pi_\infty(\kappa_1)$, but not a $(K_0(p^{n}), \kappa_1)$-eigenvector, and so
\[
\deg(\tau^n v_0) = p^n.
\]
\end{lemma}
\begin{proof}
The fact that $\tau^nv_0$ is a~$(K_0(p^{n+1}), \kappa_1)$-eigenvector follows from the identity
\begin{equation}\label{taunIwahori}
\tau^n K_0(p^j) \tau^{-n} = \fourmatrix{\bZ_p^\times}{p^{-n}\bZ_p}{p^{j+n}\bZ_p}{\bZ_p^\times}.
\end{equation}
specialized at~$j = 1$. 
By Proposition~\ref{Morrathesis1}, the representation $\pi_\infty^+(\kappa_1)|_{U(\bZ_p)}$ is $U(\bZ_p)$-uniserial, hence its space of~$U(\bZ_p)$-invariants has dimension~$1$ and coincides with~$\soc_\Iw \pi_\infty(\kappa_1)^+ = \kappa_1^s$.
Hence there are no $(K_0(p^{n+1}), \kappa_1)$-eigenvectors in~$\pi_\infty(\kappa_1)^+$, and so~$\tau^n v_0 \in \pi_\infty(\kappa_1)$. 
To see that~$\tau^n v_0$ is not a $K_0(p^n)$-eigenvector, notice that if~$u_{n}\tau^nv_0 = \tau^nv_0$ then~$\fourmatrix 1 0 1 1v_0 = v_0$.
Since~$v_0$ is also an $\Iw$-eigenvector, the Bruhat decomposition together with the identity
\[
\fourmatrix 1 0 1 1 \fourmatrix 1 {-1} 0 1 \fourmatrix 1 0 1 1 = \fourmatrix 0 {-1} 1 0
\]
implies that~$v_0$ is a $K$-eigenvector, which is a contradiction.
The fact that~$\deg(\tau^n v_0) = p^n$ now follows from Corollary~\ref{specificeigenvector}.
\end{proof}

\begin{lemma}\label{degreei}
For all~$1\leq i \leq p-1$, let~$v_i \in \pi_\infty(\kappa_1)\setminus \soc_\Iw \pi_\infty(\kappa_1)$ be a nonzero $(K_0(p^{2}), \kappa_1\alpha^i)$-eigenvector.
Fix~$n \geq 0$.
Then
\begin{enumerate}
\item $\tau^nv_i$ is a $(K_0(p^{n+2}), \kappa_1\alpha^i)$-eigenvector in~$\pi_\infty(\kappa_1)$, 
\item there exist $\nu_0, \nu_1, \ldots, \nu_{i-1} \in k_E$, independent of~$n$ and not all zero, such that
\[
(u_{n+1}-1)\tau^nv_i = \sum_{j=0}^{i-1} \nu_j\tau^nv_j,
\]
\item the vector $\tau^n v_i$ is not a $K_0(p^{n+1})$-eigenvector, and so
\begin{equation}\label{degreeeigenvector}
\deg(\tau^n v_i) = p^n\deg(v_i) = (i+1)p^n.
\end{equation}
\end{enumerate}
Part~(ii) and~\eqref{degreeeigenvector} are true with~$\tau^n v_i$ replaced by any $(K_0(p^{n+2}), \kappa_1\alpha^i)$-eigenvector in~$\pi_\infty(\kappa_1)$ that is not a $K_0(p^{n+1})$-eigenvector.
\end{lemma}
\begin{proof}
The fact that $\tau^nv_i$ is a $(K_0(p^{n+2}), \kappa_1\alpha^i)$-eigenvector follows from~\eqref{taunIwahori} for~$j = 2$.
Since $\pi_\infty(\kappa_1)^{\lbar U(p^2 \bZ_p)}$ is $\lbar U(p\bZ_p)$-uniserial and generated by~$\tau v_0$, the same argument as in the proof of Corollary~\ref{comparetauaction} implies that there exists~$h_n \in k[\lbar U(p\bZ_p)]$ such that~$\tau^n v_i = h_n \tau^{n+1}v_0$, hence~$\tau^n v_i \in \pi_\infty(\kappa_1)$ by Lemma~\ref{degreezero}. 
This concludes the proof of the first claim.
For the second claim, since~$u_{n+1}\tau^n = \tau^nu_1$ it suffices to prove that there exist~$\nu_0, \ldots, \nu_{i-1} \in k_E$, not all zero, such that
\[
(u_1-1)v_i = \sum_{j=1}^{i-1}\nu_jv_j.
\]
This follows from the fact that~$v_i$ is not~$\Iw_1$-fixed, together with the fact that $\langle \Iw \cdot v_i \rangle$ is spanned by~$\{v_0, v_1, \ldots, v_i\}$, which is a consequence of Lemma~\ref{psfiltration}.
Now~\eqref{degreeeigenvector} follows from Corollary~\ref{specificeigenvector}.
Finally, if~$\eta$ is a $(K_0(p^{n+2}), \kappa_1\alpha^i)$-eigenvector in~$\pi_\infty(\kappa_1)$, but not a $K_0(p^{n+1})$-eigenvector, then it follows from Lemma~\ref{eigenvectorsII} that there exists $x \in k_E^\times$ such that~$\eta - x\tau^n v_i$ is a $K_0(p^{n+1})$-eigenvector.
Hence~$\deg(\eta) = \deg(x\tau^nv_i)$, and
\[
(u_{n+1}-1)\eta = (u_{n+1}-1)(x\tau^n v_i),
\]
which concludes the proof (replacing~$\nu_j$ by~$x\nu_j$).
\end{proof}

\begin{lemma}\label{Heckeexpansion}
Let~$n \geq 1$ and let~$v^+ \in \soc_\Iw \pi_\infty(\kappa_1)^+$ be a generator.
Then there exists a nonzero $\nu \in k^\times$ such that
\[
\tau^n v^+ = \left ( \sum_{j=0}^{n-1}\lambda^{-(n-1-j)} \tau^j(\nu\psi_r) \right ) + \lambda^{-n} v^+.
\]
\end{lemma}
\begin{proof}
Follows by induction on~$n$, with the base case contained in Lemma~\ref{basecase}.
\end{proof}

\begin{lemma}\label{finalcocyclecomputation}
There exist a nonzero scalar $\nu \in k_E^\times$ such that 
\begin{equation}\label{cocyclen+1III}.
(u_{n+1}-1)e_{n+1} = \sum_{i=0}^r\binom{r+1}{i}\tau^n\psi_i + (r+1)\left ( \sum_{j=0}^{n-1}\lambda^{-(n-1-j)}\tau^j(\nu\psi_r) + \lambda^{-n} \psi^+\right)
\end{equation}
for all~$n \geq 0$.
\end{lemma}
\begin{proof}
This is a direct consequence of Lemma~\ref{Heckeexpansion} and the computation in~\eqref{cocyclen+1}, together with the definitions~\eqref{notationpsiI} and~\eqref{notationpsiII}.
\end{proof}

\begin{proof}[Proof of Theorem~~\ref{finitesplittingIw}(1)]
Assume for a contradiction that~\eqref{almostsplit} is split.
Recall that the image of $e_{n+1}$ in $\pi_2$ is $\tau^n \Pi[1, x^{p-3-r}]$, which is contained in $\pi_\infty(\chi_2)$ by Lemma~\ref{degreezero}.
Hence $e_{n+1} \in \mB_{r, \lambda}$.
Let~$\lbar e_{n+1}^+$ be the image of~$e_{n+1}$ in~$\mC_{r, \lambda}$, which maps to a $(K_0(p^{n+1}), \kappa_2)$-eigenvector in~$\pi_\infty(\kappa_2)$.
Then for all~$n \geq 0$ there exists a $(K_0(p^{n+1}), \kappa_2)$-eigenvector $\eta_{n+1} \in \mC_{r, \lambda}$ with the same image in~$\pi_\infty(\kappa_2)$ as~$\lbar e_{n+1}^+$.
It follows that~$\lbar e_{n+1}^+-\eta_{n+1} \in \pi_\infty(\kappa_1)^+$ is fixed by $\fourmatrix 1 1 0 1$.
However, the representation~$\pi_{\infty}(\kappa_1)^+$ is $U(\bZ_p)$-uniserial, since~$\pi_\infty(\kappa_1)$ is $\lbar U(p \bZ_p)$-uniserial and $\pi_\infty(\kappa_1)^+ \cong \ad(\Pi)^* \pi_\infty(\kappa_1)$.
Hence the space of~$U(\bZ_p)$-fixed vectors in~$\pi_\infty(\kappa_1)^+$ is one-dimensional and coincides with $\soc_\Iw \pi_\infty(\kappa_2)^+$.
Hence~$\lbar e_{n+1}^+-\eta_{n+1}$ is actually fixed by~$K_1(p)$, and so~$\lbar e_{n+1}^+$ is $u_{n+1}$-fixed, because both~$\lbar e_{n+1}^+-\eta_{n+1}$ and~$\eta_{n+1}$ are $u_{n+1}$-fixed.
This contradicts~\eqref{cocyclen+1III}, which implies that $(u_{n+1}-1)\lbar e_{n+1}^+ = (r+1)\lambda^{-n}\psi^+ \ne 0$ (recall that~$\psi^+$ is a generator of~$\soc_\Iw\pi_\infty(\kappa_1)^+$ and that~$\psi_i \in \pi_\infty(\kappa_1)$ for all~$0\leq i \leq r$ by Lemma~\ref{degreezero} and Lemma~\ref{degreei}).
\end{proof}

\begin{lemma}\label{computedegreecocycle}
Let~$\lbar e_{n+1}$ be the image of~$e_{n+1}$ in~$\mD_{r, \lambda}$. 
Then
\[
\deg((u_{n+1}-1)\lbar e_{n+1})  = (r+1)p^n < p^{n+1}-p^n.
\]
\end{lemma}
\begin{proof}
By~\eqref{cocyclen+1III} we know that there exists~$\nu \in k^\times$ such that
\begin{equation}\label{cocycleprojection}
(u_{n+1}-1)\lbar e_{n+1} = \sum_{i=0}^r\binom{r+1}{i}\tau^n\psi_i + (r+1)\sum_{j=0}^{n-1}\lambda^{-(n-1-j)}\tau^j(\nu\psi_r).
\end{equation}
By Lemma~\ref{degreezero} and Lemma~\ref{degreei} we have
\[
\deg(\tau^j\psi_i) = (i+1)p^j \text{ for all } 0 \leq i \leq r, j \geq 0.
\]
Hence the summands in~\eqref{cocycleprojection} have pairwise distinct degrees, whose maximum is~$(r+1)p^n$, and we conclude by Lemma~\ref{degreelinearindependence}, since all the summands in~\eqref{cocycleprojection} are nonzero.
\end{proof}

\begin{pp}\label{comparecocycles}
Let~$\lbar e_{n+1} \in \mD_{n, \lambda}$ be the image of~$e_{n+1}$, where~$\mD_{n, \lambda}$ is as in~\eqref{staysnonsplit}.
For every finite-dimensional $\Iw$-submodule $X \subset \pi_\infty(\kappa_1)$ there exists~$N > 0$ such that $n \geq N$ implies that
\[
(u_{n+1}-1)(\lbar e_{n+1}) - (u_{n+1}-1)(\psi) \not \in X
\]
for every $(B(\bZ_p), \kappa_1\alpha^{r+1})$-eigenvector $\psi \in \pi_\infty(\kappa_1)$.
(Recall that~$\kappa_2 = \kappa_1\alpha^{r+1}$.)
\end{pp}
\begin{proof}
Choose~$N$ such that~$(r+1)p^{N-1} > \dim(X)$.
Fix~$n \geq N$, let~$\psi \in \pi_\infty(\kappa_1)$ be a $(B(\bZ_p), \kappa_1\alpha^{r+1})$-eigenvector, and write
\[
\Delta = (u_{n+1}-1)(\lbar e_{n+1}) - (u_{n+1}-1)(\psi).
\]
Assume first that~$\psi$ is not fixed by~$u_{n+2}$.
Then $\deg(\psi) > p^{n+1}$, and so $\deg((u_{n+1}-1)\psi) > p^{n+1}-p^n$ by Lemma~\ref{cocycledegree}.
On the other hand, Lemma~\ref{computedegreecocycle} implies that $\deg((u_{n+1}-1)\lbar e_{n+1}) < p^{n+1}-p^n$.
Hence
\begin{equation}\label{firstbound}
\deg(\Delta) > p^{n+1}-p^n.
\end{equation}
Assume now that~$\psi$ is fixed by~$u_{n+1}$, so that $\Delta = (u_{n+1}-1)(\lbar e_{n+1})$.
Then again by Lemma~\ref{computedegreecocycle} we have
\begin{equation}\label{secondbound}
\deg(\Delta) = \deg((u_{n+1}-1)\lbar e_{n+1}) = (r+1)p^n.
\end{equation}
Finally, assume that~$\psi$ is fixed by~$u_{n+2}$ but not by~$u_{n+1}$, so that~$\psi$ is a $K_0(p^{n+2})$-eigenvector but not a $K_0(p^{n+1})$-eigenvector.
We know from Lemma~\ref{finalcocyclecomputation} that there exists~$\nu \in k^\times$ such that
\[
(u_{n+1}-1)\lbar e_{n+1} = \sum_{i=0}^r\binom{r+1}{i}\tau^n\psi_i + (r+1)\sum_{j=0}^{n-1}\lambda^{-(n-1-j)}\tau^j(\nu\psi_r).
\]
for all~$n \geq 0$.
On the other hand, by Lemma~\ref{degreei} there exist scalars~$\nu_0, \nu_1, \ldots, \nu_r \in k_E$ such that
\[
(u_{n+1}-1)\psi = \sum_{j=0}^r \nu_j \tau^n \psi_j.
\]
Putting these together, we find that
\[
\Delta = \sum_{i=0}^r\left ( \binom{r+1}{i}-\nu_i \right )\tau^n\psi_i + (r+1)\sum_{j=0}^{n-1}\lambda^{-(n-1-j)}\tau^j(\nu\psi_r).
\]
Since the nonzero terms of this sum have pairwise distinct degrees, and all the terms in the second sum are nonzero, it follows that
\begin{equation}\label{thirdbound}
\deg(\Delta) \geq (r+1)p^{n-1}.
\end{equation}
It now follows from~\eqref{firstbound}, \eqref{secondbound} and~\eqref{thirdbound} that~$\deg(\Delta) > \dim(X)$, and so~$\Delta \not \in X$.
\end{proof}

\begin{proof}[Proof of Theorem~\ref{finitesplittingIw}(2)]
Assume that~\eqref{staysnonsplit} splits.
Let~$N > 0$ be as in Proposition~\ref{comparecocycles}.
Then for every~$n \geq N$ there exists a $(K_0(p^{n+1}), \kappa_2)$-eigenvector $\eta_{n+1} \in \mD_{r, \lambda}/X$ with the same image in~$\pi_\infty(\kappa_2)$ as~$\lbar e_{n+1}$.
Let~$e^*_{n+1}$ be the image of $\lbar e_{n+1}$ in $\mD_{r, \lambda}/X$.
It follows from the above that~$e^*_{n+1}-\eta_{n+1}$ is a $(K_0(p^{n+2}), \kappa_2)$-eigenvector in~$\pi_\infty(\kappa_1)/X$.
Since $B(\bZ_p)$ is a compact $p$-adic Lie group, $\Ext^1_{B(\bZ_p)}(\kappa_2, X)$ is finite-dimensional.
Applying the functor $\Hom_{B(\bZ_p)}(\kappa_2, -)$ to the short exact sequence
\[
0 \to X \to \pi_\infty(\kappa_1) \to \pi_\infty(\kappa_1)/X \to 0,
\] 
we find that there exist~$d>0$ and scalars $\mu_1, \ldots, \mu_d \in k$, not all zero, such that
\[
\sum_{i=1}^{d}\mu_i(e^*_{N+i}-\eta_{N+i})
\]
equals the image in~$\pi_\infty(\kappa_1)/X$ of a $(B(\bZ_p), \kappa_2)$-eigenvector $\psi \in \pi_\infty(\kappa_1)$.
Replacing~$d$ by a smaller positive number we can assume~$\mu_d \ne 0$.
By construction, for all~$i$ we know that $e^*_{N+i}-\eta_{N+i}$ is fixed by~$u_{N+i+1}$, hence it is fixed by~$u_{N+d}$ whenever~$d > i$, and so
\[
(u_{N+d}-1)\sum_{i=1}^{d}\mu_i(e^*_{N+i}-\eta_{N+i}) = (u_{N+d}-1)\mu_d(e^*_{N+d}-\eta_{N+d}) = \mu_d(u_{N+d}-1)(e^*_{N+d}).
\]
where the last equality holds since~$\eta_{N+d}$ is a $K_0(p^{N+d})$-eigenvector, and so is fixed by~$u_{N+d}$.
So we deduce that
\[
\mu_d(u_{N+d}-1)(e^*_{N+d}) = (u_{N+d}-1)(\psi) \in \pi_{\infty}(\kappa_1)/X,
\] 
or equivalently that
\[
(u_{N+d}-1)(\lbar e_{N+d}) - (u_{N+d}-1)(\mu_d^{-1}\psi) \in X,
\]
which contradicts Proposition~\ref{comparecocycles}.
\end{proof}

\subsubsection{Morphisms.}
In this section we will relate the $G$-action and the $K$-action on~$\mA_{r, s, \lambda}$. 
The following example shows that one cannot expect a direct analogue of Theorem~\ref{ssmorphisms}.

\begin{example}
Let $\mu, \lambda \in k^\times, \mu \ne \lambda^{\pm 1}$. Since
\[
\cInd_{KZ}^G(\Sym^r k^2)/(T-\lambda)|_K \cong \cInd_{KZ}^G(\Sym^r k^2)/(T-\mu)|_K,
\]
there exists a nonzero $KZ$-linear morphism
\[
\mA_{p-3-r, \mu} \otimes (\omega^{r+1} \circ \det) \to \mA_{r,\lambda}
\]
but there are no nonzero $G$-morphisms between these representations.
\end{example}

However, we will be able to establish an analogue of Theorem~\ref{ssmorphisms} once we restrict to isomorphisms (Theorem~\ref{deformations}), or to endomorphisms of a single~$\mA_{r, s, \lambda}$.
We begin with the case of endomorphisms.

\begin{pp}\label{atomesmorphisms}
The spaces $\Hom_{KZ}(\mA_{r, s, \lambda}, \mA_{r, s, \lambda})$ and $\Hom_{N}(\mA_{r, s, \lambda}, \mA_{r, s, \lambda})$ are both one-dimensional, hence coincide with $\Hom_G(\mA_{r, s, \lambda}, \mA_{r, s, \lambda})$.
\end{pp}
\begin{proof}
It suffices to consider the case~$s = 0$.
There is an exact sequence
\[
0 \to \Hom_{KZ}(\pi_2, \mA_{r, \lambda}) \to \Hom_{KZ}(\mA_{r, \lambda}, \mA_{r, \lambda}) \to \Hom_{KZ}(\pi_1, \mA_{r, \lambda}).
\]
Given a nonzero element $\pi_2 \to \mA_{r, \lambda}$ of the first term, composing it with the projection to~$\pi_2$ yields a nonzero scalar because of Corollary~\ref{cor:psmorphisms}. 
Hence we have constructed a $KZ$-splitting of the exact sequence~(\ref{atome}), hence an $\Iw$-splitting, contradicting (for example) the computation of the $\Iw_1$-invariants of~$\mA_{r, \lambda}$ in~\cite[Section~20]{BPmodp}. 
It follows that the first term vanishes. 
For the last term, given $\alpha: \pi_1 \to \mA_{r, \lambda}$, the composition $\pi_1 \xrightarrow{\alpha} \mA_{r, \lambda} \to \pi_2$ is zero by Corollary~\ref{cor:psmorphisms}. 
Hence~$\alpha$ factors through the subspace~$\pi_1$ in~(\ref{atome}). 
By Corollary~\ref{cor:psmorphisms} it follows that the term~$\Hom_{KZ}(\pi_1, \mA_{r, \lambda})$ is one-dimensional, and so $\Hom_{KZ}(\mA_{r, \lambda}, \mA_{r, \lambda})$ is also one-dimensional.

The proof goes through unchanged for the group~$N$.
\end{proof}
Next we consider the case of morphisms between atomes automorphes of the same Serre weight.

\begin{thm}\label{sameparameter}
Let ${\lambda_1}, {\lambda_2} \in k^\times$ and assume $\Hom_{K}(\mA_{r, s, {\lambda_1}}, \mA_{r, s, {\lambda_2}}) \not = 0$. 
Then ${\lambda_1} = \pm {\lambda_2}$, hence $\mA_{r, s, {\lambda_1}} \cong \mA_{r, s, {\lambda_2}} \otimes (\nr_{\pm 1} \circ \det)$.
\end{thm}
\begin{proof}
Since~$\mA_{r, s, {\lambda_i}}$ is a twist of~$\mA_{r, {\lambda_i}}$, it suffices to prove the theorem when~$s = 0$.
Let $\alpha: \mA_{r, {\lambda_1}} \to \mA_{r, {\lambda_2}}$ be a nonzero $K$-linear morphism. 
We know that~$\alpha$ induces a commutative diagram
\[
\begin{tikzcd}
0 \arrow{r} & \pi(r, {\lambda_1}, 1) \arrow{r} \arrow{d}{\alpha_1} & \mA_{r, {\lambda_1}} \arrow{r}\arrow{d}{\alpha} & \pi(p-3-r, {\lambda_1}^{-1}, \omega^{r+1}) \arrow{r}\arrow{d}{\alpha_2} & 0\\
0 \arrow{r} & \pi(r, {\lambda_2}, 1) \arrow{r} & \mA_{r, {\lambda_2}}\arrow{r} & \pi(p-3-r, {\lambda_2}^{-1}, \omega^{r+1}) \arrow{r} & 0
\end{tikzcd}
\]
because there are no nonzero $\Iw$-linear maps between generic principal series with nonisomorphic $K$-socle, by Corollary~\ref{cor:psmorphisms}, and $\Sym^r \not \cong \Sym^{p-3-r} \otimes \det^{r+1}$ (since $r \not = p-2$, the~$\det$ factor is not trivial). 
We claim that~$\alpha$ is an isomorphism. 
For this it suffices to prove that~$\alpha_1$ and~$\alpha_2$ are not zero, since by Corollary~\ref{cor:psmorphisms} they are then isomorphisms.
If they are both zero then~$\alpha$ factors through~$\pi(p-3-r, {\lambda_1}^{-1}, \omega^{r+1}) $ to give a map $\pi(p-3-r, {\lambda_1}^{-1}, \omega^{r+1}) \to \pi(r, {\lambda_2}, 1) $, which is zero by Corollary~\ref{cor:psmorphisms}. 
If~$\alpha_1 = 0$ but~$\alpha_2$ is not, then Corollary~\ref{cor:psmorphisms} implies that $\alpha_2$ is an isomorphism, and then $\alpha\alpha_2^{-1}$ gives a $K$-section of the projection $\mA_{r, {\lambda_2}} \to \pi(p-3-r, {\lambda_2}^{-1}, \omega^{r+1})$. 
Similarly, if $\alpha_2 = 0$ but~$\alpha_1$ is not, then~$\alpha_1^{-1}\alpha$ gives a $K$-retraction of the inclusion $\pi(r, {\lambda_1}, 1) \to \mA_{r, {\lambda_1}}$. These would contradict the fact that~$\soc_K(\mA_{r, {\lambda_i}})$ is irreducible.

Now choose nonzero $KZ$-linear morphisms $\iota_{\lambda_i}:\sigma_{p+1+r} \otimes \mathrm{det}^{-1}\to\mA_{r, {\lambda_i}}$ and fix~$n \geq 0$.
These give rise to special elements in~$\mA_{r, {\lambda_i}}$ as in Section~\ref{Iwahoriatomes}, and we will write~$e_{n+1, {\lambda_i}}$ for the elements defined in~\eqref{specialelementI}.
By Lemma~\ref{onedim} up to replacing~$\alpha$ by a scalar multiple of~$\alpha$ we can assume that~$\alpha \iota_{\lambda_1} = \iota_{\lambda_2}$, and so $\alpha_2[1, x^{p-3-r}] = [1, x^{p-3-r}]$.
Hence by Proposition~\ref{exchange} we have
\[
\alpha_2 \tau^n \Pi [1, x^{p-3-r}] = {\lambda_2}^{n+1}{\lambda_1}^{-(n+1)} \tau^n \Pi [1, x^{p-3-r}].
\]
It follows that
\[
\eta = e_{n+1, {\lambda_2}}-{\lambda_2}^{-(n+1)}{\lambda_1}^{n+1}\alpha({e}_{n+1, {\lambda_1}}) \in \pi(r, {\lambda_2}, 1) 
\]
and~$\eta$ is a $(K_0(p^{n+2}), \kappa_2)$-eigenvector.
We know from~\eqref{cocyclen+1III} that for all~$1 \leq j \leq r$ there exist nonzero $(K_0(p^2), \kappa_1\alpha^j)$-eigenvectors~$\psi_{j, {\lambda_i}} \in \pi_\infty(\kappa_1) \subset \pi(r, {\lambda_i}, 1)$ as well as generators $\psi_{0, \lambda_i} \in \soc_\Iw\pi_\infty(\kappa_1)$, $\psi^+_{\lambda_i} \in \soc_\Iw\pi_\infty(\kappa_1)^+$ such that
\[
(u_{n+1}-1)(e_{n+1, \lambda_i}) = \sum_{j=0}^r\binom{r+1}{j}\tau^n\psi_{j, {\lambda_i}} + (r+1)\left ( \sum_{k=0}^{n-1}{\lambda_i}^{-(n-1-k)}\tau^k(\nu_i\psi_{r, {\lambda_i}}) + {\lambda}_i^{-n} \psi_{\lambda_i}^+\right)
\]
for some nonzero scalars~$\nu_i \in k_E^\times$.
On the other hand, we know by Lemma~\ref{psfiltration} that there exist nonzero scalars~$\mu_j, \mu^+ \in k^\times$ such that
\[
\alpha(\psi_{j, \lambda_1}) = \mu_j\psi_{j, \lambda_2} \text{ and } \alpha(\psi^+_{\lambda_1}) = \mu^+ \psi^+_{\lambda_2}.
\]
It follows from Corollary~\ref{comparetauaction} that 
\[
\alpha(\tau^k\psi_{j, \lambda_1}) = \mu_j\lambda_2^{-k}\lambda_1^{k}\tau^k\psi_{j, \lambda_2} \text{ for all } k \geq 0.
\]
Hence $(u_{n+1}-1)(e_{n+1, {\lambda_2}}-\lambda_2^{-(n+1)}\lambda_1^{n+1}\alpha(e_{n+1, \lambda_1}))$ is congruent to
\begin{equation}\label{cocyclecomparisonmorphisms}
\sum_{j=0}^r\binom{r+1}{j}(1-\lambda_2^{-2n-1}\lambda_1^{2n+1}\mu_j)\tau^n\psi_{j, {\lambda_2}} + (r+1)\sum_{k=0}^{n-1}({\lambda_2}^{-(n-1-k)}\nu_2-\lambda_2^{-(n+1+k)}\lambda_1^{n+1+k-(n-1-k)}\mu_r\nu_1)\tau^k(\psi_{r, {\lambda_2}})
\end{equation}
modulo~$\pi_\infty^+(\kappa_1)$.
Let~$\lbar \eta$ be the projection of~$\eta \in \pi(r, \lambda_2, 1)$ onto~$\pi_\infty(\kappa_1)$.
Then we have just shown that for all~$n\geq 0$ there exists a $(K_0(p^{n+2}), \kappa_2)$-eigenvector $\lbar \eta \in \pi_\infty(\kappa_1)$ such that~$(u_{n+1}-1)\lbar \eta$ equals~\eqref{cocyclecomparisonmorphisms}.
By Lemma~\ref{degreei} this implies
\begin{equation}\label{vanishingcoefficients}
{\lambda_2}^{-(n-1-k)}\nu_2-\lambda_2^{-(n+1+k)}\lambda_1^{n+1+k-(n-1-k)}\mu_r\nu_1 = 0 \text{ for all } 0\leq k \leq n-1.
\end{equation}
To see this, recall from Lemma~\ref{degreezero} and~Lemma~\ref{degreei} that the vectors $\tau^k \psi_{j, {\lambda_2}}$ have pairwise different degrees for $1 \leq j \leq r, 0 \leq k \leq n$, and so they are linearly independent by Lemma~\ref{degreelinearindependence}.
Now we have two cases: either~$\lbar \eta$ is $u_{n+1}$-fixed, in which case~\eqref{cocyclecomparisonmorphisms} is zero, or it is not, in which case $(u_{n+1}-1)\lbar \eta$ is a linear combination of~$\{\tau^n \psi_{0, \lambda_2}, \ldots, \tau^n \psi_{r, \lambda_2}\}$ by Lemma~\ref{degreei}.
In both cases we deduce the vanishing of~\eqref{vanishingcoefficients}.
Now it follows from~\eqref{vanishingcoefficients} that $(\lambda_1^{-1}\lambda_2)^{2(k+1)}$ is independent of~$1 \leq k \leq n-1$.
Letting~$n$ tend to infinity this implies that~$\lambda_1^2 = \lambda_2^2$ since $(\lambda_1^{-1}\lambda_2)^2 \in k^\times$ has finite multiplicative order.
\end{proof}

\subsubsection{An exact sequence.}
In preparation for Section~\ref{atomesextensions}, we present some general material concerning resolutions of smooth $\GL_2(\bQ_p)$-representations. 
We will apply these results to study self-extensions of~$\mA_{r, \lambda}$.
Write $\delta: N \to k^\times$ for the orientation character, which is trivial on~$\Iw Z$ and takes value~$-1$ at~$\fourmatrix 0 1 p 0$. 
We identify the representation space of~$\delta$ with~$k$.
Then we have a complex of $k[G]$-representations
\begin{equation}\label{tree}
0 \to \cInd_N^G(\delta) \xrightarrow{\partial} \cInd_{KZ}^G(\triv) \xrightarrow{\mathrm{sum}} \triv \to 0
\end{equation}
defined as follows. 
In either induction, we write~$[g, 1]$ for the function supported on $KZg^{-1}$, respectively $Ng^{-1}$, and taking value~$1$ at~$g^{-1}$. These functions span~$\cInd_{KZ}^G(\triv)$, respectively $\cInd_N^G(\delta)$, over~$k$.
We define the maps~$\mathrm{sum}$ and~$\partial$ by
\[
\mathrm{sum}[x, 1] = 1, \; \partial[g, 1] = [g, 1] - [g\Pi, 1].
\]
Since $[gn, 1] = \delta(n)[g, 1]$ if~$n \in N$, these are well-defined.
Since (in either compact induction) we have $h[g, 1] = [hg, 1]$ for all $g, h \in G$, these are $G$-linear maps.

\begin{pp}\label{exacttree}
The complex~(\ref{tree}) is exact.
\end{pp}
\begin{proof}
This is a standard result that holds in much greater generality, but we provide a proof for completeness. We write~$X$ for the Bruhat--Tits tree of~$\PGL_2(\bQ_p)$, whose vertices are in bijection with $G/KZ$, and whose edges are in bijection with~$G/N$.
We write~$[x] = xKZ$ for~$x \in G$. A \emph{$0$-chain} on~$X$ is a function
\[
g: G/KZ \to k
\]
with finite support.
An \emph{oriented $1$-chain} on~$X$ is a function~$f$ with finite support on the set of oriented edges of~$X$, such that $f([x], [y]) = -f([y], [x])$, whenever $\{[x], [y] \}$ is an edge of~$X$.

We identify~$\cInd_{KZ}^G(\triv)$ with the space of $0$-chains on~$X$ by letting~$[x, 1]$ be the chain supported in $[x]$ with value~$1$ at~$[x]$. 
We identify~$\cInd_N^G(\delta)$ with the space of oriented $1$-chains on~$X$ by letting~$[g, 1]$ be the chain
\[
([g], [g\Pi]) \mapsto 1, ([g\Pi], [g]) \mapsto -1, \text{ and zero elsewhere}.
\]
This is well-defined because $[gn, 1] = \delta(n)[g, 1]$.
With these identifications, the complex~(\ref{tree}) corresponds to
\[
0 \to C^{1}_{\mathrm{or}}(X, k) \xrightarrow{\partial} C^0(X, k) \xrightarrow{\mathrm{sum}} k \to 0
\]
with $\partial(f)[x] = \sum_{\text{edges } \{[x], [y]\} }f([x], [y])$. 
This complex computes the simplicial homology of~$X$ with coefficients in~$k$, and~$X$ is contractible, hence it is an exact sequence.
\end{proof}

\begin{lemma}\label{compacttensor}
Let~$G$ be a locally profinite group and~$H$ an open subgroup of~$G$. Let~$V, W$ be smooth $k$-representations of~$G$ and~$H$, respectively. Then there is a $G$-linear isomorphism
\[
\left ( \cInd_H^G W \right ) \otimes_{k} V \to \cInd_H^G(W \otimes_{k} V|_H)
\]
functorial in~$V$ and~$W$.
\end{lemma}
\begin{proof}
This is because the two sides (co)represent the same functor on the category of smooth $k[G]$-representations, by Frobenius reciprocity and the adjunction between $\Hom_{k}$ and $\otimes_{k}$.
\end{proof}

Now let~$V, W$ be smooth $k[\GL_2(\bQ_p)]$-representations. Applying $- \otimes_{k} V$ to~(\ref{tree}), we obtain by Lemma~\ref{compacttensor} an exact sequence
\[
0 \to \cInd_N^G (V \otimes_{k} \delta) \to \cInd_{KZ}^G(V) \to V \to 0.
\]
Applying $\Hom_G(-, W)$ and using Frobenius reciprocity for~$\Ext$ (see~\cite[5.9(e)]{Vignerasrepsbook}, which has no $\ell \not = p$ hypotheses), we obtain an exact sequence
\begin{equation}\label{longtree}
0 \to \Hom_G(V, W) \to \Hom_{KZ}(V, W) \to \Hom_N(V\delta, W) \to \Ext^1_G(V, W) \to \Ext^1_{KZ}(V, W) \to \Ext^1_N(V\delta, W)
\end{equation}
in which the arrow $\Hom_G(V, W) \to \Hom_{KZ}(V, W)$ is the identity (sending every $G$-linear map~$f$ to~$f$ itself viewed as a $KZ$-linear map) and the arrow $\Ext^1_G(V, W) \to \Ext^1_{KZ}(V, W)$ is the restriction map.

\subsubsection{Extensions.}\label{atomesextensions}
We are now ready to prove the following theorem concerning self-extensions of~$\mA_{r, s, \lambda}$.
\begin{thm}\label{psextensions}
The restriction map $\Ext^1_G(\mA_{r, s, \lambda}, \mA_{r, s, \lambda}) \to \Ext^1_{KZ}(\mA_{r, s, \lambda}, \mA_{r, s, \lambda})$ is injective.
\end{thm}
\begin{proof}
Any element in the kernel of this map is contained in $\Ext^1_{k[G], \zeta}(\mA_{r, s, \lambda}, \mA_{r, s, \lambda})$, i.e. it has central character~$\zeta$.
Hence, by~\cite[Lemma~5.7]{Paskunasimage}, it suffices to prove the theorem over~$\lbar k$.
By twisting, it suffices to prove the theorem when~$s = 0$.
Combining the exact sequence~(\ref{longtree}) with Proposition~\ref{atomesmorphisms}, the theorem is equivalent to the vanishing of the group $\Hom_N(\mA_{r, \lambda}\delta, \mA_{r, \lambda})$. We proceed by contradiction, assuming that there exists a nonzero $N$-linear morphism $\alpha: \mA_{r, \lambda}\delta \to \mA_{r, \lambda}$.
Equivalently, $\alpha$ is a nonzero $\Iw Z$-linear morphism such that $\alpha \Pi = -\Pi \alpha$, and so $\alpha$ is not $N$-linear but $\alpha^2$ is $N$-linear.
Hence by Proposition~\ref{atomesmorphisms} there exists~$y' \in \lbar k$ such that~$\alpha^2 = y'$, and so~$\alpha^2 = y^2$ for some~$y \in \lbar k$.
Recalling Corollary~\ref{cor:psmorphisms}, we see that since $\Hom_{\Iw}(\pi_1, \pi_2) = 0$ every $\Iw Z$-linear morphism $\mA_{r, \lambda} \to \mA_{r, \lambda}$, such as~$\alpha$, preserves~$\pi_1$ and passes to the quotient to~$\pi_2$.
Since~\eqref{Iwahorisplit} implies that $\End_{\Iw Z}(\pi_i)\cong k \times k$ is a reduced ring, if~$y = 0$ then the induced map $\alpha: \pi_i \to \pi_i$ must be zero, and so $\alpha$ must factor through an $\Iw Z$-morphism $\pi_2 \to \pi_1$, which is necessarily zero.
Hence~$y \ne 0$, and since~$p \ne 2$ this implies that $\mA_{r, \lambda} = \mA_{r, \lambda}^{\alpha = y} \oplus \mA_{r, \lambda}^{\alpha = -y}$ and these two summands are switched by~$\Pi$, hence are both nonzero.
Since~$\End_{\Iw Z}(\pi_i) \cong k \times k$ has precisely four idempotents, up to replacing~$y$ by~$-y$ we can assume that the image of~$\mA_{r, \lambda}^{\alpha = y}$ in~$\pi_2$ is $\pi_\infty(\kappa_2)$. 
Similarly, the intersection~$\mA_{r, \lambda}^{\alpha = y} \cap \pi_1$ is one of~$\pi_\infty(\kappa_1)$ and~$\pi_\infty(\kappa_1)^+$.
It follows that there exists an $\Iw$-stable subspace of~$\mC_{r, \lambda}$ or~$\mD_{r, \lambda}$ that projects isomorphically onto~$\pi_\infty(\kappa_2)$, which contradicts Theorem~\ref{finitesplittingIw}.
\end{proof}

\subsection{Summary.}
We summarize some consequences of the results above.
For simplicity, we are going to write $\Ext^i_{\mO[G]}$ for the $\Ext$-functor computed in the category of smooth representations on $p^n$-torsion $\mO$-modules.

\begin{thm}\label{morphisms}
Let~$\pi_1, \pi_2$ be finite length smooth $\mO[\GL_2(\bQ_p)]$-representations of central character~$\zeta$.
\begin{enumerate}
\item If the Jordan--H\"older factors of the~$\pi_i$ are all generic and supersingular, then $\Hom_{\mO[G^+]}(\pi_1, \pi_2) = \Hom_{\mO[\Iw Z]}(\pi_1, \pi_2)$ and the restriction map $\Ext^1_{\mO[G^+]}(\pi_1, \pi_2) \to \Ext^1_{\mO[\Iw Z]}(\pi_1, \pi_2)$ is injective.
Similarly, $\Hom_{\mO[G]}(\pi_1, \pi_2) = \Hom_{\mO[N]}(\pi_1, \pi_2)$ and the restriction map $\Ext^1_{\mO[G]}(\pi_1, \pi_2) \to \Ext^1_{\mO[N]}(\pi_1, \pi_2)$ is injective.
\item Fix~$r \in \{1, \ldots, p-4 \}, s \in \{0, \ldots, p-2\}$, and $\lambda \in k^\times$. 
If the~$\pi_i$ both admit an exhaustive $G$-stable filtration with all graded factors isomorphic to~$\mA_{r, s, \lambda}$, then $\Hom_{G}(\pi_1, \pi_2) = \Hom_{KZ}(\pi_1, \pi_2)$ and the restriction map $\Ext^1_{\mO[G]}(\pi_1, \pi_2) \to \Ext^1_{\mO[KZ]}(\pi_1, \pi_2)$ is injective.
\end{enumerate}
\end{thm}
\begin{proof}
For part (1), the Jordan--H\"older factors of $\pi_i|_{G^+}$ are all of the form~$\pi_\sigma$ for some generic Serre weight~$\sigma$ (depending on the factor), since we are fixing the central character. Then the first statement follows from Corollaries~\ref{ssmorphisms} and~\ref{ssextensions2} by a d\'evissage argument.
(See for instance~\cite[Lemma~A.1]{Paskunasextensions}. Notice that this argument requires injectivity of the restriction map $\Ext^1_{\mO[G^+]}(\pi_{\sigma_2}, \pi_{\sigma_1}) \to \Ext^1_{\mO[\Iw Z]}(\pi_{\sigma_2}, \pi_{\sigma_1})$ for all generic Serre weights~$\sigma_1, \sigma_2$, but this is immediate from Corollary~\ref{ssextensions2}, since if an extension of~$\pi_{\sigma_2}$ by~$\pi_{\sigma_1}$ splits over~$\Iw Z$ then the maximal ideal of~$\mO$ acts trivially on the extension.)

Similarly, the second statement of part~(i) follows from Corollary~\ref{ssmorphismsN} and Theorem~\ref{ssextensions}, and part~(ii) follows from Proposition~\ref{atomesmorphisms} and Theorem~\ref{psextensions}.
\end{proof}

\begin{corollary}\label{deformations}
Let~$A$ be an Artin local $\mO$-algebra with residue field~$k$ and maximal ideal~$\fm_A$.
Let~$\pi_1, \pi_2$ be smooth $A[\GL_2(\bQ_p)]$-representations on flat $A$-modules.
\benum
\item Assume that $\pi_1 \otimes_A k \cong \pi_2 \otimes_A k$ are supersingular generic irreducible representations.
Then~$\pi_1 |_{A[\Iw Z]} \cong \pi_2 |_{A[\Iw Z]}$ if and only if~$\pi_1 |_{A[G^+]} \cong \pi_2|_{A[G^+]}$, and $\pi_1 |_{A[N]} \cong \pi_2|_{A[N]}$ if and only if $\pi_1 \cong \pi_2$.
\item Assume that there exist~$r_i \in \{1, \ldots, p-4\}, s_i \in \{0, \ldots, p-2\}$ and~$\lambda_i \in k^\times$ such that $\pi_i \otimes_A k \cong \mA_{r_i, s_i, \lambda_i}$. 
If $f: \pi_1 \to \pi_2$ is an $A[KZ]$-linear isomorphism, then either $(r_1, s_1, \lambda_1) = (r_2, s_2, \lambda_2)$ and $f$ is $A[G]$-linear, or $(r_1, s_1, \lambda_1) = (r_2, s_2, -\lambda_2)$ and~$f$ induces an $A[G]$-linear isomorphism $\pi_1 \to (\nr_{-1} \circ \det)\otimes_{A} \pi_2$.
\eenum
\end{corollary}
\begin{proof}
For part~(i) it suffices to notice that by the flatness assumption the representations~$\pi_i$ satisfy the assumptions of part~(i) of Theorem~\ref{morphisms} (to see this, tensor the $\fm$-adic filtration of~$A$ with~$\pi_i$).
For part~(ii), notice the map $f \otimes_A k$ is a $KZ$-linear isomorphism $\mA_{r_1, s_1, \lambda_1} \to \mA_{r_2, s_2, \lambda_2}$. 
Comparing the $K$-socle, it follows that $r_1 = r_2$ and~$s_1 = s_2$. 
But now Theorem~\ref{sameparameter} implies that $\lambda_1^2 = \lambda_2^2$. 
\textcolor{black}{If~$\lambda_1 = \lambda_2$ then~$\pi_1$ and~$\pi_2$ satisfy the assumptions of part~(ii) of Theorem~\ref{morphisms}.
If $\lambda_1 = - \lambda_2$, we can regard $\nr_{-1}$ as valued in~$A^\times$, and then $\pi_1$ and $\pi_2 \otimes_A (\nr_{-1} \circ \det)$ satisfy the assumptions of part~(ii) of Theorem~\ref{morphisms}.
In both cases, the claim follows from Theorem~\ref{morphisms}.}
\end{proof}

\section{Banach space representations.}
In this section we prove the two theorems in the introduction, starting from Theorem~\ref{restriction}.
Let~$\Pi$ be an absolutely irreducible admissible unitary $E$-Banach space representation of~$\GL_2(\bQ_p)$, with fixed central character~$\zeta: \bQ_p^\times \to \mO^\times$. 
After possibly replacing~$E$ by a quadratic extension and twisting by an unramified character, we can assume that $\zeta(p) = 1$, and we will do so throughout this section.
Assume that~$\Pi$ is very generic, as defined in Section~\ref{defngeneric}.
\textcolor{black}{In this section we will often deal with open bounded lattices in~$\Pi$.
To abbreviate, we will refer to them simply as \emph{lattices}.
Since~$\Pi$ will usually denote a Banach space, in this section we will denote the matrix~$\fourmatrix 0 1 p 0$ by~$\nu$.}

\subsection{Proof of Theorem~\ref{restriction}: supersingular reduction.} \label{supersingularreduction}
Assume that~$\Pi$ has a $G$-stable lattice~$\Theta$ whose reduction~$\Theta \otimes_{\mO} k \cong \pi$ is an absolutely irreducible supersingular representation of~$G$. Let~$\Pi^1 \subset \Pi$ be a proper $K$-stable closed $E$-subspace. 
Then $\Theta^1 = \Theta \cap \Pi^1$ is a $KZ$-stable lattice in~$\Pi^1$, and
\[
\Theta^1 \otimes_{\mO} k \to \Theta \otimes_{\mO} k
\]
is injective. Write $\Theta \otimes_{\mO} k = \pi_\sigma \oplus \pi_{\sigma^{[s]}}$, as in Section~\ref{secsupersingular}.

\begin{pp}\label{Kmorphisms}
Let~$\sigma_1$ and~$\sigma_2$ be distinct generic Serre weights. Let $X_j \subset \pi_{\sigma_j}$ be a finite-dimensional $\Iw$-stable subspace. Then there are no nonzero $\Iw$-linear morphisms $\lambda: \pi_{\sigma_1}/X_1 \to \pi_{\sigma_2}/X_2$. 
\end{pp}
\begin{proof}
Without loss of generality, $X_1 = 0$. Since~$X_2$ is finite-dimensional, it is contained in~$\soc^i_\Iw(\pi_{\sigma_2})$ for~$i$ large enough. 
Consider the induced morphism
\[
\lambda_i: \pi_{\sigma_1} \to \pi_{\sigma_2}/X_2 \to \pi_{\sigma_2}/\soc^i_\Iw(\pi_{\sigma_2}),
\]
and recall from~(\ref{Iwsupersingular}) that
\[
\pi_{\sigma_2} / \soc^i(\pi_{\sigma_2}) \cong M_{\sigma_2} / \soc^i(M_{\sigma_2}) \oplus M^+_{\sigma_2^{[s]}} / \soc^i \left (M^+_{\sigma_2^{[s]}} \right ).
\]
Let us compose~$\lambda_i$ with the surjection
\[
M_{\sigma_1} \oplus M_{\sigma_{1}^{[s]}}^+ \to \pi_{\sigma_1},
\]
and apply Proposition~\ref{sstwist} and Theorem~\ref{ssIwahori}. Since $\sigma_1$ and $\sigma_2$ are not isomorphic, neither are $\sigma_1^{[s]}$ and~$\sigma_2^{[s]}$, and we deduce that~$\lambda_i = 0$.

It follows that~$\image\lambda \subseteq \soc^i(\pi_{\sigma_2})/X_2$ is finite-dimensional. 
Now the proposition follows from the fact that~$\pi_\sigma$ has no nonzero finite-dimensional $\Iw$-quotients. 
To see this, observe that Lemma~\ref{uniserial} implies that~$M_{\sigma}$ has no nonzero finite-dimensional $\Iw$-quotients. 
Then the same is true for its twist~$M_{\sigma}^+$, so the image of $M_\sigma$ and $M_{\sigma^{[s]}}^+$ in a finite-dimensional $\Iw$-representation vanishes. But $\pi_\sigma$ is a quotient of $M_\sigma \oplus M^+_{\sigma^{[s]}}$.
\end{proof}

\begin{pp}\label{subspacefactor}
Let~$\sigma$ be a generic Serre weight. Assume that~$X$ is an infinite-dimensional $K$-stable subspace of~$\pi_\sigma$. Then $X = \pi_\sigma$.
\end{pp}
\begin{proof}
It suffices to prove the proposition after extending scalars to~$\lbar k$.
Write~$\chi = \soc_{\Iw}(\pi_{\sigma})$. Then
\[
\pi_\sigma/\chi = M_\sigma/\chi \oplus M_{\sigma^{[s]}}^+/\chi
\]
is the direct sum of two uniserial $\Iw$-representations, and its socle filtration is given by Lemma~\ref{soclesum}. Assume $\soc^i_{\Iw}(\pi_\sigma)$ is contained in $X$ but $\soc^{i+1}_{\Iw}(\pi_\sigma)$ is not. Then $X/ \soc^i(\pi_\sigma)$ intersects trivially one of the summands of $\pi_\sigma / \soc^i \pi_\sigma$, since otherwise it would contain the socle of both summands, hence it would contain $\soc(\pi_\sigma/\soc^i \pi_\sigma )$ and~$X$ would contain $\soc^{i+1}(\pi_\sigma)$. 

If $\left ( X/\soc^i\pi_\sigma \right ) \cap \left ( M^+_{\sigma^{[s]}} / \soc^i M^+_{\sigma^{[s]}} \right ) = 0$, then we obtain an injection of $X/\soc^i(\pi_\sigma)$ in $M_\sigma/\soc^i(M_\sigma)$, through the canonical projection. 
By Lemma~\ref{uniserial}, every proper $\Iw$-submodule of $M_\sigma / \soc^i M_\sigma$ is finite-dimensional, hence this injection is an isomorphism onto $M_\sigma / \soc^i M_\sigma$.
By the same argument as Lemma~\ref{sstwist} there are no nonzero $\Iw$-linear maps $M_\sigma / \soc^i(M_\sigma) \to M^+_{\sigma^{[s]}}/\soc^i M^+_{\sigma^{[s]}}$. 
Hence $X/\soc^i(\pi_\sigma)$ is contained in the first summand $M_\sigma/ \soc^i M_\sigma$, hence it coincides with it. But then $X$ contains~$M_\sigma$. 
The other case similarly implies that $X$ contains~$M^+_{\sigma^{[s]}}$.

Now it suffices to prove that if~$X$ is a $K$-stable subspace of~$\pi_\sigma$ containing $M_\sigma$ or $\nu M_{\sigma^{[s]}}$, then~$X = \pi_\sigma$.
A previous version of this paper showed this by identifying $\pi_\sigma$ with a direct summand of a supersingular $\GL_2(\bQ_p)$-representation, and using the action of $G = \GL_2(\bQ_p)$.
We thank a referee for suggesting the following alternative argument which avoids using the~$\GL_2(\bQ_p)$-action.
By the proof of~\cite[Proposition~4.12]{Paskunasextensions}, we have $s \nu M_{\sigma^{[s]}} \subset M_{\sigma}$.
Hence if~$X$ contains~$M_\sigma$, then it contains~$s\nu M_{\sigma^{[s]}}$, and applying~$s$ we see that it contains~$\nu M_{\sigma^{[s]}}$.
Since~$\pi_\sigma = M_\sigma + \nu M_{\sigma^{[s]}}$, we deduce that~$X = \pi_\sigma$.
For the other case, assume~$X$ contains~$\nu M_{\sigma^{[s]}}$ and is $K$-stable. 
Then~$X$ contains the $k$-vector space~$s\nu M_{\sigma^{[s]}}$, and so it contains $\langle \Iw \cdot s\nu M_{\sigma^{[s]}} \rangle$.
However, as remarked above, $s\nu M_{\sigma^{[s]}} \subset M_{\sigma}$, hence $\langle \Iw \cdot s\nu M_{\sigma^{[s]}} \rangle$ is an infinite-dimensional $k[\Iw]$-submodule of~$M_{\sigma}$.
Since~$M_{\sigma}$ is uniserial, we deduce that $\langle \Iw \cdot s\nu M_{\sigma^{[s]}} \rangle = M_{\sigma}$, hence $M_{\sigma} \subset X$, which concludes the proof.
\end{proof}

\begin{pp}\label{subspace}
Let~$\pi$ be an absolutely irreducible, supersingular $k[G]$-representation. Let~$X$ be an infinite-dimensional $K$-stable subspace of~$\pi = \pi_\sigma \oplus \pi_{\sigma^{[s]}}$. Then~$X$ contains one of the summands~$\pi_\sigma, \pi_{\sigma^{[s]}}$, hence it is a direct sum of subspaces of~$\pi_\sigma, \pi_{\sigma^{[s]}}$.
\end{pp}
\begin{proof}
Assume it does not. 
Then~$X$ intersects the summands in finite-dimensional subspaces $X_\sigma, X_{\sigma^{[s]}}$ by Proposition~\ref{subspacefactor}. 
Then the image of~$X$ in $\pi / \left ( X_\sigma \oplus X_{\sigma^{[s]}} \right )$ does not intersect any of the summands. 
It follows that the projections of $X/\left ( X_\sigma \oplus X_{\sigma^{[s]}} \right )$ to $\pi_\sigma/X_\sigma$ and $\pi_{\sigma^{[s]}}/X_{\sigma^{[s]}}$ are both injective. 
They are surjective by Proposition~\ref{subspacefactor}, since otherwise~$X$ is finite-dimensional. Hence they are isomorphisms, and this contradicts the fact that there are no nonzero $\Iw$-linear maps between $\pi_{\sigma}/X_{\sigma}$ and $\pi_{\sigma^{[s]}}/X_{\sigma^{[s]}}$, by Proposition~\ref{Kmorphisms}.
\end{proof}

Let us now consider the image of $\Theta^1 \otimes_{\mO} k$ in~$\Theta \otimes_{\mO} k$. 
If it is finite-dimensional, then $\Theta^1$ is finitely generated over~$\mO$, by the following version of Nakayama's lemma. 
Hence~$\Pi^1$ is finite-dimensional and we are done.

\begin{lemma}\label{completedNakayama}
Let $f: M_1 \to M_2$ be an $\mO$-linear map between $\pi_E$-adically separated and complete $\mO$-modules. If~$f \otimes_{\mO} k$ is surjective, then~$f$ is surjective.
\end{lemma}
\begin{proof}
The map~$f$ is $\pi_E$-adically continuous since it is $\mO_E$-linear. 
Let~$x \in M_2$. 
By the assumption on~$f \otimes_\mO k$ there exists $x_0 \in M_1$ such that $x - f(x_0) \in \pi_E M_2$. Repeating, we find that there exist~$x_i \in M_1$ such that $x = \sum_{n=0}^{+\infty} \pi_E^i f(x_i)$. The continuity of~$f$ implies that if $y = \sum_{n=1}^{+\infty}\pi_E^i x_i$ then $x = f(y)$. 
\end{proof}

There remains to consider the case that~$\Theta^1 \otimes_{\mO} k$ is infinite-dimensional, in which case we can assume by Proposition~\ref{subspace} that its image in~$\Theta \otimes_{\mO} k$ is equal to $\pi_{\sigma} \oplus X_{\sigma^{[s]}}$ for some finite-dimensional~$X_{\sigma^{[s]}}$ (since otherwise $\Pi^1 = \Pi$ by Proposition~\ref{subspacefactor} and Lemma~\ref{completedNakayama}). We know that the matrix~$\nu$ swaps the two direct summands of~$\pi$. 
\textcolor{black}{Let us introduce $\Pi^2 = \nu \Pi^1$, an $\Iw$-stable closed subspace of~$\Pi$, and define $\Theta^2 = \Theta \cap \Pi^2 = \Theta \cap \nu \Pi^1$. 
Since~$\Theta$ is $G$-stable, we have $\Theta^2 = \nu\Theta^1 = \nu(\Theta \cap \Pi^1)$. }

\begin{lemma}
The space~$\Pi^1 \cap \Pi^2$ is finite-dimensional over~$E$ and $\Iw$-stable.
\end{lemma}
\begin{proof}
The space~$\Pi^1 \cap \Pi^2$ is $\Iw$-stable since so are the~$\Pi^i$. 
Notice that~$\Theta^1 \cap \Theta^2 = \Theta \cap (\Pi^1 \cap \Pi^2)$ is an $\Iw Z$-stable lattice in~$\Pi^1 \cap \Pi^2$.
Since~$\Theta^2 = \nu\Theta^1$, the injection $\Theta^2 \otimes_{\mO} k \to \Theta \otimes_{\mO} k$ has image $\nu X_{\sigma^{[s]}} \oplus \pi_{\sigma^{[s]}}$. 
Since $X_{\sigma^{[s]}}$ is finite-dimensional and the injection $\left ( \Theta^1 \cap \Theta^2 \right ) \otimes_{\mO} k \to \Theta \otimes_{\mO} k$ has image contained in $\nu X_{\sigma^{[s]}} \oplus X_{\sigma^{[s]}}$, the claim follows again from Lemma~\ref{completedNakayama}. 
\end{proof}

Since~$\Pi^1 \cap \Pi^2$ is $\Iw$-stable, we have completed the proof of Theorem~\ref{restriction} in the supersingular case if $\Pi^1 \cap \Pi^2 \not = 0$ (because we can induce $\Pi^1 \cap \Pi^2$ to a finite-dimensional $K$-stable closed $E$-subspace of~$\Pi$). 

\begin{lemma}\label{directsumIwahori}
Assume~$\Pi^1 \cap \Pi^2 = 0$. Then $\Pi|_{\Iw} = \Pi^1 \oplus \Pi^2$.
\end{lemma}
\begin{proof}
By Lemma~\ref{completedNakayama} and our assumptions on~$\Theta^1$ and~$\Theta^2$ we know that the map $\Theta^1 \oplus \Theta^2 \to \Theta$ induced by the inclusion of~$\Theta^i$ in~$\Theta$ is surjective.
Upon inverting~$p$ it follows that $\Pi = \Pi^1 + \Pi^2$ (algebraic sum), and by assumption the sum is direct. 
Finally, by definition of the~$\Theta_i$ we find that $\Pi^1 \oplus \Pi^2 \to \Pi$ is a bijective continuous morphism, hence it is a topological isomorphism.
\end{proof}

The following lemma applied to the orthogonal idempotents defining the decomposition in Lemma~\ref{directsumIwahori} implies that~$\Pi^1, \Pi^2$ are actually~$G^+$-stable.

\begin{lemma}\label{liftingtolattices}
Let~$\Pi_1$ and~$\Pi_2$ be absolutely irreducible, unitary, admissible $E$-Banach space representations admitting $G$-stable lattices~$\Theta_i$ whose reductions are isomorphic to the same supersingular irreducible representation.
Then $\Hom_{G^+}^\cont(\Pi_1, \Pi_2) = \Hom_\Iw^\cont(\Pi_1, \Pi_2)$ and $\Hom_G^\cont(\Pi_1, \Pi_2) = \Hom_N^\cont(\Pi_1, \Pi_2)$.
\end{lemma}
\begin{proof}
Let~$\lambda: \Pi_1 \to \Pi_2$ be continuous and $\Iw$-linear.
Multiplying~$\lambda$ by a power of~$\pi_E$, we can assume that~$\lambda(\Theta_1) \subseteq \Theta_2$.
Let~$g \in G^+$. It suffices to prove that $\lambda g - g \lambda|_{\Theta_1} = 0$.
To do so, since~$\Theta_2$ is separated it suffices to prove that $\lambda g - g\lambda$ induces the zero map $\Theta_1/\pi_E^n \to \Theta_2/\pi_E^n$ for all~$n > 0$.
But this is true by part~(i) of Theorem~\ref{morphisms}, since~$\Theta_i / \pi_E^n$ is flat over~$\mO_E/\pi_E^n$.

The same proof works for~$N$ and~$G$.
\end{proof}

Finally, stability of~$\Pi^1$ under~$G^+$ implies that $X_{\sigma^{[s]}} = 0$: this is a consequence of Pa{\v s}k{\=unas}'s result that~$\pi_{\sigma^{[s]}}$ is absolutely irreducible as a $G^+$-representation. 
But then Proposition~\ref{subspacefactor} implies that any proper $K$-stable closed $E$-subspace of~$\Pi^1$ or~$\Pi^2$ is finite-dimensional. 
So either the~$\Pi^i$ are topologically irreducible as $K$-representations or $\Pi$ has a finite-dimensional $K$-stable closed $E$-subspace. This concludes the proof of Theorem~\ref{restriction} in the supersingular case.

\begin{rk}\label{casesforPi}
All cases of Theorem~\ref{restriction} occur already for representations with supersingular reduction.
More precisely, case~(iii) holds if and only if a twist of~$\Pi$ is associated to a potentially semistable irreducible Galois representation $\rho: \Gal_{\bQ_p} \to \GL_2(E)$ with distinct Hodge--Tate weights under the $p$-adic Langlands correspondence, and case~(ii) holds if and only if~$\Pi$ is associated to the induction of a character of~$\Gal_{\bQ_{p^2}}$ that does not extend to~$\Gal_{\bQ_p}$.
All other representations~$\Pi$ fall into case~(i).
\end{rk}

\begin{rk}
As remarked by a referee, it follows from the results above that when~$\Pi = \Pi^1 \oplus \Pi^2$ we actually have $\End_{\Iw Z}^\cont(\Pi, \Pi) \cong E \times E$.
In fact, in this case we have~$\Pi|_N = \Ind_{\Iw Z}^N(\Pi^1) = \Ind_{\Iw Z}^N(\Pi^2)$.
This equation together with Frobenius reciprocity and Lemma~\ref{liftingtolattices} implies that 
\[
\Hom_{\Iw Z}^\cont(\Pi^i, \Pi^i) = E \text{ and } \Hom_{\Iw Z}^\cont(\Pi^1, \Pi^2) = 0,
\]
which implies that $\End_{\Iw Z}^\cont(\Pi, \Pi) \cong E \times E$.
The vanishing $\Hom_{\Iw Z}^\cont(\Pi^1, \Pi^2) = 0$ can also be proved directly.
In fact, assume that $\lambda : \Pi^1 \to \Pi^2$ is a nonzero $\Iw$-linear morphism.
Then there exist $G^+$-stable lattices~$\Theta^i \subset \Pi^i$ such that~$\lambda$ factors through a morphism $\lambda: \Theta^1 \to \Theta^2$ inducing a nonzero map $\lambda: \Theta^1/\pi_E \to \Theta^2/\pi_E$.
This contradicts Corollary~\ref{ssmorphisms} since~$\Theta_1/\pi_E \cong \pi_\sigma$ and~$\Theta_2/\pi_E \cong \pi_{\sigma^{[s]}}$.
\end{rk}

\subsection{Proof of Theorem~\ref{restriction}: reducible reduction.}\label{residuallyreducible}
Now assume that~$\Pi$ is not ordinary (as defined in~\cite{Paskunasimage}) but has no $G$-stable lattice with supersingular reduction.
Then by the main results of~\cite{Paskunasimage} the Jordan--H\"older factors of the reduction of any $G$-stable lattice are principal series representations~$\{ \pi_1, \pi_2\}$ with distinct $K$-socle. 
Assume that~$\Pi^1 \subset \Pi$ is a proper $K$-stable closed $E$-subspace. 
We are going to prove that~$\Pi^1$ is finite-dimensional over~$E$.
We introduce the following piece of notation: if~$\Theta \subset \Pi$ is a $G$-stable lattice, we write~ $\lbar \Theta = \Theta \otimes_\mO k$, and we write $\lbar \Theta^{\sub}$ for the image of the injection
\[
(\Theta \cap \Pi^1) \otimes_\mO k \to \Theta \otimes_\mO k.
\]
It is a $KZ$-stable subspace of~$\Theta \otimes_\mO k$.
\textcolor{black}{We also introduce the following definition.}

\begin{defn}\label{neighbours}
\textcolor{black}{Let~$\Theta, \Psi$ be open and bounded lattices in~$\Pi$ (not necessarily $G$-stable).
We say that~$\Theta$ and~$\Psi$ are neighbours if
\[
\pi_E \Theta \subseteq \Psi \subseteq \Theta \subseteq \pi_E^{-1}\Psi.
\]}
\end{defn}

\textcolor{black}{The reason for the name is that if~$\Theta$ and~$\Psi$ are $G$-stable neighbour lattices then the distance between $[\Theta]$ and~$[\Psi]$ in the graph constructed in the appendix is at most~$1$.
Notice that if $H$ is a subgroup of~$G$, and~$\Theta, \Psi$ are $H$-stable neighbour lattices, then the reductions~$\Theta/\pi_E\Theta$ and~$\Psi/\pi_E\Psi$ admit one-step $H$-stable filtrations with the same graded pieces (up to reordering).
In case~$\Psi = \Theta$ or~$\Psi = \pi_E\Theta$, one of these graded pieces is equal to zero.}

\subsubsection{Outline of the argument.}\label{outline}
If~$\Pi$ admits a $G$-stable lattice~$\Theta$ such that~$\lbar \Theta^\sub$ is finite-dimensional, then Lemma~\ref{completedNakayama} implies that~$\Pi^1$ is finite-dimensional. 
Otherwise, $\lbar \Theta^\sub$ is infinite-dimensional for all $G$-stable lattices~$\Theta$, and it is a proper subspace of~$\lbar \Theta$.
We will prove that the latter case leads to a contradiction, eventually arising from Proposition~\ref{morphismsps} about $\Iw$-linear morphisms between quotients of~$\pi_1$ and~$\pi_2$.
Since our argument is quite involved, we begin by giving a brief outline.

Applying the results in Appendix~\ref{Ribetappendix} we find for~$i = 1, 2$ a $G$-stable lattice~$\Theta_i$ in~$\Pi$ whose reduction is the atome automorphe surjecting onto~$\pi_i$.
Then we consider~$\lbar \Theta_i^\sub$, and we prove that it contains the other factor~$\pi_{3-i}$.
To do so, we prove that if $\lbar \Theta_i^\sub$ does not contain~$\pi_{3-i}$ then~$\lbar \Theta_i^\sub$ surjects onto~$\pi_i$: in fact, Proposition~\ref{subspacefactorps} implies that otherwise~$\lbar \Theta_i^\sub$ would be finite-dimensional, since it would intersect~$\pi_{3-i}$ in a finite-dimensional subspace, and it would have finite-dimensional image in~$\pi_i$.
But if $\lbar \Theta_i^\sub$ surjects onto~$\pi_i$ then Theorem~\ref{finitesplitting} implies that $\lbar \Theta_i^\sub = \lbar \Theta_i$ and so~$\Pi^1 = \Pi$, a contradiction.

On the other hand, by a result analogous to Proposition~\ref{subspace}, we prove in Proposition~\ref{subspaceps} that if~$\Theta$ is a $G$-stable lattice with semisimple reduction, then~$\lbar \Theta^\sub$ contains one the factors~$\pi_1, \pi_2$.
In Proposition~\ref{constantindex} we go further and we prove that this factor does not depend on the choice of~$\Theta$: up to renumbering, we can therefore assume it is~$\pi_2$.

At this point we know that $\lbar \Theta_2^\sub$ contains~$\pi_1$, and~$\Theta_2$ has a neighbour~$\Theta$ such that~$\lbar \Theta^\sub$ contains~$\pi_2$. 
\textcolor{black}{(The lattice~$\Theta$ might be~$\Theta_1$ or a $G$-stable lattice with semisimple reduction.)}
Since~$\Theta_2$ and~$\Theta$ are neighbours we have inclusions
\[
\pi_E \Theta_2 \subseteq \Theta \subseteq \Theta_2 \subseteq \pi_E^{-1}\Theta
\]
and since $\pi_E\Theta_2 \cap \Pi^1 = \pi_E(\Theta_2 \cap \Pi^1)$ and $\pi_E^{-1}\Theta \cap \Pi^1 = \pi_E^{-1}(\Theta\cap\Pi^1)$ we see that the lattices~$\Theta \cap \Pi^1, \Theta_2 \cap \Pi^1$ are also neighbours.
It follows that there are one-step filtrations on~$\lbar \Theta^\sub, \lbar \Theta_2^\sub$ with the same graded factors (up to reordering).
We conclude the argument by proving that this produces a nonzero morphism between certain quotients of~$\pi_1$ and~$\pi_2$, contradicting Proposition~\ref{morphismsps}.

\subsubsection{Subspaces of~$\Theta \otimes_\mO k$.} \label{sumsubspaces}
We prove some analogues of the results in Section~\ref{supersingularreduction}, concerning the $K$-stable subspaces of~$\Theta \otimes_\mO k$.

\begin{pp}\label{morphismsps}
Let~$\pi_1, \pi_2$ be generic principal series representations of~$\GL_2(\bQ_p)$ of distinct Serre weight, and write~$\kappa_i = (\soc_K(\pi_i)^{\Iw_1})^+$. 
Let~$X_i$ be a finite-dimensional $\Iw$-stable subspace of~$\pi_i$. 
Then there are no nonzero $\Iw$-linear morphisms $\pi_1 / X_1 \to \pi_2 / X_2$.
\end{pp}
\begin{proof}
The same argument as Proposition~\ref{Kmorphisms} goes through, substituting $\pi_\infty(\kappa_i)$ for~$M_{\sigma_i}$, $\pi_{\infty}^+(\kappa_i)$ for~$M_{\sigma_i^{[s]}}^+$, and appealing to Theorem~\ref{psIwahori1} and the proof of Proposition~\ref{pstwist}.
\end{proof}

\begin{pp}\label{subspacefactorps}
Let~$\pi$ be a generic principal series representation of~$\GL_2(\bQ_p)$. 
Assume $X$ is an infinite-dimensional $K$-stable subspace of~$\pi$. Then $X = \pi$.
\end{pp}
\begin{proof}
The same argument as Proposition~\ref{subspacefactor} proves that~$X$ has to contain one of the summands in the decomposition $\pi|_{\Iw} = \pi_\infty(\kappa) \oplus \pi_\infty^+(\kappa)$. 
(Notice that the $G$-action was not used to establish this, and there is no need here to extend scalars to~$\lbar k$.) 
If~$X$ contains $\pi_\infty(\kappa)$, then we are done since $\pi_\infty(\kappa)$ generates $\pi \cong \Ind_{K_0(p)}^K \pi_\infty(\kappa)$ over~$K$. 
Otherwise, let $\varphi_{n+1} \in \pi_{n+1}(\kappa)$ be the $K_0(p^{n+1})$-eigenvector supported in $K_0(p^{n+1})$ with $\varphi_{n+1}(1) = 1$. 
Then
\[
s\varphi_{n+1}\fourmatrix a b {pc} d = \varphi_{n+1} \fourmatrix b a d {pc} = 0
\]
implies that $s\varphi_{n+1}$ is supported in $K \setminus \Iw$, and so $s\varphi_{n+1} \in \pi_{\infty}^+(\kappa)$. 
But this implies that if $X$ is $K$-stable and contains~$\pi_\infty^+(\kappa)$, then $X = \pi$, since~$\pi_\infty(\kappa)$ is $\Iw$-generated by the~$\varphi_{n+1}$ as~$n$ tends to infinity.
\end{proof}

\begin{pp}\label{subspaceps}
\textcolor{black}{Assume that~$\Theta$ is a $G$-stable lattice in~$\Pi$ with semisimple reduction. Write $\Theta \otimes_{\mO} k$ as a direct sum of principal series representations:
\[
\pi = \Theta \otimes_{\mO} k = \pi_1 \oplus \pi_2.
\]
If $X \subset \Theta \otimes_{\mO} k$ is a proper $K$-stable infinite-dimensional subspace, then $X$ contains one of the two summands, hence it is a direct sum of subspaces of~$\pi_1$ and~$\pi_2$.}
\end{pp}
\begin{proof}
Given Propositions~\ref{morphismsps} and~\ref{subspacefactorps}, the same argument as Proposition~\ref{subspace} goes through.
\end{proof}

The following lemma will be often employed together with the previous results.

\begin{lemma}\label{furtherquotient}
Let~$\Theta \subset \Pi$ be a $G$-stable lattice with semisimple reduction, so that $\lbar \Theta^\sub \subset \Theta \otimes_\mO k$ is a proper subspace that contains~$\pi_i$ for some~$i \in \{1, 2\}$. 
Let~$\Fil \lbar \Theta^\sub$ be a $K$-stable subspace of~$\lbar \Theta^\sub$.
Then there exists a unique $\lbar \Theta^{\sub, \infty} \in \{\Fil \lbar \Theta^\sub, \lbar \Theta^\sub/\Fil \lbar \Theta^\sub\}$ that is infinite-dimensional. 
Furthermore, $\lbar \Theta^{\sub, \infty}$ has a finite-dimensional $K$-subspace~$X$ such that $\lbar \Theta^{\sub, \infty}/X$ is $K$-isomorphic to an infinite-dimensional quotient of~$\pi_i$.
\end{lemma}
\begin{proof}
By Proposition~\ref{subspaceps}, we can write $\lbar \Theta^\sub = \pi_i \oplus T$ for a finite-dimensional subspace~$T$ of the other summand of~$\Theta \otimes_\mO k$. 
Existence of~$\lbar \Theta^{\sub, \infty}$ follows because~$\lbar \Theta^\sub$ is infinite-dimensional, and uniqueness because an infinite-dimensional subspace of~$\lbar \Theta^\sub$ has the form $\pi_i \oplus T'$ for some finite-dimensional subspace of~$T$, again by Proposition~\ref{subspaceps}. 
This also implies the last assertion of the lemma in the case that~$\lbar \Theta^{\sub, \infty}  = \Fil \lbar \Theta^\sub$ (the subspace~$X$ we are looking for is~$T'$).
Otherwise, $X$ can be taken to be the finite-dimensional subspace $\left ( \Fil \lbar \Theta^\sub + T \right ) / \Fil \lbar \Theta^\sub$ of $\lbar \Theta^{\sub, \infty} = \lbar \Theta^\sub/\Fil \lbar \Theta^\sub$.  
\end{proof}

\textcolor{black}{Recall our ongoing assumption that for every $G$-stable lattice~$\Theta \subset \Pi$ the representation~$\lbar \Theta^\sub$ is infinite-dimensional (since otherwise we can conclude with an application of Lemma~\ref{completedNakayama}).
Let~$\Theta$ be a $G$-stable lattice in~$\Pi$ with semisimple reduction.}
By Proposition~\ref{subspaceps}, if~$\lbar \Theta^\sub$ is infinite-dimensional then there exists an index~$i(\Theta) \in \{1, 2\}$ such that~$\lbar \Theta^\sub$ contains~$\pi_{i(\Theta)}$.

\begin{pp}\label{constantindex}
Let~$\Theta, \Psi$ be $G$-stable lattices in~$\Pi$ with semisimple reduction. Assume that $\lbar \Theta^\sub$ and~$\lbar \Psi^\sub$ are infinite-dimensional.
Then
\[
i(\Theta) = i(\Psi).
\]
\end{pp}
\begin{proof}
Without loss of generality, assume for a contradiction that~$i(\Theta) = 1$ and~$i(\Psi) = 2$.
Applying Theorem~\ref{BanachRibet}, it suffices to prove the theorem when~$\Theta$ and~$\Psi$ are neighbours.
Then $\lbar \Theta^\sub$ and~$\lbar \Psi^\sub$ have one-step $K$-stable filtrations with the same graded factors up to reordering.

We apply Lemma~\ref{furtherquotient}.
Assume that~$\lbar \Theta^{\sub, \infty}$ is a subspace of~$\lbar \Psi^\sub$: then $\lbar \Psi^\sub/\lbar \Theta^{\sub, \infty}$ is finite-dimensional, by Lemma~\ref{furtherquotient}. 
Since $\pi_2$ has no finite-dimensional $K$-quotients (by Proposition~\ref{subspacefactorps}) we see that our assumption that $\pi_2 \subseteq \lbar \Psi^\sub$ actually implies $\pi_2 \subseteq \lbar \Theta^{\sub, \infty}$. 
But then Lemma~\ref{furtherquotient} implies a contradiction to Proposition~\ref{morphismsps}, because it allows us to construct a nonzero $K$-linear morphism from~$\pi_2$ to an infinite-dimensional quotient of~$\pi_1$.

Similarly, assume that there is a surjection $\lbar \Psi^\sub \to \lbar \Theta^{\sub, \infty}$.
By Lemma~\ref{furtherquotient} its kernel is finite-dimensional, hence the restriction of this map to~$\pi_2$ still has infinite-dimensional image. 
Then Lemma~\ref{furtherquotient} again provides a contradiction to Proposition~\ref{morphismsps}, since $\lbar \Theta^{\sub, \infty}$ surjects onto an infinite-dimensional quotient of~$\pi_1$, and the kernel of this surjection is finite-dimensional.
\end{proof}

Up to renumbering, we can therefore assume that~$\pi_2 \subset \lbar \Theta^\sub$ for all $G$-stable lattices~$\Theta \subset \Pi$ with semisimple reduction. 

\subsubsection{Splitting~$\mA_{r, s, \lambda}$.}\label{splitting}
\textcolor{black}{Recall that Theorem~\ref{BanachRibet} provides us with two $G$-stable lattices $\Theta_1, \Theta_2 \subset \Pi$ such that  $\lbar \Theta_i$ is indecomposable and surjects onto~$\pi_i$.
Hence~$\lbar \Theta_i$ is the atome automorphe surjecting onto~$\pi_i$, whereas $\pi_2 \subset \lbar \Theta_1$ and~$\pi_1 \subset \lbar \Theta_2$.
The following theorem implies that $\pi_{2} \subset \lbar \Theta_1^\sub$, and similarly $\pi_1 \subset \lbar \Theta_2^\sub$.
Indeed, $\lbar \Theta_2^\sub$ is a proper subspace of~$\lbar \Theta_2$, hence by Theorem~\ref{finitesplitting} it does not surject onto~$\pi_2$, but then we deduce that $\pi_1 \subset \lbar \Theta_2^\sub$ since otherwise $\lbar \Theta_2^\sub$ would be finite-dimensional, by Proposition~\ref{subspacefactorps}.}

\begin{thm}\label{finitesplitting}
Let $0 \to \pi_1 \to \mA_{r, s, \lambda} \to \pi_2 \to 0$ be a very generic atome automorphe, and assume that $X \subseteq \mA_{r, s, \lambda}$ is a $K$-stable subspace that surjects onto~$\pi_2$ via the given projection. Then~$X = \mA_{r, s, \lambda}$.
\end{thm}
\begin{proof}
Twisting by the determinant, we can assume without loss of generality that~$s = 0$.
We have an exact sequence
\begin{equation}\label{atomesequenceI}
0 \to \pi_{\infty}(\kappa_1) \oplus \pi_{\infty}(\kappa_1)^+ \to \mA_{r, \lambda} \to \pi_{\infty}(\kappa_2) \oplus \pi_{\infty}(\kappa_2)^+ \to 0
\end{equation}
or $\Iw$-representations. 
Recall from Section~\ref{Iwahoriatomes} that~$\mB_{r, \lambda}$ is the preimage of~$\pi_{\infty}(\kappa_2)$, and $\mD_{r, \lambda} = \mB_{r, \lambda}/\pi_{\infty}(\kappa_1)^+$. 
There is an exact sequence
\begin{equation}\label{atomesequenceII}
0 \to \pi_{\infty}(\kappa_1) \to \mD_{r, \lambda} \to \pi_{\infty}(\kappa_2) \to 0.
\end{equation}
If~$X \subset \mA_{r, \lambda}$ is a proper $K$-stable subspace that surjects onto~$\pi_2$, then~$X \cap \pi_1$ is finite-dimensional. 
If we let~$Y$ be the preimage of~$\pi_\infty(\kappa_2)$ in~$X$, then~$Y \cap \pi_1$ is finite-dimensional and the image~$Z$ of~$Y$ in~$\mD_{r, \lambda}$ surjects onto~$\pi_{\infty}(\kappa_2)$.
If~$Z = \mD_{r, \lambda}$ then~$Y + \pi_{\infty}(\kappa_1)^+ = \mB_{r, \lambda}$, and so $(Y \cap \pi_1) + \pi_{\infty}(\kappa_1)^+ = \pi_1$.
This is not true since~$Y \cap \pi_1$ is finite-dimensional and~$\pi_1/\pi_{\infty}(\kappa_1)^+ \cong \pi_{\infty}(\kappa_1)$ is infinite-dimensional.
So we have constructed a proper $\Iw$-stable subspace~$Z \subset \mD_{r, \lambda}$ that surjects onto~$\pi_{\infty}(\kappa_2)$. 
This contradicts Theorem~\ref{finitesplittingIw}, since $Z \cap \pi_\infty(\kappa_1)$ is finite-dimensional.
\end{proof}

\subsubsection{End of proof.}
Now we can conclude the proof of Theorem~\ref{restriction} in the non-ordinary case. 
Let~$\Theta$ be a $G$-stable neighbour of~$\Theta_2$ which is not homothetic to~$\Theta_2$.
\textcolor{black}{We have seen in Section~\ref{splitting} that~$\pi_1 \subset \lbar \Theta_2^\sub$ and $\pi_2 \subset \lbar \Theta_1^\sub$.
We have also fixed the numbering of~$\Theta_1$ and~$\Theta_2$ in such a way that if~$\Psi$ is a $G$-stable lattice in~$\Pi$ with semisimple reduction then~$\lbar \Psi^\sub$ contains~$\pi_2$.
Hence in all cases we know that~$\lbar \Theta^\sub$ contains~$\pi_2$.
Furthermore, by Definition~\ref{neighbours} there exist one-step filtrations on~$\lbar \Theta^\sub$ and~$\lbar \Theta_2^\sub$ with the same graded pieces up to reordering.} 
By Lemma~\ref{furtherquotient}, precisely one between $\Fil \lbar \Theta^\sub$ and $\lbar \Theta^\sub / \Fil \lbar \Theta^\sub$ has infinite dimension, and we denote it $\lbar \Theta^{\sub, \infty}$.
Assume that $\lbar \Theta^{\sub, \infty} \cong \Fil \lbar \Theta_2^\sub$. 
By Proposition~\ref{subspacefactorps}, $\pi_1$ has no nonzero finite-dimensional $K$-linear quotients, hence $\pi_1 \subset \Fil \lbar \Theta_2^\sub$. 
By Lemma~\ref{furtherquotient}, we deduce a contradiction to Proposition~\ref{morphismsps}, since there exists a surjection of~$\lbar \Theta^{\sub, \infty}$, with finite-dimensional kernel, onto an infinite-dimensional quotient of~$\pi_2$.
\textcolor{black}{Assume now that $\lbar \Theta^{\sub, \infty} \cong \lbar \Theta_2^\sub / \Fil \lbar \Theta_2^\sub$. }
Then there exists a surjection $\lbar \Theta_2^\sub \to \lbar \Theta^{\sub, \infty}$ with finite-dimensional kernel, and so the restriction
\[
\pi_1 \subset \lbar \Theta_2^\sub \to \lbar \Theta^{\sub, \infty} 
\]
has infinite-dimensional image. 
Again, Lemma~\ref{furtherquotient} provides a contradiction to Proposition~\ref{morphismsps}.
This completes the proof of Theorem~\ref{restriction} in the case that the reduction of~$\Pi$ has the same semisimplification as a reducible very generic atome automorphe.

\subsection{Proof of Theorem~\ref{restriction}: ordinary representations.} This is simpler than the previous two cases. 
By~\cite[Theorem~1.1]{Paskunasimage}, if~$\Pi$ as in the statement of Theorem~\ref{restriction} has not yet been treated, then the genericity assumption implies that the reduction of any $G$-stable lattice in~$\Pi$ is an irreducible $G$-representation isomorphic to a generic principal series representation. 
But then Proposition~\ref{subspaceps} and Lemma~\ref{completedNakayama} imply that if $\Pi^1 \subset \Pi$ is a $K$-stable closed $E$-subspace, then either $\Pi^1 = \Pi$ or $\Pi^1$ is finite-dimensional.

\subsection{Proof of Theorem~\ref{isomorphismclasses}.}
The second theorem of the introduction is as follows.
\begin{thm}\label{isomorphismrestriction}
Let $\Pi_1, \Pi_2$ be absolutely irreducible, very generic, non-ordinary unitary admissible $E$-Banach space representations of~$\GL_2(\bQ_p)$, with central character~$\zeta$.
\begin{enumerate}
\item If $\Pi_1$ and $\Pi_2$ have supersingular reduction, then $\Pi_1 |_N \cong \Pi_2 |_N$ if and only if $\Pi_1 \cong \Pi_2$, and $\Pi_1 |_{\Iw Z} \cong \Pi_2|_{\Iw Z}$ if and only if~$\Pi_1 \cong \Pi_2 \otimes (\nr_{\pm 1} \circ \det)$.
\item If $\Pi_1$ and $\Pi_2$ have reducible reduction, then $\Pi_1 |_{KZ} \cong \Pi_2 |_{KZ}$ if and only if $\Pi_1 \cong \Pi_2 \otimes (\nr_{\pm 1} \circ \det)$.
\item If~$\Pi_1$ and~$\Pi_2$ have different reduction type, then there are no $\Iw Z$-linear topological isomorphisms $\alpha: \Pi_1 \to \Pi_2$.
\end{enumerate}
\end{thm}
\begin{proof}
We begin with part~(iii).
Let $\alpha: \Pi_1 \to \Pi_2$ be an $\Iw$-linear topological isomorphism, and let~$\Theta_i \subset \Pi_i$ be open bounded $G$-stable lattices. 
Assume~$\Theta_1$ has supersingular reduction. 
Since all open bounded lattices in~$\Pi_i$ are commensurable, it is possible to multiply~$\alpha$ by a power of~$\pi_E$ so that it induces a saturated morphism $\alpha: \Theta_1 \to \Theta_2$ (by definition, this means that $\alpha \otimes_{\mO} k$ is nonzero). 
Then it suffices to prove that there are no nonzero $\Iw$-linear morphisms $\pi_\sigma \to \pi(r, \lambda, \chi)$ for~$\lambda \ne 0$, and by looking at the socle filtration as in Propositions~\ref{pstwist} and~\ref{sstwist}, it suffices to prove the following lemma. 

\begin{lemma}
There are no nonzero $\Iw$-linear morphisms $\alpha: M_\sigma^+ \to \pi_\infty(\kappa)$, for any given generic Serre weight~$\sigma$ and generic character~$\kappa$. 
\end{lemma}

\begin{proof}
\textcolor{black}{Let~$\kappa' = \soc_\Iw(\sigma)^+$. 
By Lemma~\ref{supersingularuniserial}, since~$\alpha$ is nonzero the kernel~$\ker(\alpha)$ is finite-dimensional.
Recall that we have an exhaustive filtration of~$M^+_\sigma$ by $\Iw$-subspaces $M_{\sigma, n}^+$, which by Lemma~\ref{structureMsigma} are proper quotients of representations of the form~$\pi_{2n+1}(\kappa')$, and whose dimension~$d_{\sigma, n}+1$ is given by~\eqref{dimensionMsigma}.
It follows that~$\ker(\alpha) \subset M_{\sigma, n}^+$ for all~$n$ large enough.
Composing with a surjection $\pi_{2n+1}(\kappa') \to M_{\sigma, n}^+$ we find that~$\alpha$ induces $\Iw$-linear morphisms
\[
\pi_{2n+1}(\kappa') \to M_{\sigma, n}^+ \xrightarrow{\alpha} \pi_{2n+1}(\kappa)
\]
whose images have dimension $\dim M_{\sigma, n}^+ - \dim \ker(\alpha)$. 
Now Lemma~\ref{cokernel} implies that the $p$-adic expansion of $\dim M_{\sigma, n}^+ - \dim \ker(\alpha)$ has only one nonvanishing $p$-adic digit.
Since~$\dim \ker(\alpha)$ does not depend on~$n$, this gives a contradiction to~\eqref{dimensionMsigma} as soon as $p^{2n-3} > \dim \ker(\alpha)$.}
\end{proof}

This concludes the proof of part~(iii).
For part~(i), let~$\alpha: \Pi_1|_N \isom \Pi_2 |_N$ be a topological isomorphism inducing a saturated map $\alpha: \Theta_1 \to \Theta_2$ between two open bounded $G$-stable lattices $\Theta_i \subset \Pi_i$.
By Corollary~\ref{ssmorphismsN}, the induced map $\overline{\alpha} : \Theta_1 \otimes_\mO k \to \Theta_2 \otimes_\mO k$ is a bijection, and so the reductions of the~$\Theta_i$ are isomorphic supersingular irreducible representations.
Hence the hypotheses of Lemma~\ref{liftingtolattices} hold, and so~$\alpha$ is $G$-linear.
If we only assume that~$\alpha$ is $\Iw Z$-linear, the induced map $\overline{\alpha}$ need not be a bijection, but since it is not zero we deduce that the representations $\Theta_i \otimes_\mO k$ have the same Serre weights in their $K$-socle, by Corollary~\ref{ssmorphisms}.
Hence we can still deduce that they are isomorphic.
Then Lemma~\ref{liftingtolattices} implies that~$\alpha$ is $G^+$-linear, and the claim follows from Clifford theory.
Indeed, diagonalizing the action of~$G/G^+$ on $\Hom^\cont_{G^+}(\Pi_1, \Pi_2)$ shows that every continuous $G^+$-linear map $\alpha: \Pi_1 \to \Pi_2$ can be written as the sum of a $G$-linear continuous map $\Pi_1 \to \Pi_2$ and a $G$-linear continuous map $\Pi_1 \to \Pi_2 \otimes (\nr_{-1} \circ \det)$.
Now one uses the fact that every nonzero $G$-linear continuous map between topologically irreducible admissible $E$-Banach space representations of~$G$ is an isomorphism (which follows from the fact that the category of admissible $E$-Banach space representations is abelian, hence a morphism is an isomorphism if and only if it has trivial kernel and cokernel). This concludes the proof of part~(i).

For part~(ii), we will apply the results of Appendix~\ref{Ribetappendix}. 
Let $\alpha: \Pi_1|_{KZ} \isom \Pi_2|_{KZ}$ be an isomorphism and choose a $G$-stable lattice $\Theta_1 \subset \Pi_1$ with nonsplit reduction isomorphic to~$\mA_{r, s, \lambda}$. 
Let~$\Theta_2 \subset \Pi_2$ be a $G$-stable lattice, and assume that~$\alpha$ induces a saturated map $\alpha: \Theta_1 \to \Theta_2$.

\begin{lemma}\label{subspaceatome}
The representation~$\mA_{r, s, \lambda}$ has no nonzero finite-dimensional $K$-stable quotients.
\end{lemma}
\begin{proof}
\textcolor{black}{This is because the irreducible $G$-subquotients of~$\mA_{r, s, \lambda}$ have no nonzero finite-dimensional $K$-stable quotients, by Proposition~\ref{subspacefactorps}.}
\end{proof}
\begin{corollary}\label{latticesurjection}
There exists a $K$-linear surjection from $\Theta_1 \otimes_{\mO} k$ to one of the Jordan--H\"older factors of~$\Theta_2 \otimes_{\mO} k$.
\end{corollary}
\begin{proof}
There exists an exact sequence
\[
0 \to \pi_1' \to \Theta_2 \otimes_{\mO} k \to \pi_2' \to 0
\]
where~$\pi_1', \pi_2'$ are irreducible generic principal series representation.
Compose the map~$\alpha \otimes_\mO k$ with the projection to~$\pi_2'$.
By Lemma~\ref{subspaceatome} and Proposition~\ref{subspacefactorps}, either this map is surjective or it is the zero map.
If it is zero, then the image of~$\alpha \otimes_\mO k$ is contained in~$\pi_1'$, and since~$\alpha \otimes_\mO k$ is not zero it must be a surjection onto~$\pi_1'$.
\end{proof}
By Corollary~\ref{cor:psmorphisms} and the exact sequence defining~$\mA_{r, s, \lambda}$, the existence of the surjection in Corollary~\ref{latticesurjection} implies that the irreducible $G$-constituents of $\Theta_1 \otimes_{\mO} k$ and $\Theta_2 \otimes_{\mO} k$ have the same $K$-socle. 
We deduce from Ribet's lemma that~$\Pi_2$ contains an open bounded $G$-stable lattice~$\Theta$ such that $\Theta \otimes_{\mO} k \cong \mA_{r, s, \mu}$ as $G$-representations, for some~$\mu \in k^\times$. 
Now we can scale~$\alpha$ so that it induces a saturated morphism $\Theta_1 \to \Theta$. By Theorem~\ref{sameparameter} and Proposition~\ref{atomesmorphisms}, we deduce that $\mu = \pm \lambda$ and that~$\alpha$ induces a $K$-linear isomorphism $\lbar \Theta_1 \to \lbar \Theta$ on the mod~$\pi_E$ reductions (after possibly a twist by~$\nr_{-1}$). 
By the same argument as Lemma~\ref{liftingtolattices}, the claim follows from part~(ii) of Corollary~\ref{deformations}.
\end{proof}

Finally, the following proposition together with the previous theorem implies Corollary~\ref{inertialcorrespondence}.

\begin{pp}\label{CliffordtheoryGalois}
Let~$\rho_1, \rho_2: \Gal_{\bQ_p} \to \GL_2(E)$ be absolutely irreducible continuous representations. Assume that~$\rho_1 |_{I_{\bQ_p}} \cong \rho_2 |_{I_{\bQ_p}}$ and~$\det \rho_1 = \det \rho_2$.
Then $\rho_1 \cong \rho_2 \otimes \nr_{\pm 1}$.
\end{pp}
\begin{proof}
\textcolor{black}{By~\cite[Lemma~5.1]{Paskunasimage} it suffices to prove that~$\rho_1 \cong \rho_2 \otimes \nr_{\pm 1}$ after extending scalars to a finite extension of~$E$.
Since $\Hom_{I_{\bQ_p}}(\rho_1, \rho_2)$ is a finite-dimensional $E$-vector space with a continuous action of~$\Gal_{\bQ_p}/I_{\bQ_p} \cong \wht \bZ$, replacing~$E$ by a finite extension we can assume that the topological generator of~$\wht \bZ$ has an eigenvector~$\varphi$.
The continuity condition implies that~$\varphi$ is an eigenvector for~$\wht \bZ$, and so defines a nonzero $\Gal_{\bQ_p}$-linear morphism $\varphi: \rho_1 \to \rho_2 \otimes \nr_\lambda$ for some~$\lambda \in E^\times$.
Now $\varphi$ is an isomorphism because~$\rho_1$ and~$\rho_2 \otimes \nr_\lambda$ are irreducible, and since~$\det(\rho_1) = \det(\rho_2)$ we have~$\lambda = \pm 1$.}
\end{proof}

\begin{acknowledgements}
The problem of relating the actions of~$I_{\bQ_p}$ and parahoric subgroups has been considered by Caraiani--Emerton--Gee--Geraghty--Pa{\v s}k{\=u}nas--Shin and (independently) Gabriel Dospinescu.
Their methods were different, making use of Colmez's functor, and were not brought to completion (for instance, the role of the Iwahori subgroup seems to be unexpected).
I learned about this problem from Toby Gee, and I am grateful to him as well as Matthew Emerton for helpful conversations on these and related subjects.
I thank Vytautas Pa{\v s}k{\=u}nas for explaining me the exact sequence~(\ref{tree}), and Stefano Morra for sharing some of his unpublished notes on~\cite{Morraatomes}.
\textcolor{black}{Finally, I am particularly grateful to the anonymous referees for a close reading that helped me to catch a mistake in a previous version of this paper, as well as to improve the quality of the exposition.}
The author was supported at various stages of this work by the Engineering and Physical Sciences Research Council [EP/L015234/1], The EPSRC Centre for Doctoral Training in Geometry and Number Theory (The London School of Geometry and Number Theory), University College London, and Imperial College London; by the James D. Wolfensohn Fund at the Institute for Advanced Study; and by a Royal Society University Research Fellowship.
\end{acknowledgements}

\appendix

\section{Ribet's lemma for Banach spaces.}\label{Ribetappendix}
Ribet's lemma for~$\Gal_{\bQ_p}$ is the statement that if~$\rho: \Gal_{\bQ_p} \to \GL_2(E)$ is an irreducible continuous representation whose reduction has two distinct Jordan--H\"older factors then the $\rho$-stable homothety classes of lattices form a bounded segment of length at least two in the Bruhat--Tits tree of~$E^{\oplus 2}$.
Furthermore, the lattices~$\rho_1^\circ, \rho_2^\circ$ corresponding to the extremal points in the segment have indecomposable reductions with nonisomorphic socle, and all the other lattices have semisimple reduction.

Now let~$\Pi$ be an absolutely irreducible, admissible, unitary $E$-Banach space representation of~$G = \GL_2(\bQ_p)$ with a $G$-stable open and bounded lattice~$\Theta$ such that $\lbar \Theta = \Theta/\pi_E\Theta$ is a reducible representation with two non-isomorphic Jordan--H\"older factors~$\{\pi_1, \pi_2\}$.
In this appendix we prove an analogue of Ribet's lemma for~$\Pi$.
One can do this using Colmez's functor if~$\Pi$ is generic in the sense of Section~\ref{defngeneric}: it suffices to invoke~\cite[Remarque~III.10(iii), Proposition~III.54]{CDcompletes}.
We will provide a different proof in order to make the result independent of the $p$-adic Langlands correspondence for~$\GL_2(\bQ_p)$.
In fact, we will deal with the more general case that~$\Pi$ is an $E$-Banach space with a topologically irreducible $E$-linear action of a group~$G$ that stabilizes an open and bounded lattice~$\Theta$, and such that~$\lbar \Theta$ is a $k[G]$-representation of length two with distinct Jordan--H\"older factors.

Our approach follows Serre's proof of Ribet's lemma as closely as possible.
An obstruction to do this is the absence of a Bruhat--Tits building for the infinite-dimensional $E$-vector space $\Pi$, and so we begin by providing a substitute: we do this by adapting the arguments in~\cite{Serrearbres}.
Unless otherwise stated, in this appendix we abbreviate ``open and bounded lattice" to~``lattice".
Since any two open and bounded lattices in~$\Pi$ are commensurable, every $G$-stable lattice in~$\Pi$ has reduction of length two over~$k[G]$, with Jordan--H\"older factors~$\{\pi_1, \pi_2\}$.
We will say that an inclusion $\Theta_1 \subset \Theta_2$ of lattices is \emph{saturated} if~$\Theta_1 \not \subset \pi_E\Theta_2$.

\begin{defn}
Define a graph $\Gamma$ with set of vertices given by homothety classes of $G$-stable lattices in~$\Pi$, such that $\{[\Theta_0], [\Theta_1]\}$ is an edge if and only if there are representatives of these homothety classes such that $\Theta_1 \subset \Theta_0$ and~$\Theta_0/\Theta_1$ is an irreducible $k[G]$-representation.
\end{defn}

\begin{rk}
By the word ``graph" we mean a one-dimensional simplicial complex. These correspond to the \emph{graphes combinatoires} in~\cite{Serrearbres}.
\end{rk}

\begin{defn}
If~$\Theta_1 \subset \Theta_0$ is a saturated inclusion between $G$-stable lattices, the $G$-representation $\Theta_0/\Theta_1$ has finite length. 
We define the distance
\[
d(\Theta_0, \Theta_1) = \operatorname{length}_G(\Theta_0/\Theta_1).
\]
This only depends on the homothety class of~$\Theta_0, \Theta_1$.
\end{defn}

In order to prove that the distance is a symmetric function one can use the following lemma on saturated inclusions.

\begin{lemma}\label{annihilators}
Assume that $\Theta_n \subset \Theta_0$ is a saturated inclusion of $G$-stable lattices in~$\Pi$ and that $\length_G(\Theta_0/\Theta_n) = n$.
Then $\Ann_\mO(\Theta_0/\Theta_n) = \pi_E^n\mO$.
\end{lemma}
\begin{proof}
Since $\pi_E$ annihilates all Jordan--H\"older factors of~$\Theta_0/\Theta_n$ we deduce immediately that $\pi_E^n(\Theta_0/\Theta_n) = 0$.
For the other direction, we need to prove that $\pi_E^{n-1}\Theta_0 \not \subset \Theta_n$.
We use induction on~$n$ and we start by pulling back a Jordan--H\"older series for $\Theta_0/\Theta_n$ to a sequence of open and bounded lattices
\[
\Theta_n \subset \Theta_{n-1} \subset \cdots \subset \Theta_0.
\]
Since $\Theta_n \subset \Theta_0$ is saturated, the homothety classes of the~$\Theta_i$ are pairwise distinct.
In addition, we know that~$\Theta_{n-1}$ contains both~$\Theta_n$ and~$\pi_E\Theta_{n-2}$, and
\[
\length_G(\Theta_{n-1}/\Theta_n) = \length_G(\Theta_{n-1}/\pi_E\Theta_{n-2}) = 1.
\]
Putting these together, and using the fact that $\Theta_{n-1}/\pi_E\Theta_{n-1}$ has at most two proper nonzero $G$-stable subspaces, we find that $\Theta_n \cap \pi_E\Theta_{n-2} = \pi_E\Theta_{n-1}$.
Our inductive assumption says that $\pi_E^{n-2}\Theta_0 \not \subset \Theta_{n-1}$, and we know that $\pi_E^{n-2}\Theta_{0} \subset \Theta_{n-2}$.
Multiplying by~$\pi_E$, we deduce that $\pi_E^{n-1}\Theta_0 \not \subset \Theta_n$, which concludes the proof.
\end{proof}

\begin{corollary}
Let~$[\Theta], [\Theta']$ be vertices of~$\Gamma$.
Then $d([\Theta], [\Theta']) = d([\Theta'], [\Theta])$.
\end{corollary}
\begin{proof}
Choose a saturated inclusion $\Theta' \subset \Theta$ between representatives of these homothety classes, and let~$m = d([\Theta], [\Theta'])$.
Then we have a chain of lattices
\[
\pi_E^m\Theta \subset \Theta' \subset \Theta
\]
and by Lemma~\ref{annihilators} the inclusion~$\pi_E^m\Theta \subset \Theta'$ is saturated.
By additivity of length we find that
\[
d([\Theta'], [\Theta]) + m = \length_G(\Theta/\pi_E^m\Theta) = 2m
\]
which yields the claim.
\end{proof}

\begin{corollary}\label{uniseriallattice}
If~$\Theta' \subset \Theta$ is a saturated inclusion of $G$-stable lattices in~$\Pi$, then the $G$-representation~$\Theta/\Theta'$ is uniserial.
\end{corollary}
\begin{proof}
Let~$n = \length_G(\Theta/\Theta')$.
If~$\Theta/\Theta'$ admits a $G$-stable filtration with~$m$ graded pieces, all of which are semisimple, then~$\pi_E^m(\Theta/\Theta') = 0$.
Hence Lemma~\ref{annihilators} implies that both the socle and the cosocle filtration have length equal to~$n$, so they have simple graded pieces, and so by Lemma~\ref{soclefiltration} the representation is uniserial.
\end{proof}
The following is our version of Ribet's lemma for~$\Pi$.

\begin{thm}\label{BanachRibet}
The graph~$\Gamma$ is a finite line segment of length at least two, corresponding to a chain of pairwise non-homothetic $G$-stable lattices
\[
\Theta_n \subset \cdots \subset \Theta_0.
\]
The $G$-representations $\Theta_0/\pi_E\Theta_0$ and~$\Theta_n/\pi_E\Theta_n$ are indecomposable and not isomorphic.
The $G$-representations $\Theta_i/\pi_E\Theta_i$ for $0 < i < n$ are semisimple. 
\end{thm}
\begin{proof}
We give the proof by means of several lemmas.
\begin{lemma}\label{Gammatree}
The graph~$\Gamma$ is a tree, i.e. a connected and simply connected graph.
\end{lemma}
\begin{proof}
This follows by the same argument as~\cite[\S1, Chapitre~II]{Serrearbres}.
Namely, given two vertices of~$\Gamma$ we can find representatives~$\Theta_0, \Theta_1$ and a saturated inclusion $\Theta_1 \to \Theta_0$.
Pulling back a Jordan--H\"older sequence of the finite length $G$-representation $\Theta_0/\Theta_1$ then yields a path in~$\Gamma$ between~$[\Theta_0]$ and~$[\Theta_1]$, proving that~$\Gamma$ is connected.
To see that it is simply connected, it suffices to check that for all~$n \geq 1$ and every path $[\Theta_0], [\Theta_1], \ldots, [\Theta_n]$ without backtracking (i.e. such that $[\Theta_i] \ne [\Theta_{i+2}]$ for all~$i$) we have~$[\Theta_0] \ne [\Theta_n]$.
Such a path can be represented by a sequence of lattices
\[
\Theta_n \subset \Theta_{n-1} \ldots \subset \Theta_0
\]   
where $\length_G(\Theta_i/\Theta_{i+1}) = 1$ for all~$i$.
It suffices to prove by induction on~$n$ that the inclusion $\Theta_n \subset \Theta_0$ is saturated, i.e. that $\Theta_n \not \subseteq \pi_E\Theta_0$.
When~$n = 2$, the fact that $\length_G(\Theta_0/\Theta_2) = 2$ implies that if the inclusion is not saturated then~$\Theta_2 = \pi_E\Theta_0$, which we are assuming not to be the case.
Assuming the statement true for~$n-1$, notice that~$\Theta_n$ and~$\pi_E\Theta_{n-2}$ are distinct by the assumption of no backtracking, and contain~$\pi_E\Theta_{n-1}$, and so they identify with the only two $G$-stable subspaces of~$\Theta_{n-1}/\pi_E\Theta_{n-1}$: this implies that
\[
\Theta_{n-1} = \Theta_n + \pi_E\Theta_{n-2}.
\] 
Now we see that $\Theta_n \subset \pi_E\Theta_0$ implies~$\Theta_{n-1} \subset \pi_E\Theta_0$, contradicting the inductive assumption.
This concludes the proof that~$\Gamma$ is a tree.
\end{proof}

Given a vertex~$\Theta_0$ of~$\Gamma$, we know that~$\Theta_0/\pi_E\Theta_0$ is a $G$-representation of length two, and so~$\Theta_0$ has at most two adjacent vertices in~$\Gamma$: this implies that~$\Gamma$ is a line segment (possibly infinite or half-infinite).
To prove that~$\Gamma$ is finite, it suffices therefore to prove that there is no infinite sequence
\[
\ldots \subset \Theta_n \subset \cdots \subset \Theta_1 \subset \Theta_0
\]
of pairwise non-homothetic $G$-stable lattices in~$\Theta_0$ such that $\length_G(\Theta_i/\Theta_{i+1}) = 1$ for all~$i$.
To do so, define
\[
\wht \Theta = \varprojlim_n \Theta_0/\Theta_n.
\]
\begin{lemma}\label{completelattice}
The natural map $\Theta_0 \to \wht \Theta$ is surjective, and~$\wht \Theta$ is $\pi_E$-adically separated and complete and $\pi_E$-torsion free.
\end{lemma}
\begin{proof}
By the proof of Lemma~\ref{Gammatree} each of the inclusions $\Theta_n \subset \Theta_0$ is saturated, and so by Corollary~\ref{uniseriallattice} the quotients ~$\Theta_0/\Theta_n$ are uniserial $G$-representations.
Let~$m, j$ be nonnegative integers and consider the map
\[
\pi_E^m : \Theta_0/\Theta_{m+j} \to \Theta_0/\Theta_{m+j}. 
\]
We claim that its image is $\Theta_m/\Theta_{m+j}$, or equivalently that
\[
\pi_E^m \Theta_0 + \Theta_{m+j} = \Theta_m.
\]
We know that the $G$-representation~$\Theta_m/\Theta_{m+j}$ is uniserial, and its radical is~$\Theta_{m+1}/\Theta_{m+j}$.
By Lemma~\ref{annihilators}, we know that $\pi_E^m \Theta_0 \not \subset \Theta_{m+1}$.
Hence the natural map $\pi_E^m\Theta_0 \to \Theta_m/\Theta_{m+j}$ is surjective, and the claim follows.
Now, since~$\Theta_0/\Theta_n$ is a finite length $G$-representation, we have exact sequences
\[
0 \to \varprojlim_{j>0} \Theta_j/\Theta_{m+j} \to \wht \Theta \xrightarrow{\pi_E^m} \wht \Theta \to \varprojlim_{j>0}\Theta_0/\Theta_m \to 0.
\]
This proves that~$\wht \Theta$ is $\pi_E$-adically separated and complete and~$\pi_E$-torsion free, and that $\wht \Theta \otimes_\mO k \cong \Theta_0/\Theta_1$.
By Lemma~\ref{completedNakayama}, it follows that the natural map~$\Theta_0 \to \wht \Theta$ is surjective.
\end{proof}
Let~$\Theta^1 = \bigcap_{i \geq 0} \Theta_i$. 
By Lemma~\ref{completelattice}, we have an exact sequence
\[
0 \to \Theta^1 \to \Theta_0 \to \wht \Theta \to 0.
\]
Since $\Theta_0 \to \wht \Theta$ is not an isomorphism (as can be seen by applying $- \otimes_{\mO} k$) we see that~$\Theta^1$ is a nonzero proper closed $G$-stable $\mO$-submodule of~$\Pi$.
The following lemma then provides a contradiction to the assumed topological irreducibility of~$\Pi$, and it follows that~$\Gamma$ is a finite line segment.

\begin{lemma}
The $E$-vector space~$\Pi^1 = \Theta^1[1/p]$ is a nonzero proper closed $G$-stable subspace of~$\Pi$.
\end{lemma}
\begin{proof} 
We know that~$\Theta^1$ is nowhere dense in the sense that its set-theoretical complement $\Pi \setminus \Theta^1$ is open and dense: indeed, if $\Theta^1$ contained a set open in~$\Pi$, then after a translation it would contain~$\pi_E^m \Theta_0$ for some positive integer value of~$m$, but then $\pi_E^m\Theta_0 \subset \Theta_j$ for all~$j \geq 0$, contradicting Lemma~\ref{annihilators} as soon as~$j> m$.
Hence the Baire category theorem implies that $\Pi^1 = \Theta^1[1/p]$ is a proper subspace of~$\Pi$.
To check that it is closed, let~$x_n \in \Pi^1$ be a sequence in~$\Pi^1$ converging to~$x \in \Pi$.
To prove that~$x \in \Pi^1$ we can assume without loss of generality that~$x \in \Theta_0$, multiplying by a suitable power of~$\pi_E$.
Passing to a subsequence, we can find elements~$\theta_n \in \Theta_0$ such that
\[
x - x_n = \pi_E^n \theta_n.
\]
Fix~$n > 0$.
If~$m$ is large enough that~$\pi_E^m x_n \in \Theta^1$, we find that $\pi_E^m x$ and  $\pi_E^{n+m} \theta_n$ have the same image in~$\wht \Theta$.
Since~$\wht \Theta$ is $\pi_E^m$-torsion free by Lemma~\ref{completelattice}, this implies that the image of~$x$ is contained in~$\pi_E^n \wht \Theta$.
Since this holds for all~$n > 0$ and~$\wht \Theta$ is $\pi_E$-adically separated by Lemma~\ref{completelattice}, we see that~$x \in \Theta^1$, which concludes the proof.
\end{proof}

Finally, it is part of our assumptions that~$\Gamma$ is not empty, and that if~$[\Theta]$ is a vertex then $\Theta/\pi_E\Theta$ is a reducible $G$-representation.
Pulling back a $G$-stable proper nonzero subspace we obtain another vertex of~$\Gamma$, which has therefore length at least two.
Now the following lemma concludes the proof of Theorem~\ref{BanachRibet}.
\end{proof}

\begin{lemma}
With the notation in the statement of Theorem~\ref{BanachRibet}, the lattices~$\Theta_0, \Theta_n$ have nonisomorphic reductions with different cosocle, and the lattices~$\Theta_i$ for~$1 < i < n$ have semisimple reduction.
\end{lemma}
\begin{proof}
The statement about the reduction type is an immediate consequence of the number of neighbours of~$[\Theta_i]$ in~$\Gamma$.
There remains to prove the statement about cosocles. 
To do so, we can assume without loss of generality that~$\cosoc_G(\Theta_0/\pi_E\Theta_0) \cong \pi_1$.
Then it suffices to prove that~$\Theta_i/\Theta_{i+1} \cong \pi_1$ for all~$0 \leq i \leq n-1$.
Indeed, this implies that
\[
\cosoc_G(\Theta_n/\pi_E\Theta_n) = \Theta_n/\pi_E\Theta_{n-1} \cong \pi_2,
\]
which was to be proved.
We use induction on~$i$, the base case being contained in our assumption,
By the inductive assumption we know that $\pi_E\Theta_{i-1}/\pi_E\Theta_i \cong \pi_1$.
Since~$\Theta_{i+1}$ and~$\Theta_{i-1}$ are not homothetic, the image of~$\Theta_{i+1}$ in $\Theta_i/\pi_E\Theta_{i}$ is not~$\pi_E\Theta_{i-1}/\pi_E\Theta_i$ but the other $G$-stable subspace, and so~$\Theta_{i+1}/\pi_E\Theta_{i} \cong \pi_2$.
This implies that~$\Theta_i/\Theta_{i+1} \cong \pi_1$, which was to be proved.
\end{proof}

\bibliographystyle{amsalpha}
\bibliography{refpapers}

\end{document}